\documentclass[a4paper,12pt,reqno]{amsart}
\usepackage{fullpage}
\usepackage[dvipsnames]{xcolor}
\usepackage{amsmath,amssymb}
\usepackage{bm} % preferred version of bold symbols
\usepackage[breaklinks,bookmarks=false]{hyperref}
\hypersetup{colorlinks, linkcolor=blue, citecolor=blue,
urlcolor=blue, plainpages=false, pdfwindowui=false,
pdfstartview={FitH}, pdftitle={}}
\usepackage{graphicx}
\usepackage{psfrag}
\usepackage{enumerate}
\usepackage{tikz}
\usepackage[normalem]{ulem} % for strikethrough (\sout) in collaborative writing
\usepackage{cancel} % for strikethrough in collaborative writing (math formulas)
\usepackage[textsize=tiny,color=blue!50!white]{todonotes}
\newtheorem{theorem}{Theorem}[section]
\newtheorem{definition}[theorem]{Definition}
\newtheorem{remark}[theorem]{Remark}

\newtheorem{proposition}[theorem]{Proposition}
\newtheorem{lemma}[theorem]{Lemma}
\newtheorem{assumption}[theorem]{Assumption}

\def\N{\mathbb{N}}
\def\KK{\mathcal{K}}
\def\TT{\mathcal{T}}
\def\F{\bm{F}}
\def\EE{\mathcal{F}}
\def\ee{F}
\def\p{\bm{p}}
\def\uni#1{{\rm uni}(#1)}

\def\reff#1#2{\stackrel{\eqref{#1}}{#2}}
\def\refff#1#2#3{\stackrel{\substack{\eqref{#1}\\\eqref{#2}}}{#3}}
\newcommand{\norm}[3][]{#1\Vert #2#1\Vert_{#3}}
\def\bPhi{\bm{\Phi}}
\def\H{\bm{H}}
\def\V{\bm{V}}
\def\RT{\bm{RT}}
\def\bv{\bm{v}}
\def\bx{\bm{x}}
\def\w{\bm{w}}
\def\n{\bm{n}}
\def\R{\mathbb{R}}
\def\dual#1#2{(#1\,,\,#2)}
\def\coarse{h}

\def\II{\mathcal{I}}
\def\VV{\mathcal{V}}
\def\osc{\mathrm{osc}}
\def\supp{\mathrm{supp}}
\def\div{\mathrm{div}}
\def\diam{\mathrm{diam}}
\def\level{\mathrm{level}}
\def\set#1#2{\{#1\,:\,#2\}}

\def\argmin#1{\underset{#1}{\operatorname{argmin}}}

\def\ssigma{\boldsymbol{\sigma}}
\def\ttau{\boldsymbol{\tau}}
\def\ttau{\boldsymbol{\tau}}

\newcommand\ie{i.e.}

\newcommand\eg{e.g.}

\newcommand\eq{:=}

\newcommand\ls{\lesssim}
\newcommand\Dv{\nabla {\cdot}} % divergence
\newcommand\ver{{\bm a}}
\newcommand\vertt{{\bm b}}
\newcommand\nv{\bm 0}
\newcommand\QQ{\mathbb{Q}}
\newcommand\oma{\omega_\ver}
\newcommand\Ta{\TT_\ver}
\newcommand\psia{\psi_\ver}

\begin{document}

% title
\author{Gregor Gantner}
\author{Martin Vohral\'ik}
\title{Inexpensive polynomial-degree-robust\\
equilibrated flux a posteriori estimates \\
for isogeometric analysis}
\thanks{This project has received funding from the European Research Council (ERC) under the European Union's Horizon 2020 research and innovation program (grant agreement No 647134 GATIPOR).
In addition, the first author has been supported by the Austrian Science Fund (FWF) under grant P29096 and J4379-N}
\def\fpage{\pageref{CorrectFirstPageLabel}}

\address{
Institute of Analysis and Scientific Computing, TU Wien, Wiedner Hauptstra{\ss}e 8-10, 1040 Vienna, Austria (\href{mailto:gregor.gantner@asc.tuwien.ac.at}{\texttt{gregor.gantner@asc.tuwien.ac.at}}). Current address: Institute for Numerical Simulation, University of Bonn, Friedrich-Hirzebruch-Allee 7, 53115 Bonn, Germany}

\address{Inria, 2 rue Simone Iff, 75589 Paris, France \& CERMICS, Ecole des Ponts, 77455 Marne-la-Vall\'ee, France
(\href{mailto:martin.vohralik@inria.fr}{\texttt{martin.vohralik@inria.fr}}).}

\begin{abstract}
We consider isogeometric discretizations of the Poisson model problem, focusing on high polynomial degrees and strong hierarchical refinements. We derive a posteriori error estimates by equilibrated fluxes, \ie, vector-valued mapped piecewise polynomials lying in the $\H(\div)$ space which appropriately approximate the desired divergence constraint. Our estimates are constant-free in the leading term, locally efficient, and robust with respect to the polynomial degree. They are also robust with respect to the number of hanging nodes arising in adaptive mesh refinement employing hierarchical B-splines.
Two partitions of unity are designed, one with larger supports corresponding to the mapped splines, and one with small supports corresponding to mapped piecewise multilinear finite element hat basis functions. The equilibration is only performed on the small supports, avoiding the higher computational price of equilibration on the large supports or even the solution of a global system. Thus, the derived estimates are also as inexpensive as possible.
An abstract framework for such a setting is developed, whose application to a specific situation only requests a verification of a few clearly identified assumptions.
Numerical experiments illustrate the theoretical developments.
\end{abstract}

\maketitle

\tableofcontents

%%%%%%%%%%%%%%%%%%%%%%%%%%%%%%%%%%%%%%%%%%%%%%%%%%%%%%%%%%%%%%%%%%%%%%%%%%%%%%%%%%%%%
%%%%%%%%%%%%%%%%%%%%%%%%%%%%%%%%%%%%%%%%%%%%%%%%%%%%%%%%%%%%%%%%%%%%%%%%%%%%%%%%%%%%%
\section{Introduction}
%%%%%%%%%%%%%%%%%%%%%%%%%%%%%%%%%%%%%%%%%%%%%%%%%%%%%%%%%%%%%%%%%%%%%%%%%%%%%%%%%%%%%
%%%%%%%%%%%%%%%%%%%%%%%%%%%%%%%%%%%%%%%%%%%%%%%%%%%%%%%%%%%%%%%%%%%%%%%%%%%%%%%%%%%%%
%==A posteriori error estimates==
A posteriori error estimates for standard finite element methods (FEM) are nowadays well understood~\cite{Ainsw_Oden_a_post_FE_00,Repin_book_08,Verf_13}.
On the one hand, they allow to assess the quality of the computed approximation, and, on the other hand, they indicate where to refine the underlying mesh of the computational domain.
%==equilibrated fluxes==
Among the existing error estimators, those based on equilibrated fluxes~\cite{Prag_Syng_47, Lad_Leg_83, Dest_Met_expl_err_CFE_99, Braess_Scho_a_post_edge_08} have the advantage that they provide a \emph{guaranteed} upper bound for the approximation error which is constant-free in the leading term.
%==polynomial-degree-robust estimates==
The estimator from~\cite{Dest_Met_expl_err_CFE_99, Braess_Scho_a_post_edge_08} even turns out to be robust with respect to the polynomial degree~\cite{Brae_Pill_Sch_p_rob_09}, i.e., it provides also a lower bound with an efficiency constant that does \emph{not} depend on the polynomial degree.
This result has been recently generalized in several directions; see \cite{ev15, ev20} and the references therein.
We also mention \cite{Dol_Ern_Voh_hp_16}, which covers standard FEM on triangular/rectangular meshes with hanging nodes, and its generalization to arbitrarily many hanging nodes~\cite{Ern_Smears_Voh_H-1_lift_17}.

%==isogeometric analysis==
In this work, we aim to generalize this result to isogeometric analysis (IGA) \cite{hcb05,bbchs06,chb09}.
The central idea of IGA is to use the same ansatz functions for the representation of the problem geometry in computer-aided design (CAD) and for the discretization of the partial differential equation (PDE).
While the CAD standard for spline representation in a multivariate setting relies on tensor-product B-splines, several extensions of the B-spline model have emerged  that allow for adaptive refinement, e.g., hierarchical splines~\cite{vgjs11,gjs12}, (analysis-suitable) T-splines~\cite{szbn03,slsh12,bbsv13},  or LR-splines~\cite{dlp13,jkd14}; see also \cite{jrk15,hkmp17,bggpv22} for a  comparison of these approaches.

%%%%%%%%%%%%%%%%%%%%%%%%%%%%%%%%%%%%%%%%%%%%%%%%%%%%%%%%%%%%%%%%%%%%%%%%%%%%%%%%%%%%%
\subsection{Available results}
%%%%%%%%%%%%%%%%%%%%%%%%%%%%%%%%%%%%%%%%%%%%%%%%%%%%%%%%%%%%%%%%%%%%%%%%%%%%%%%%%%%%%
To steer an adaptive refinement, rigorous a posteriori error estimators have been developed.
Assuming a certain admissibility condition of the employed meshes, the works~\cite{bg16a,gp20} generalize the weighted residual error estimator from standard FEM to IGA with hierarchical splines and analysis-suitable T-splines, respectively.
It has even been proved that the corresponding adaptive algorithms converge at optimal algebraic rate with respect to the number of mesh elements~\cite{bg17,ghp17,gp20}, see also the recent overview article~\cite{bggpv22}.
However, the reliability and efficiency constants depend on the employed polynomial degree, which is indeed witnessed in numerical experiments, see, e.g., \cite{bggpv22} for hierarchical splines.
Similarly, for the hierarchical spline space from \cite{bg16b}, the work~\cite{Buf_Gar_a_post_IGA_18} derives a weighted residual estimator being the sum of indicators associated to basis functions instead of elements. It is shown to be reliable for \emph{arbitrary} hierarchical meshes, with an unknown reliability constant that again particularly depends on the used polynomial degree $p$.

In the spirit of \cite{repin99,repin00a,repin00b}, the works~\cite{kt15,Klei_Tom_a_post_IGA_15,Matc_a_post_IGA_18} present guaranteed fully computable upper bounds of the approximation error for tensor-product splines and hierarchical splines, respectively. A second estimate is also derived, giving a lower bound of the error.
However, to compute these so-called functional-type estimators, a \emph{global} minimization problem may need to be solved or an $\H(\div)$ flux not in the equilibrium with the load
may be employed, and the efficiency of the upper bound, as well as the reliability of the lower bound, are theoretically unclear.
In the recent work~\cite{Thai_Cham_Ha-Mi_a_post_IGA_19}, a non-$\H(\div)$ approximation of an equilibrated flux is constructed for tensor-product splines in extension of the concepts of~\cite{Lad_Leg_83}. It only requires to solve locally, on the knot spans, a pair of a low-order problem for the normal fluxes together with a high-order problem for the equilibrated flux approximation.
A generalization to hierarchical splines is briefly sketched and a corresponding numerical example is provided. While this yields a fully computable and approximately guaranteed upper bound on the error, efficiency of the resulting estimator is unclear.

%%%%%%%%%%%%%%%%%%%%%%%%%%%%%%%%%%%%%%%%%%%%%%%%%%%%%%%%%%%%%%%%%%%%%%%%%%%%%%%%%%%%%
\subsection{Outline}
%%%%%%%%%%%%%%%%%%%%%%%%%%%%%%%%%%%%%%%%%%%%%%%%%%%%%%%%%%%%%%%%%%%%%%%%%%%%%%%%%%%%%
The present manuscript is organized as follows.
We introduce the model problem and its general Galerkin discretization in Section~\ref{sec:setting}. We then motivate and introduce our new approach in the simplified spline setting in Section~\ref{sec:motivation}.
Sections~\ref{sec_splines} and~\ref{sec:partitions} subsequently contain a detailed description of the discretization space formed by hierarchical B-splines. This technical description might be skipped in a first reading, since all our results are only based on a couple of abstract assumptions collected in Section~\ref{sec:assumptions}, where we also verify them in the hierarchical B-splines setting.
Section~\ref{sec:equilibration} describes in detail our inexpensive flux equilibration in the general context. The a posteriori error estimates are then derived in Section~\ref{sec:a_post}, where we prove their reliability and (local) efficiency.
Numerical experiments illustrating the theoretical findings are collected in Section~\ref{sec:numerics}. Finally, Appendix~\ref{sec:extension} proves the broken polynomial extension property that is central for the polynomial-degree robustness in the present setting.

%%%%%%%%%%%%%%%%%%%%%%%%%%%%%%%%%%%%%%%%%%%%%%%%%%%%%%%%%%%%%%%%%%%%%%%%%%%%%%%%%%%%%
%%%%%%%%%%%%%%%%%%%%%%%%%%%%%%%%%%%%%%%%%%%%%%%%%%%%%%%%%%%%%%%%%%%%%%%%%%%%%%%%%%%%%
\section{Model problem and its Galerkin discretization} \label{sec:setting}
%%%%%%%%%%%%%%%%%%%%%%%%%%%%%%%%%%%%%%%%%%%%%%%%%%%%%%%%%%%%%%%%%%%%%%%%%%%%%%%%%%%%%
%%%%%%%%%%%%%%%%%%%%%%%%%%%%%%%%%%%%%%%%%%%%%%%%%%%%%%%%%%%%%%%%%%%%%%%%%%%%%%%%%%%%%

Let $\Omega$ be an open bounded connected Lipschitz domain in $\R^d$ with $d =2,3$.
We consider the Poisson model problem  with homogeneous Dirichlet data of finding $u: \Omega \rightarrow \mathbb{R}$ such that
\begin{align}\label{eq:Poisson}
\begin{split}
-\Delta u&=f \quad\text{in }\Omega,\\
u&=0\quad\text{on }\partial\Omega,
\end{split}
\end{align}
where $f\in L^2(\Omega)$ is a given source term.
On an arbitrary subset $\omega\subseteq\overline\Omega$, let $\dual{\cdot}{\cdot}_\omega\eq\int_\omega (\cdot) (\cdot)\,{\rm d}x$ and $\norm{\cdot}{\omega}$ denote the $L^2$-scalar product and the corresponding norm, respectively; we also denote by $\norm{\cdot}{\infty,\omega}$ the $L^\infty$-norm.
Then, the weak formulation of problem~\eqref{eq:Poisson} consists in finding $u\in H_0^1(\Omega)$ such that
\begin{align}\label{eq:variational}
\dual{ \nabla u}{\nabla v}_\Omega=\dual{f}{v}_\Omega\quad\text{for all } v\in H_0^1(\Omega).
\end{align}
Let $V_\coarse$ be a finite-dimensional subspace of $H_0^1(\Omega)$.
The corresponding Galerkin approximation is to find $u_\coarse\in V_\coarse$ with
\begin{align}\label{eq:Galerkin}
\dual{ \nabla u_\coarse}{\nabla v_\coarse}_\Omega=\dual{f}{v_\coarse}_\Omega\quad\text{for all } v_\coarse\in V_\coarse.
\end{align}

In this work, we will choose $V_\coarse$ as the space of mapped \emph{hierarchical splines}.
To this end, we assume that $\Omega$ can be parametrized over $\widehat\Omega\eq(0,1)^d$ via a bi-Lipschitz mapping $\F:\widehat\Omega\to\Omega$ with positive Jacobian determinant, i.e., $\Omega$ = $\F(\widehat\Omega)$, $\F\in W^{1,\infty}({\widehat\Omega})$, $\F^{-1}\in W^{1,\infty}(\Omega)$, and $\det(D\F)>0$.
Note that $\F$ and its inverse $\F^{-1}$ can both be continuously extended to Lipschitz continuous functions on $\overline{\widehat\Omega}$ and $\overline\Omega$, respectively.

%%%%%%%%%%%%%%%%%%%%%%%%%%%%%%%%%%%%%%%%%%%%%%%%%%%%%%%%%%%%%%%%%%%%%%%%%%%%%%%%%%%%%
%%%%%%%%%%%%%%%%%%%%%%%%%%%%%%%%%%%%%%%%%%%%%%%%%%%%%%%%%%%%%%%%%%%%%%%%%%%%%%%%%%%%%
\section{Motivation and the new approach in a simplified spline setting} \label{sec:motivation}
%%%%%%%%%%%%%%%%%%%%%%%%%%%%%%%%%%%%%%%%%%%%%%%%%%%%%%%%%%%%%%%%%%%%%%%%%%%%%%%%%%%%%
%%%%%%%%%%%%%%%%%%%%%%%%%%%%%%%%%%%%%%%%%%%%%%%%%%%%%%%%%%%%%%%%%%%%%%%%%%%%%%%%%%%%%

In this section, we recall equilibration for standard $C^0$ finite elements, discuss difficulties for a generalization to IGA, and introduce our new approach.
In order to explain the main ideas as clearly as possible, we will consider a strongly simplified setting.
In particular, we suppose in this section that $\F$ is just the identity, so that $\Omega = \widehat \Omega = (0,1)^d$, and that $\TT_\coarse$ is a uniform tensor-product mesh of $\Omega$ consisting of elements $K$ being all $d$-rectangular parallelepipeds.
Let $\mathbb{Q}^{\bm{p}}(\TT_\coarse)$ denote all $\TT_\coarse$-piecewise tensor-product polynomials of fixed degree $\bm{p}\eq(p,\dots,p)$ in each component, with $p\geq1$.
We assume additionally that the right-hand side $f$ in~\eqref{eq:Poisson} is a $\TT_\coarse$-piecewise tensor-product polynomial of degree $\bm{p}-\bm{1}:= \bm{p} - (1,\dots,1)$, i.e.,
\begin{align*}
	f\in \mathbb{Q}^{\bm{p}-\bm{1}}(\TT_\coarse).
\end{align*}

%%%%%%%%%%%%%%%%%%%%%%%%%%%%%%%%%%%%%%%%%%%%%%%%%%%%%%%%%%%%%%%%%%%%%%%%%%%%%%%%%%%%%
\subsection{Standard equilibration in finite elements} \label{sec:standard_equi}
%%%%%%%%%%%%%%%%%%%%%%%%%%%%%%%%%%%%%%%%%%%%%%%%%%%%%%%%%%%%%%%%%%%%%%%%%%%%%%%%%%%%%
In finite elements, the ansatz space $V_\coarse$ from~\eqref{eq:Galerkin} is the subspace of $\mathbb{Q}^{\bm{p}}(\TT_\coarse)$ formed by functions which are continuous over the mesh interfaces and zero on the boundary of $\Omega$, i.e.,
\begin{align}\label{eq:standard_elements}
	V_\coarse := \set{v_\coarse \in\mathbb{Q}^{\bm{p}}(\TT_\coarse) \cap C^0(\Omega)}{v_\coarse|_{\partial\Omega} = 0} = \mathbb{Q}^{\bm{p}}(\TT_\coarse) \cap H_0^1(\Omega).
\end{align}

Let $\VV_\coarse$ collect the vertices of the mesh $\TT_\coarse$.
For each $\ver\in\VV_\coarse$, we define the associated continuous piecewise bilinear hat function
\begin{align}\label{eq_psia}
	\psi_\ver\in \mathbb{Q}^{\bm{1}}(\TT_\coarse) \cap C^0(\Omega)
\end{align}
taking value $1$ in the vertex $\ver$ and $0$ in all other vertices from $\VV_\coarse$.
We denote the interior of its support by
\begin{align}\label{eq_supp}
	\omega_\ver := {\rm int}(\supp(\psi_\ver)).
\end{align}
Note that it corresponds to the union of all (closed) elements in $\TT_\coarse$ containing $\ver$, which we denote by
\begin{align*}
	\TT_\ver := \set{T\in\TT_\coarse}{T \subseteq \supp(\psi_\ver)}.
\end{align*}
Crucially, the $\psi_\ver$ form a partition of unity in that
\begin{align*}
	\sum_{\ver \in \VV_\coarse} \psi_\ver = 1 \quad \text{in }\Omega.
\end{align*}
The hat functions $\psi_\ver$ and their supports $\supp(\psi_\ver)$ are illustrated for $d=2$ in Figure~\ref{fig:bsplines135}, left.

For each vertex $\ver \in \VV_\coarse$, we further introduce the space
\begin{align*}%\label{eq:H0 div for a}
\H_0(\div,\omega_\ver)\eq
\begin{cases}
\set{\bv\in \H(\div,\omega_\ver)}{\bv{\cdot}\n_{\omega_\ver}=0 \text{ on }\partial\omega_\ver}&\hspace{-2mm}\text{if }\psi_\ver\in H_0^1(\Omega),\\
\set{\bv\in \H(\div,\omega_\ver)}{\bv{\cdot}\n_{\omega_\ver}=0\text{ on }\partial\omega_\ver\cap(\psi_\ver)^{-1}(\{0\})}&\hspace{-2mm}\text{else},\\
\end{cases}
\end{align*}
where $\n_{\omega_\ver}$ denotes the outer normal vector on $\partial\omega_\ver$ and $\bv{\cdot}\n_{\omega_\ver}$ is understood in the appropriate weak sense.
With the usual broken (elementwise) Raviart--Thomas space of order $\widetilde\p:=(\widetilde p,\dots,\widetilde p):=\p+\bm{1}$
\begin{align} \label{eq:RT1}
	\RT^{\widetilde\p}(\TT_\ver) :=
	\begin{cases}
	\mathbb{Q}^{\widetilde\p+(1,0)}(\TT_\ver)\times \mathbb{Q}^{\widetilde\p+(0,1)}(\TT_\ver)\quad&\text{ if } d=2,
	 \\
	 \mathbb{Q}^{\widetilde\p+(1,0,0)}(\TT_\ver)\times \mathbb{Q}^{\widetilde\p+(0,1,0)}(\TT_\ver)\times  \mathbb{Q}^{\widetilde\p+(0,0,1)}(\TT_\ver)&\text{ if } d=3,
	 \end{cases}
\end{align}
we also define the discrete subspace
\begin{align}\label{eq_Vha_FE}
	\V_\coarse^{\ver}\eq \RT^{\widetilde\p}(\TT_\ver) \cap \H_0(\div,\omega_\ver).
\end{align}
Then, the local equilibrated fluxes
\begin{align}\label{eq:sigma_a}
	\ssigma_\coarse^{\ver}\eq\argmin{\substack{\bv_\coarse\in\V_\coarse^{\ver}\\
	\Dv\bv_\coarse = f\psi_\ver-\nabla u_\coarse{\cdot}\nabla \psi_\ver}}
	\norm{\bv_\coarse+\psi_\ver \nabla u_\coarse}{\omega_\ver}
\end{align}
give rise to a global equilibrated flux
\begin{align*}
	\ssigma_\coarse := \sum_{\ver \in \VV_\coarse} \ssigma_\coarse^\ver \quad \text{ with } \Dv\ssigma_\coarse = f.
\end{align*}

Following~\cite{Brae_Pill_Sch_p_rob_09, ev15, ev20}, which rely on the tools from~\cite{Cost_Daug_Demk_ext_08, Cost_McInt_Bog_Poinc_10, Demk_Gop_Sch_ext_III_12}, one can show that the corresponding a posteriori estimator $\norm{\ssigma_\coarse + \nabla u_\coarse}{\Omega}$ constitutes a guaranteed upper bound as well as a $p$-robust lower bound for the discretization error, i.e.,
\begin{align}\label{eq:equivalence}
	\norm{\nabla(u_\coarse - u)}{\Omega} \le \norm{\ssigma_\coarse + \nabla u_\coarse}{\Omega} \le C_{\rm eff} \norm{\nabla(u_\coarse - u)}{\Omega},
\end{align}
with a constant $C_{\rm eff}>0$ depending only on the shape-regularity of the mesh $\TT_\coarse$ but \emph{not} on the polynomial degree $p$.
The efficiency bound holds even locally on all mesh elements.
This result has been extended to general rectangular meshes with an arbitrary number of hanging nodes in~\cite{Dol_Ern_Voh_hp_16, Ern_Smears_Voh_H-1_lift_17}, see also Remark~\ref{rem2:naive_approach} below.

%%%%%%%%%%%%%%%%%%%%%%%%%%%%%%%%%%%%%%%%%%%%%%%%%%%%%%%%%%%%%%%%%%%%%%%%%%%%%%%%%%%%%
\subsection{A straightforward (expensive) generalization to IGA}\label{sec:standard_iga}
%%%%%%%%%%%%%%%%%%%%%%%%%%%%%%%%%%%%%%%%%%%%%%%%%%%%%%%%%%%%%%%%%%%%%%%%%%%%%%%%%%%%%
Compared to standard finite elements~\eqref{eq:standard_elements}, the IGA/spline space is the subspace of $\mathbb{Q}^{\bm{p}}(\TT_\coarse)$ formed by functions which are continuous over the mesh interfaces {\em including their derivatives up to order $p-m$}, i.e.,
\begin{align}\label{eq:splines}
	V_\coarse := \set{v_\coarse\in\mathbb{Q}^{\bm{p}}(\TT_\coarse) \cap C^{p-m}(\Omega)}{ v_\coarse|_{\partial \Omega}=0}.
\end{align}
Here $m$ takes values between $1$ and $p$. Examples for $m=1$, leading to the maximal smoothness $p-1$, and $p=1,3$ are given in Figure~\ref{fig:bsplines135}; note that for $m=p$, \eqref{eq:splines} and~\eqref{eq:standard_elements} coincide. A support of $\psia$ for the maximal smoothness $p-1$ and $p=5$ is then shown in the left part of Figure~\ref{fig:1st_idea}.
Alternatively, in place of~\eqref{eq:splines}, one can write $V_\coarse = \mathbb{S}^{\p}(\KK_\coarse) \cap H_0^1(\Omega)$, where $\mathbb{S}^{\p}(\KK_\coarse)$ is the spline space with the smoothness $p-m$ encoded in the global knot vector $\KK_\coarse$, see Section~\ref{sec:multivariate splines} below.

Whenever the requested smoothness $p-m$ is greater than $0$, the finite element hat functions $\psi_\ver$ from~\eqref{eq_psia} do not belong to the IGA space $V_\coarse$ given by~\eqref{eq:splines}, being merely continuous. A straightforward adaptation of the flux reconstruction of the previous section would be to employ instead appropriately scaled $\psi_\ver$ from the spline space $\mathbb{Q}^{\bm{p}}(\TT_\coarse) \cap C^{p-m}(\Omega)=\mathbb{S}^{\p}(\KK_\coarse)$, still leading to the partition of unity
\begin{align*}
	\sum_{\ver \in \VV_\coarse} \psi_\ver = 1 \quad \text{in }\Omega,
\end{align*}
where now $\VV_\coarse$ is a suitable node set rather than the set of vertices of the mesh $\TT_\coarse$.
Standard B-splines, see Section~\ref{sec:multivariate splines}, constitute such a partition of unity.
The major inconvenience of such a construction is that the supports $\omega_\ver$ of $\psi_\ver$ (cf.~\eqref{eq_supp}), which appear in the minimization~\eqref{eq:sigma_a}, are much larger for splines. In particular, for the maximal smoothness $p-1$, the support of the corresponding B-splines of degree $p$ consists of $(p+1)^d$ elements, see Figure~\ref{fig:1st_idea}. This would make the dimension of the space $\V_\coarse^{\ver}$ from~\eqref{eq_Vha_FE} and consequently the size of the linear system in~\eqref{eq:sigma_a} grow steeply with $p$, making the equilibration very expensive, see Figure~\ref{fig:1st_idea} for illustration. In addition, in the construction~\eqref{eq:sigma_a},
one needs $\psi_\ver \nabla u_\coarse \in \RT^{\widetilde\p}(\TT_\ver)$ and $f\psi_\ver-\nabla u_\coarse{\cdot}\nabla \psi_\ver \in \mathbb{Q}^{\widetilde\p}(\TT_\ver)$ to avoid all oscillation terms. In the finite element context of Section~\ref{sec:standard_equi}, this is satisfied for the choice $\widetilde\p =\p+\bm{1}$, but it would request $\widetilde\p =2\p$ here. Moreover, it is not clear whether~\eqref{eq:equivalence} would still hold with a constant $C_{\rm eff}$ independent of the polynomial degree $p$ of the considered B-splines.

\begin{figure}[t]
\begin{center}
\includegraphics[width=0.51\textwidth]{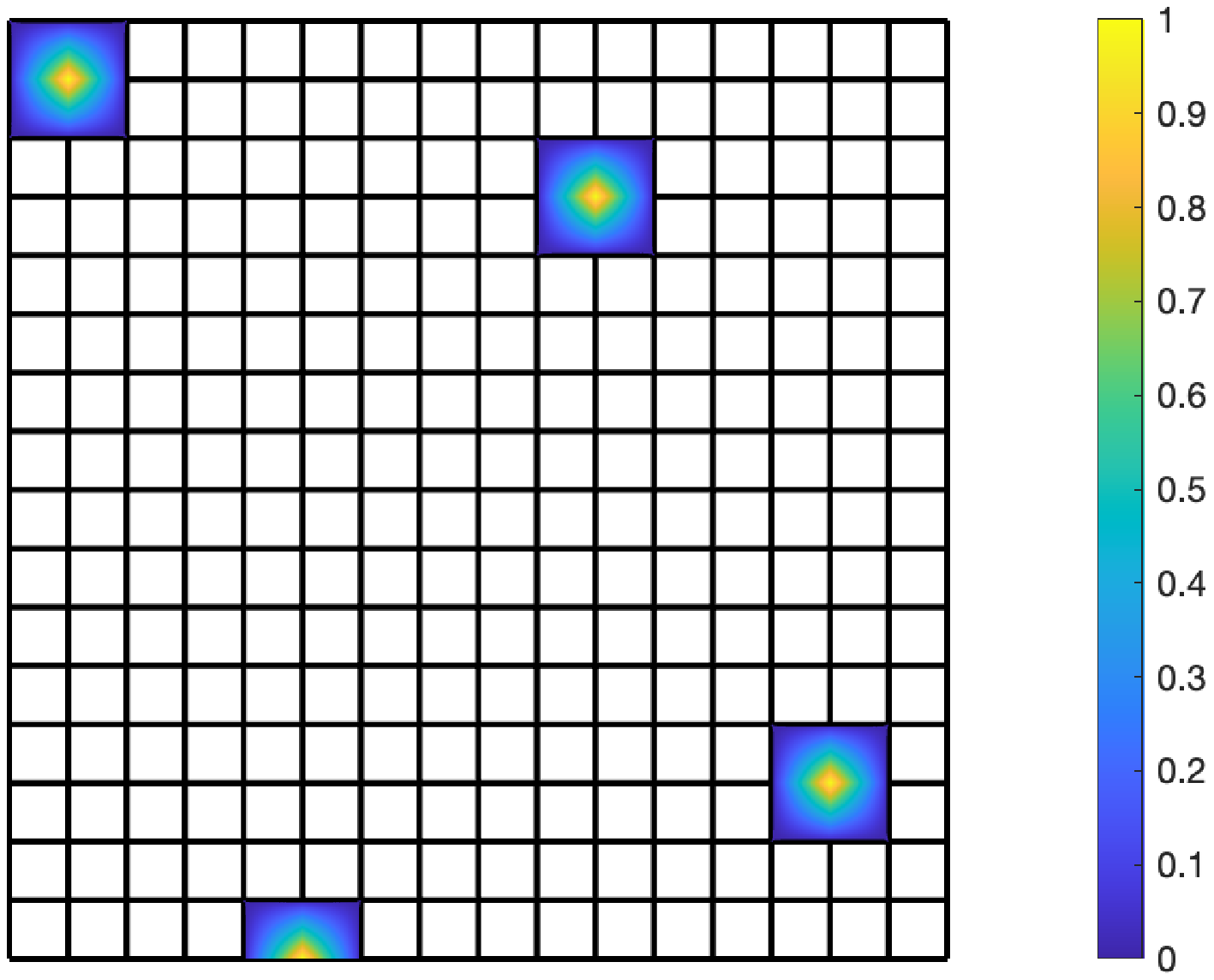}
\hspace{-7.5mm}
\includegraphics[width=0.51\textwidth]{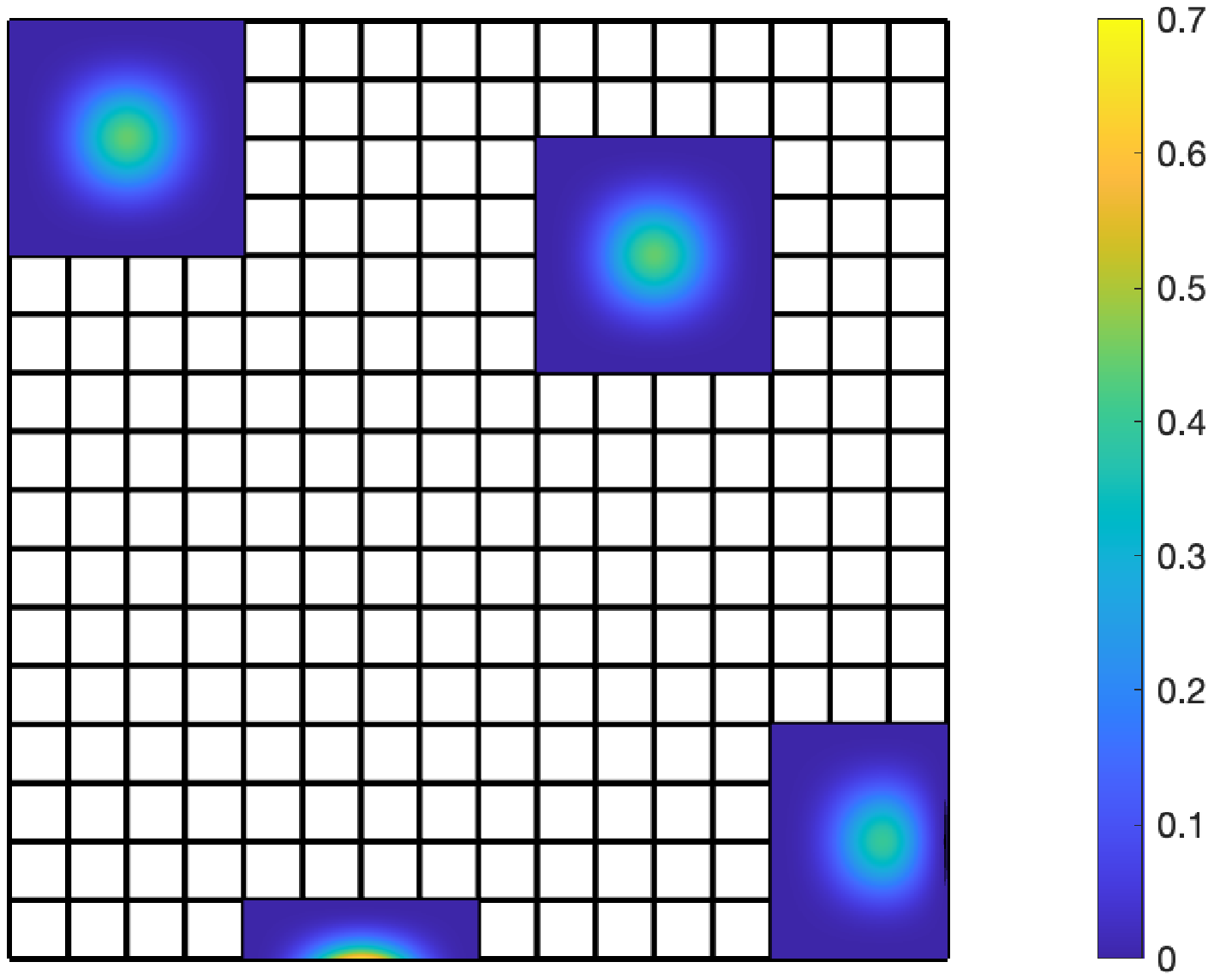}

\vspace{-8mm}
$p=1$ \hspace{70mm} $p=3$

\end{center}
\caption{\label{fig:bsplines135}
B-splines of degree $p=1,3$ and maximal smoothness $p-1$ on a two-dimensional uniform tensor mesh $\TT_\coarse$.}
\end{figure}

\begin{figure}[t]
\begin{center}
\includegraphics[width=0.7\textwidth]{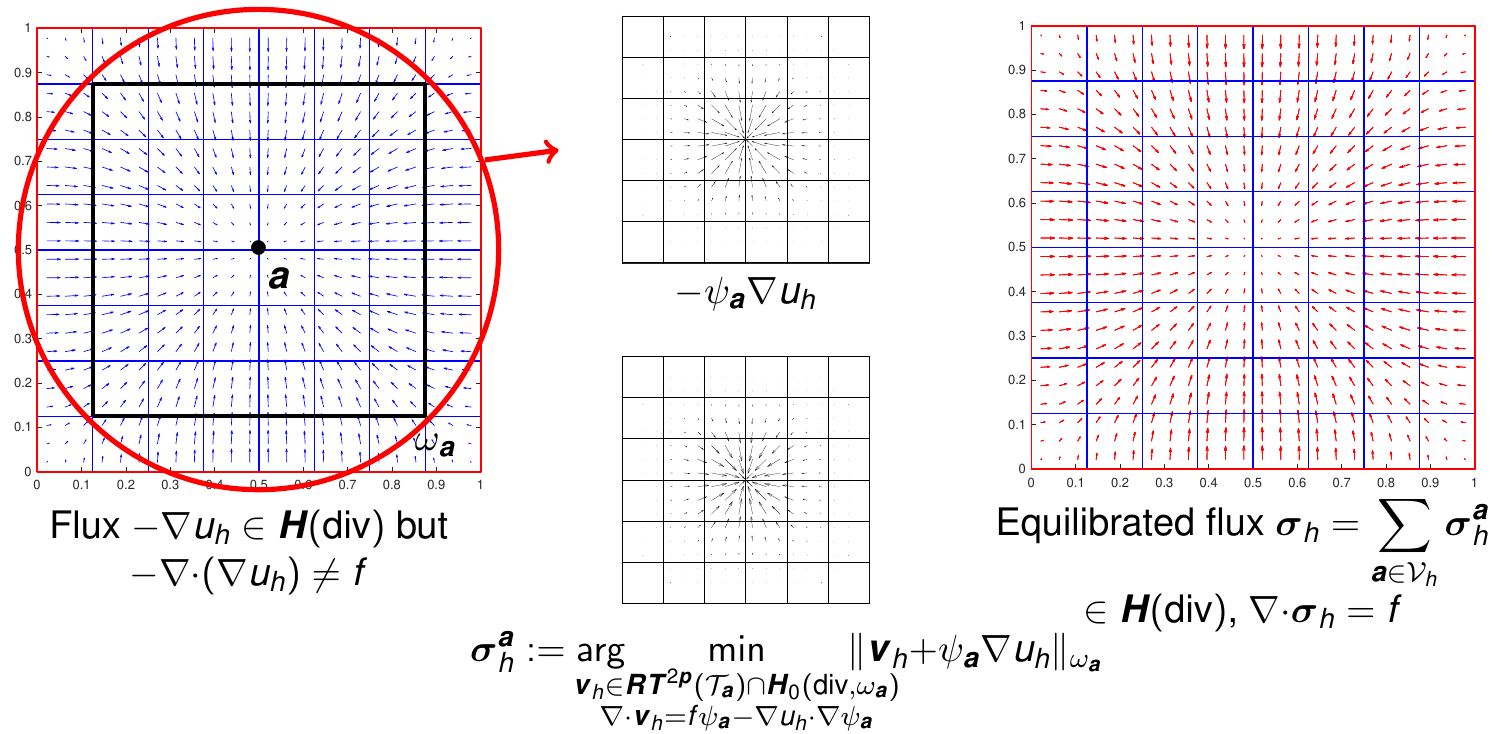}
\end{center}
\caption{A straightforward (expensive) extension of the FEM equilibration to IGA with B-splines of degree $p=5$ and maximal smoothness $p-1$ on a two-dimensional uniform tensor mesh $\TT_\coarse$\label{fig:1st_idea}.}
\end{figure}

%%%%%%%%%%%%%%%%%%%%%%%%%%%%%%%%%%%%%%%%%%%%%%%%%%%%%%%%%%%%%%%%%%%%%%%%%%%%%%%%%%%%%
\subsection{Inexpensive equilibration in IGA}\label{sec:inex_iga}
%%%%%%%%%%%%%%%%%%%%%%%%%%%%%%%%%%%%%%%%%%%%%%%%%%%%%%%%%%%%%%%%%%%%%%%%%%%%%%%%%%%%%

Our main idea is to replace the straightforward approach of Section~\ref{sec:standard_iga} by a construction that does not request the expensive solution of~\eqref{eq:sigma_a} employing the large spaces $\V_\coarse^{\ver}$ of~\eqref{eq_Vha_FE} defined over the large spline supports $\omega_\ver$. We design a two-stage procedure where: first, lowest-order $\mathbb{Q}^{\bm{1}}(\TT_\ver) \cap C^0(\oma)$ scalar problems are solved on each large patch $\TT_\ver$; second,  systems of the form~\eqref{eq:sigma_a}, but only over the finite element {\em hat} basis functions supports $\omega_\vertt$ with $2^d$ mesh elements, are solved. This is schematically viewed in Figure~\ref{fig:procedure}.
Similar $p$-robust constructions were designed in~\cite{Sch_Mel_Pech_Zagl_p_prec_08} in the context of an additive Schwarz smoother/preconditiner (just one global lowest-order solve and then $p$-order local patch problems) or in~\cite{Ern_Smears_Voh_H-1_lift_17, Ern_Sme_Voh_heat_HO_Y_17, Pap_Voh_MG_expl_22} for a posteriori error estimates (allowing for fast equilibration with $H^{-1}(\Omega)$ right-hand sides/arbitrary coarsening in parabolic time stepping/inexpensive estimates of the algebraic error).

\begin{figure}[t]
\begin{center}
\includegraphics[width=0.7\textwidth]{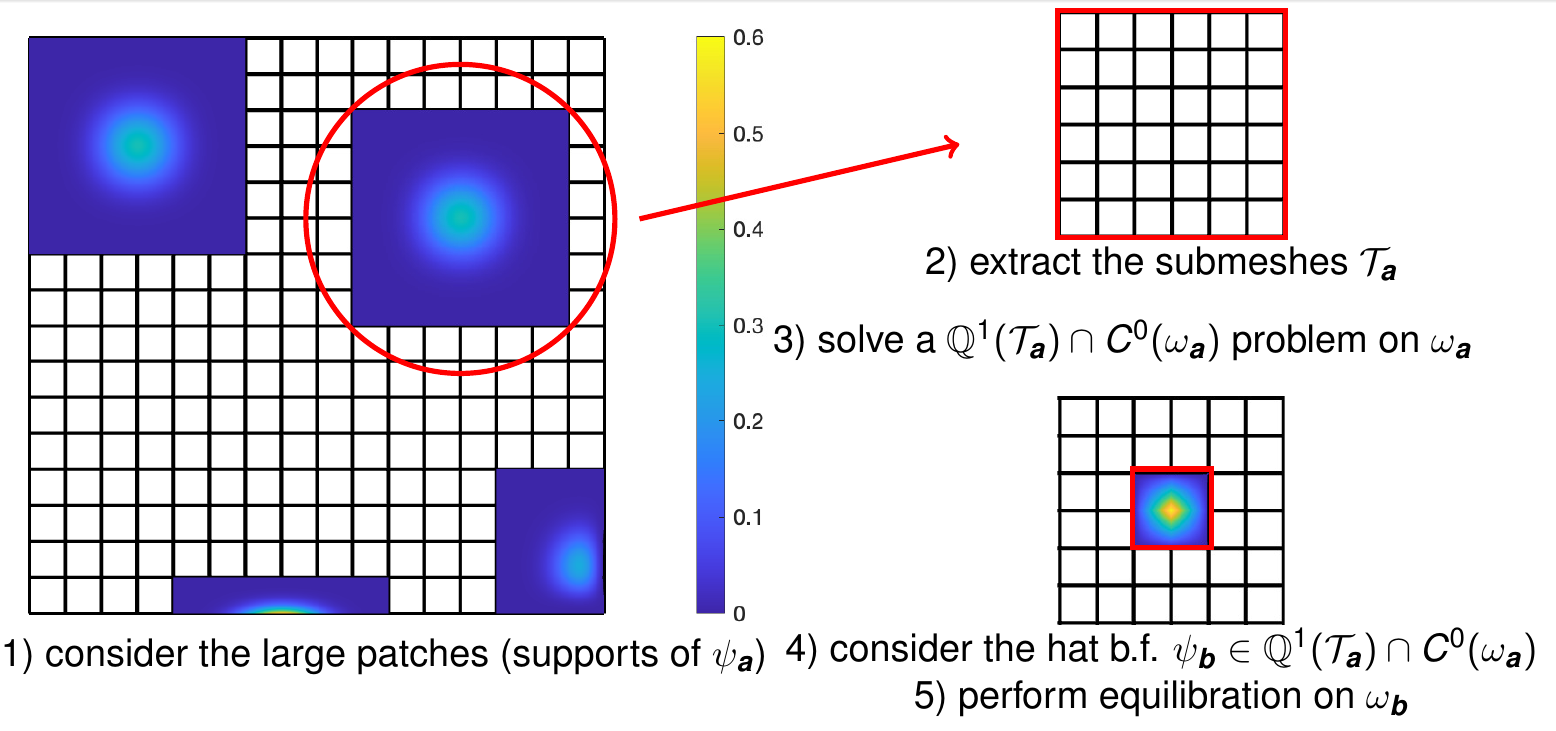}
\end{center}
\caption{Inexpensive equilibration in IGA in the simplified setting of uniform tensor-product mesh, $p=5$ and maximal smoothness $p-1$\label{fig:procedure}.}
\end{figure}

More precisely, we first, on each (large) patch subdomain $\omega_\ver$ with local mesh $\TT_\ver$ solve the primal lowest-order problem: find
\begin{align*}
	r_\coarse^\ver\in V_\coarse^\ver \eq
\begin{cases}
\set{v_\coarse\in \QQ^1(\Ta) \cap C^0(\oma)  }{\dual{v_\coarse}{1}_{\omega_\ver}=0}\quad&\text{if }\psi_\ver\in H_0^1(\Omega),\\
\set{v_\coarse\in \QQ^1(\Ta) \cap C^0(\oma) }{v_\coarse=0\text{ on }
\partial\omega_\ver\setminus\psi_\ver^{-1}(\{0\})}
\quad&\text{else}
\end{cases}
\end{align*}
such that
\begin{align}\label{eq:rh beta_I}
\dual{\nabla r_\coarse^\ver}{\nabla v_\coarse}_{\omega_\ver}
=\dual{f}{v_\coarse \psi_\ver}_{\omega_\ver}- \dual{\nabla u_\coarse}{\nabla(v_\coarse \psi_\ver)}_{\omega_\ver}
\quad\text{for all }v_\coarse\in V_\coarse^\ver,
\end{align}
see Figure~\ref{fig:procedure}, steps 1)--3).
This yields the scalar-valued lifting $r_\coarse^\ver$ of the $\psia$-weighted residual of~\eqref{eq:Galerkin} with respect to~\eqref{eq:variational}. Second, we consider the finite element hat basis functions $\psi_\vertt$ of vertices $\vertt$ of the local mesh $\TT_\ver$ and solve the dual high-order problems
\begin{align}\label{eq:sigma beta b_I}
\ssigma_\coarse^{\ver,\vertt}\eq\argmin{\substack{\bv_\coarse\in\V_\coarse^{\ver,\vertt}\\
\Dv\bv_\coarse=f\psi_\ver \psi_\vertt-\nabla u_\coarse{\cdot}\nabla(\psi_\ver \psi_\vertt) -\nabla r_\coarse^\ver{\cdot}\nabla \psi_\vertt}}
\norm{\bv_\coarse+\psi_\vertt(\psi_\ver \nabla u_\coarse+\nabla r_\coarse^\ver)}{\omega_\vertt},
\end{align}
where $\V_\coarse^{\ver,\vertt}$ is defined similarly to~\eqref{eq_Vha_FE}, with
\begin{align*}
\H_0(\div,\omega_\vertt)\eq
\begin{cases}
\set{\bv\in \H(\div,\omega_\vertt)}{\bv{\cdot}\n_{\omega_\vertt}=0 \text{ on }\partial\omega_\vertt}&\hspace{-2mm}\text{if }\psi_\ver\psi_\vertt\in H_0^1(\Omega),\\
\set{\bv\in \H(\div,\omega_\vertt)}{\bv{\cdot}\n_{\omega_\vertt}=0\text{ on } \partial\omega_\vertt\\ \hspace{5cm} \cap(\psi_\ver\psi_\vertt)^{-1}(\{0\})}&\hspace{-2mm}\text{else},\\
\end{cases}
\end{align*}
 see Figure~\ref{fig:procedure}, steps 4)--5).
As in Section~\ref{sec:standard_iga}, we need to increase the equilibration polynomial degree $\widetilde\p$; actually, owing to the presence of the finite element hat basis function $\psi_\vertt$, we will need $\widetilde\p =2\p + 1$ and not just $\widetilde\p =2\p$ to satisfy
$\psi_\vertt(\psi_\ver \nabla u_\coarse+\nabla r_\coarse^\ver) \in \RT^{\widetilde\p}(\TT_\ver)$ and $f\psi_\ver \psi_\vertt-\nabla u_\coarse{\cdot}\nabla(\psi_\ver \psi_\vertt) -\nabla r_\coarse^\ver{\cdot}\nabla \psi_\vertt \in \mathbb{Q}^{\widetilde\p}(\TT_\ver)$ to avoid all oscillation terms. All these local problems can be solved independently one from each other, allowing for efficient parallelization.
Moreover, for every two same patch geometries, one matrix assembly and factorization is sufficient.

To finish, we sum the contributions $\ssigma_\coarse^{\ver,\vertt}$ from~\eqref{eq:sigma beta b_I} as
\begin{align}\label{eq_flux_a_I}
 \ssigma_\coarse^\ver\eq\sum_{\vertt\in\VV_\coarse^\ver} \ssigma_\coarse^{\ver,\vertt}
\end{align}
and form the final equilibrated flux as
\begin{align}\label{eq_flux_tot_I}
\ssigma_\coarse\eq\sum_{\ver\in\VV_\coarse}\ssigma_\coarse^\ver,
\end{align}
which yields a fully computable guaranteed (constant-free in the leading term) upper bound on the unknown error similarly to~\eqref{eq:equivalence}, see Proposition~\ref{prop:reliable} for details. This bound is also locally and globally efficient, see respectively
Propositions~\ref{prop:efficient} and~\ref{prop:efficiency}. The involved constant $C_{\rm eff}$ is robust with respect to the polynomial degree $p$ of the used hierarchical splines and the number of hanging nodes of the underlying hierarchical meshes, but theoretically not with respect to the smoothness of the splines and the number of patches $\omega_\ver$ overlapping in a point (see~\eqref{eq:overlap}) (in our numerical experiments, though, $C_{\rm eff}$ shows robustness with respect to the smoothness and the number of overlapping patches).
Full details of the equilibration are given in Definitions~\ref{def_ra}, \ref{def_flux_ab}, and~\ref{def_flux_a} below; Figure~\ref{fig:partition} illustrates the procedure on hierarchical meshes.

%%%%%%%%%%%%%%%%%%%%%%%%%%%%%%%%%%%%%%%%%%%%%%%%%%%%%%%%%%%%%%%%%%%%%%%%%%%%%%%%%%%%%
%%%%%%%%%%%%%%%%%%%%%%%%%%%%%%%%%%%%%%%%%%%%%%%%%%%%%%%%%%%%%%%%%%%%%%%%%%%%%%%%%%%%%
\section{Hierarchical splines} \label{sec_splines}
%%%%%%%%%%%%%%%%%%%%%%%%%%%%%%%%%%%%%%%%%%%%%%%%%%%%%%%%%%%%%%%%%%%%%%%%%%%%%%%%%%%%%
%%%%%%%%%%%%%%%%%%%%%%%%%%%%%%%%%%%%%%%%%%%%%%%%%%%%%%%%%%%%%%%%%%%%%%%%%%%%%%%%%%%%%

In this section, we first describe the piecewise polynomial space of multivariate splines in the parameter domain $\widehat\Omega$. We then introduce its hierarchical extension, covering highly graded local mesh refinement. Finally, the space $V_\coarse$ from~\eqref{eq:Galerkin} will be given by the transformation of the latter one by the mapping $\F$.

%%%%%%%%%%%%%%%%%%%%%%%%%%%%%%%%%%%%%%%%%%%%%%%%%%%%%%%%%%%%%%%%%%%%%%%%%%%%%%%%%%%%%
\subsection{Multivariate splines in the parameter domain $(0,1)^d$}
\label{sec:multivariate splines}
%%%%%%%%%%%%%%%%%%%%%%%%%%%%%%%%%%%%%%%%%%%%%%%%%%%%%%%%%%%%%%%%%%%%%%%%%%%%%%%%%%%%%
We recall here the standard definition of multivariate splines in the parameter domain $(0,1)^d$; for a detailed introduction, we refer, e.g., to~\cite{db86,db01,sch07}.

Let the integer $p\geq 1$ be a fixed positive  polynomial degree and let
\begin{align*}
{\KK}_\coarse=(\KK_{1(\coarse)},\dots,\KK_{d(\coarse)})
\end{align*} be a fixed  $d$-dimensional vector of $p$-open knot vectors,  i.e., for each spatial dimension $1\leq i \leq d$,
\begin{align*}
\KK_{i(\coarse)}=(t_{i(\coarse),0},\dots,t_{i(\coarse),N_{i(\coarse)}+p})
\end{align*}
is a \emph{$p$-open knot vector} in $[0,1]$, which means that
\begin{align*}
0=t_{i(\coarse),0}=\dots=t_{i(\coarse),p} < t_{i(\coarse),p+1} \le t_{i(\coarse),p+2} \dots < t_{i(\coarse),N_{i(\coarse)}}=\dots =t_{i(\coarse),N_{i(\coarse)}+p}=1.
\end{align*}
Moreover, we assume that each of the interior knots $t_{i(\coarse),j}\in (0,1)$ appears at most with multiplicity $p$.
By definition, the boundary knots $0$ and $1$ have multiplicity $p+1$.
An example is given in Figure~\ref{fig:bsplines}, with polynomial degree $p=2$, number of knots minus $p$ minus one $=$ $N_{i(\coarse)}=9$ (this will later correspond to the dimension of the B-splines space), and a varying multiplicity.

We define the resulting one-dimensional meshes
\begin{align*}
\widehat\TT_{i(\coarse)}\eq\set{[t_{i(\coarse),j-1},t_{i(\coarse),j}]}{j\in\{1,\dots,N_{i(\coarse)}+p\}\wedge t_{i(\coarse),j-1}<t_{i(\coarse),j}},
\quad i=1,\dots,d,
\end{align*}
as well as the resulting \emph{tensor mesh}
\begin{align*}
\widehat\TT_{\coarse}\eq\set{\widehat T_1\times\dots\times \widehat T_d}{\widehat T_i\in \widehat\TT_{i(\coarse)}\text{ for all }i\in\{1,\dots,d\}}.
\end{align*}

By ${\mathbb{S}}^{p}(\KK_{i(\coarse)})$, we denote the set of all corresponding \emph{univariate splines}, i.e., the set of all $\widehat\TT_{i(\coarse)}$-piecewise univariate polynomials of degree $p$ that are $p-\#t_{i(\coarse),j}$-times continuously differentiable at any interior knot $t_{i(\coarse),j}$, where $\#t_{i(\coarse),j}$ denotes the corresponding multiplicity of $t_{i(\coarse),j}$ within $\KK_{i(\coarse)}$.
Assuming for example that all interior multiplicities are equal to $p$, ${\mathbb{S}}^{p}(\KK_{i(\coarse)})$ is just the space of all $\widehat\TT_{i(\coarse)}$-piecewise univariate polynomials of degree $p$ that are merely continuous, i.e., $C^0(0,1)$, cf.~\eqref{eq:standard_elements}.
At the other extreme, when all interior multiplicities are equal to $1$, then ${\mathbb{S}}^{p}(\KK_{i(\coarse)})$ is the space of $\widehat\TT_{i(\coarse)}$-piecewise univariate polynomials of degree $p$ with continuous derivatives up to order $p-1$, cf.~\eqref{eq:splines} with $m=1$.
We refer again to Figure~\ref{fig:bsplines} for an example.

A basis of ${\mathbb{S}}^{p}(\KK_{i(\coarse)})$, of dimension $N_{i(\coarse)}$, is given by the set of \emph{B-splines}
\begin{align*}
\set{B_{i(\coarse),j,p}}{j\in\{1,\dots,N_{i(\coarse)}\}},
\end{align*}
where the B-splines $B_{i(\coarse),j,p}$ are recursively defined for all points $t\in(0,1)$  via
\begin{subequations}\label{eq_Bsplines}
\begin{align}
B_{i(\coarse),j,0}(t)\eq
\begin{cases}
1\quad&\text{if }t_{i(\coarse),j}\le t<t_{i(\coarse),j+1},\\
0\quad&\text{else},
\end{cases}
\end{align}
and
\begin{align}
B_{i(\coarse),j,p}(t)\eq \frac{t-t_{i(\coarse),j}}{t_{i(\coarse),j+p}-t_{i(\coarse),j}} B_{i(\coarse),j,p-1}(t) + \frac{t_{i(\coarse),j+p+1}-t}{t_{i(\coarse),j+p+1}-t_{i(\coarse),j+1}} B_{i(\coarse),j+1,p-1}(t)
\end{align}
\end{subequations}
with the formal convention that $\cdot/0\eq0$. Figure~\ref{fig:bsplines} gives an illustrative example, where the $N_{i(\coarse)}=9$ basis functions are depicted.
It is easy to see that the support is
\begin{align}\label{eq:B-spline support}
\supp(B_{i(\coarse),j,p})=[t_{i(\coarse),j},t_{i(\coarse),j+p+1}]
\end{align}
and we also remark that
\begin{align} \label{eq_scaling}
0 \leq B_{i(\coarse),j,p} \leq 1.
\end{align}

\begin{figure}[t]
\begin{center}
\includegraphics[width=0.65\textwidth]{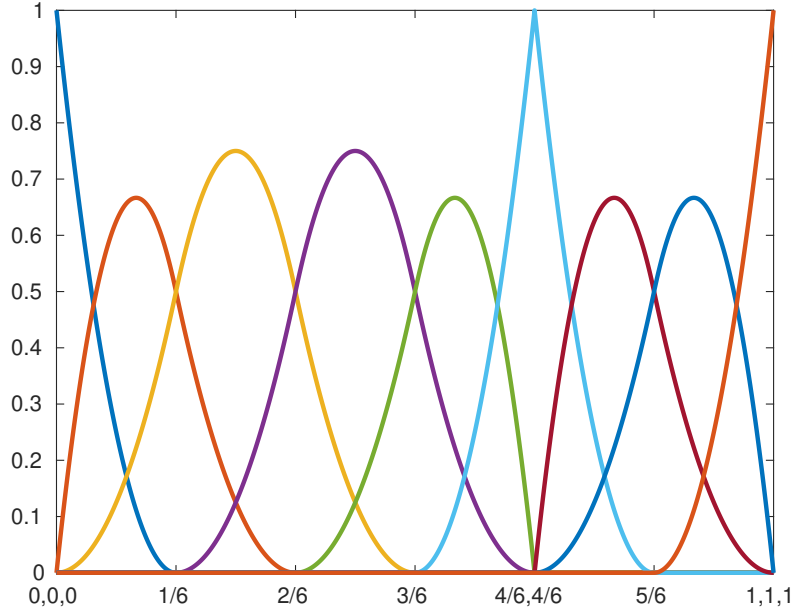}
\end{center}
\caption{\label{fig:bsplines}
The B-splines of degree $p=2$ corresponding to the knot vector $\KK=(0,0,0,1/6,2/6,3/6,4/6,4/6,5/6,1,1,1)$ are depicted.
They are at least $C^1$ at the knots $1/6,2/6,3/6,$ and $5/6$, and at least $C^0$ at $4/6$.
}
\end{figure}

We abbreviate $\p\eq(p,\dots,p)\in\N^d$.
The space  ${\mathbb{S}}^{\p}(\KK_\coarse)$ of \emph{multivariate splines} is defined as tensor-product of the univariate spline spaces.
Note that each function in ${\mathbb{S}}^{\p}(\KK_\coarse)$ is a $\widehat\TT_\coarse$-piecewise multivariate polynomial of degree $\p$.
Again, assuming for example that all interior knots have multiplicity $p$,  ${\mathbb{S}}^{\p}(\KK_\coarse)$ is just the space of all continuous $\widehat\TT_\coarse$-piecewise multivariate  polynomials of degree $\p$, i.e., up to boundary conditions, \eqref{eq:standard_elements}, while multiplicity $1$ of all interior knots yields the space of all $\widehat\TT_\coarse$-piecewise multivariate  polynomials of degree $\p$ with continuous derivatives up to order $p-1$, i.e., up to boundary conditions, \eqref{eq:splines} for $m=1$.
Clearly, the set of tensor-products of the univariate B-splines
\begin{align}\label{eq:multivariate basis}
\set{B_{\coarse,\bm{j},\p}}{\bm{j}\in\Pi_{i=1}^d\{1,\dots,N_{i(\coarse)}\}}
\quad\text{with}\quad
B_{\coarse,\bm{j},\p}(\bm{t})\eq \Pi_{i=1}^{d} B_{i(\coarse),j_i,p}(t_i),
\end{align}
provides a basis of ${\mathbb{S}}^{\p}(\KK_\coarse)$.

%%%%%%%%%%%%%%%%%%%%%%%%%%%%%%%%%%%%%%%%%%%%%%%%%%%%%%%%%%%%%%%%%%%%%%%%%%%%%%%%%%%%%
\subsection{Hierarchical splines in the parameter domain $(0,1)^d$}
\label{sec:hierarchical splines parameter}
%%%%%%%%%%%%%%%%%%%%%%%%%%%%%%%%%%%%%%%%%%%%%%%%%%%%%%%%%%%%%%%%%%%%%%%%%%%%%%%%%%%%%
We now introduce hierarchical splines, which are defined on a hierarchical mesh and are essentially coarse splines on coarse mesh elements and fine splines on fine mesh elements. For a detailed introduction, we refer, e.g., to the seminal work~\cite{vgjs11}.

Let $p$ be a fixed positive  polynomial degree as above and let
${\KK}_0=(\KK_{1(0)},\dots,\KK_{d(0)})$ be a fixed initial $d$-dimensional vector of $p$-open knot vectors in $[0,1]$ with interior multiplicities  less than or equal to $p$, as in Section~\ref{sec:multivariate splines}. Recall that we consider $d =2,3$.
We set $\KK_{\uni{0}}\eq\KK_0$ and  recursively define  $\KK_{\uni{\ell+1}}$ for $\ell\in\N_0$ as the uniform $h$-refinement of $\KK_{\uni{\ell}}$ with fixed multiplicity $m\in\N$, i.e., obtained by inserting the knot $(t_{i(\uni{\ell}),j-1}+t_{i(\uni{\ell}),j})/{2}$ to the knots $\KK_{i(\uni{\ell})}$ with multiplicity $m$ whenever $t_{i(\uni{\ell}),j-1}<t_{i(\uni{\ell}),j}$.
We use analogous notation as in Section~\ref{sec:multivariate splines}, replacing the index $h$ by $\uni{\ell}$, e.g., we write $\widehat\TT_{\uni{\ell}}$ for the induced tensor mesh.
We stress that these spline spaces are nested in the sense that
\begin{align}
 {\mathbb{S}}^{\p}(\KK_{\uni{\ell}})\subset{\mathbb{S}}^{\p}(\KK_{\uni{\ell+1}})\subset C^{0}((0,1)^d),
\end{align}
where the last relation follows from the assumption that multiplicities of interior knots is less than or equal to $p$.

\begin{remark}
As a matter of fact, the definition of the uniform refinements $\KK_{\uni{\ell}}$ allows for a certain flexibility without changing the results of the present manuscript.
For instance, instead of the natural dyadic uniform refinement, one could use $n$-adic refinement, i.e., insert the knots $(t_{i(\uni{\ell}),j-1}+t_{i(\uni{\ell}),j})/{n},\dots,(n-1)(t_{i(\uni{\ell}),j-1}+t_{i(\uni{\ell}),j})/{n}$ to the knots $\KK_{i(\uni{\ell})}$ with multiplicity $m$ whenever $t_{i(\uni{\ell}),j-1}<t_{i(\uni{\ell}),j}$.
\end{remark}

We say that a set
\begin{align*}
\widehat\TT_\coarse\subseteq\bigcup_{\ell\in\N_0}\widehat \TT_{\uni{\ell}}
\end{align*}
 is a \textit{hierarchical mesh} if it is a partition of $[0,1]^d$ in the sense that $\bigcup\widehat \TT_\coarse=[0,1]^d$, where the intersection of  two different elements $\widehat T\neq \widehat T'$ with $\widehat T,\widehat T'\in\widehat\TT_\coarse$ has ($d$-dimensional) measure zero.
Since $\widehat\TT_{\uni{\ell}}\cap\widehat\TT_{\uni{\ell'}}=\emptyset$ for $\ell,\ell'\in\N_0$ with $\ell\neq \ell'$, we can define for any element $\widehat T\in\widehat\TT_\coarse$,
\begin{align*}
\level(\widehat T)\eq\ell\in\N_0 \quad\text{with }\widehat T\in\widehat\TT_{\uni{\ell}}.
\end{align*}
 For an illustrative example of a hierarchical mesh, see Figure~\ref{fig:hiermesh}.
In particular, any uniformly refined tensor mesh $\widehat\TT_{\uni{\ell}}$ with $\ell\in\N_0$ is a hierarchical mesh.

For a hierarchical mesh $\widehat\TT_\coarse$, we define a corresponding nested sequence $(\widehat\Omega^\ell_\coarse)_{\ell\in\N_0}$ of closed subsets of $[0,1]^d$ as
\begin{align*}
\widehat\Omega_\coarse^\ell\eq\bigcup_{\ell'\ge \ell}\big(\widehat \TT_\coarse\cap\widehat\TT_{\uni{\ell'}}\big),
\end{align*}
i.e., $\widehat\Omega_\coarse^\ell$ covers all elements with level greater than or equal to $\ell$.
Note that there exists a minimal integer $L_\coarse$ such that $\widehat\Omega_\coarse^{\ell}=\emptyset$ for all $\ell\ge L_\coarse$.
It holds that
\begin{align}\label{eq:parameter mesh}
\widehat\TT_\coarse=\bigcup_{\ell\in\N_0}\set{\widehat T\in\widehat\TT_{\uni{\ell}}}{\widehat T\subseteq \widehat \Omega_\coarse^\ell\wedge \widehat T\not\subseteq\widehat\Omega_\coarse^{\ell+1}}.
\end{align}
In the literature, one usually assumes that the sequence $(\widehat\Omega^\ell_\coarse)_{\ell\in\N_0}$ is given and the corresponding hierarchical mesh is defined via~\eqref{eq:parameter mesh}.

%%%%%%%
\begin{figure}
\begin{center}
\begin{tikzpicture}[scale=7.5]
\draw[line width = 0.3mm, color = black]  (0,0) -- (1,0) -- (1,1) -- (0,1) -- cycle;
\draw[color = white, fill=gray!30]  (0,0) -- (3/8,0) -- (3/8,3/8) -- (0,3/8) -- cycle;
\draw[color = white, fill=gray!30]  (1/2,1/2) -- (1,1/2) -- (1,1) -- (1/2,1) -- cycle;
\draw[color = white, fill=gray!30]  (0,3/4) -- (1/8,3/4) -- (1/8,1) -- (0,1) -- cycle;
\draw[color = white, fill=gray!30]  (0,3/8) -- (1/8,3/8) -- (1/8,1/2) -- (0,1/2) -- cycle;
\draw[color = white, fill=gray!55]  (0,0) -- (1/16,0) -- (1/16,1/8) -- (0,1/8) -- cycle;
\draw[color = white, fill=gray!55]  (13/16,13/16) -- (1,13/16) -- (1,1) -- (13/16,1) -- cycle;
\draw[color = white, fill=gray!55]  (0,15/16) -- (1/16,15/16) -- (1/16,1) -- (0,1) -- cycle;
\draw[color = white, fill=gray!100]  (31/32,31/32) -- (1,31/32) -- (1,1) -- (31/32,1) -- cycle;
\foreach \i in {0,...,8}{
\draw[line width=0.3mm, color=black] (0,\i/8) -- (1,\i/8);
\draw[line width=0.3mm, color=black] (\i/8,0) -- (\i/8,1);
}
\foreach \i in {0,...,6}{
\draw[line width=0.3mm, color=black] (0,\i/16) -- (3/8,\i/16);
\draw[line width=0.3mm, color=black] (\i/16,0) -- (\i/16,3/8);
}
\foreach \i in {8,...,16}{
\draw[line width=0.3mm, color=black] (1/2,\i/16) -- (1,\i/16);
\draw[line width=0.3mm, color=black] (\i/16,1/2) -- (\i/16,1);
}
\foreach \i in {0,...,2}{
\draw[line width=0.3mm, color=black] (0,\i/16+14/16) -- (1/8,\i/16+14/16);
\draw[line width=0.3mm, color=black] (0,\i/16+12/16) -- (1/8,\i/16+12/16);
\draw[line width=0.3mm, color=black] (\i/16,3/4) -- (\i/16,1);
}
\foreach \i in {0,...,2}{
\draw[line width=0.3mm, color=black] (0,\i/16+6/16) -- (1/8,\i/16+6/16);
\draw[line width=0.3mm, color=black] (\i/16,3/8) -- (\i/16,4/8);
}
\foreach \i in {0,...,2}{
\draw[line width=0.3mm, color=black] (0,\i/32) -- (1/16,\i/32);
\draw[line width=0.3mm, color=black] (0,\i/32+2/32) -- (1/16,\i/32+2/32);
\draw[line width=0.3mm, color=black] (\i/32,0) -- (\i/32,2/16);
}
\foreach \i in {26,...,32}{
\draw[line width=0.3mm, color=black] (13/16,\i/32) -- (1,\i/32);
\draw[line width=0.3mm, color=black] (\i/32,13/16) -- (\i/32,1);
}
\foreach \i in {0,...,2}{
\draw[line width=0.3mm, color=black] (0,\i/32+30/32) -- (1/16,\i/32+30/32);
\draw[line width=0.3mm, color=black] (\i/32,15/16) -- (\i/32,1);
}
\foreach \i in {62,...,64}{
\draw[line width=0.3mm, color=black] (62/64,\i/64) -- (1,\i/64);
\draw[line width=0.3mm, color=black] (\i/64,62/64) -- (\i/64,1);
}
\end{tikzpicture}

\caption{\label{fig:hiermesh}
A two-dimensional hierarchical mesh $\widehat\TT_\coarse$ with level of all elements less than $4$. Levels $0$ to $3$ are respectively highlighted in white, light grey, grey, and dark grey, also denoting the corresponding  domains $\widehat\Omega^0_\coarse\setminus\widehat\Omega^1_\coarse$, $\widehat\Omega^1_\coarse\setminus\widehat\Omega^2_\coarse$,
$\widehat\Omega^2_\coarse\setminus\widehat\Omega^3_\coarse$, $\widehat\Omega^3_\coarse\setminus\widehat\Omega^4_\coarse$; $\widehat\Omega_\coarse^4=\emptyset$.}
\end{center}
\end{figure}
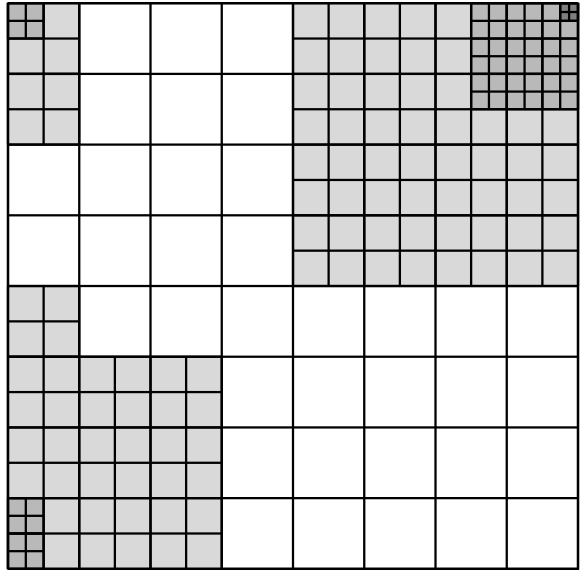
%%%%%%%

We introduce the \textit{hierarchical basis}
\begin{align}\label{eq:short cHH}
\set{B_\ver}{\ver\in\II_\coarse} \quad\text{with }\quad \II_\coarse\eq\bigcup_{\ell\in\N_0}\Big\{&(\uni{\ell},\bm{j},\p) : \bm{j}\in\Pi_{i=1}^d\{1,\dots,N_{i(\uni{\ell})}\}\\
&\wedge\supp(B_{\uni{\ell},\bm{j},\p})\subseteq\widehat\Omega_\coarse^\ell\wedge\supp(B_{\uni{\ell},\bm{j},\p})\not\subseteq\widehat\Omega_\coarse^{\ell+1}\Big\},\notag
\end{align}
where we recall the definition~\eqref{eq:multivariate basis} of a multivariate B-spline. Figure~\ref{fig:partition} gives an illustrative example.
Its elements are referred to as (multivariate) \textit{hierarchical} \textit{B-splines}.
For $\ver\in\II_\coarse$, the level  of the corresponding hierarchical B-spline is well defined
\begin{align}\label{eq:level B}
\level(B_\ver)\eq\ell\in\N_0 \quad\text{with } \ver=(\uni{\ell},\bm{j},\p).
\end{align}
It is easy to check that if $\widehat \TT_\coarse$ is a tensor mesh, and hence coincides with some $\widehat\TT_{\uni{\ell}}$, then the hierarchical basis and the standard tensor-product B-spline basis are the same.
One can prove that the hierarchical B-splines
are linearly independent; see, e.g., \cite[Lemma~2]{vgjs11}.
They span the space of \textit{hierarchical splines}
\begin{align}\label{eq:hierarchical_splines}
{\mathbb{S}}^{\p}(\KK_0,m,\widehat\TT_\coarse)\eq{\rm span}\big(\set{B_\ver}{\ver\in\II_\coarse}\big).
\end{align}

%%%%%%%
\begin{figure}
\begin{center}
\begin{tikzpicture}[scale=4.6]
\draw[line width = 0.3mm, color = black]  (0,0) -- (1,0) -- (1,1) -- (0,1) -- cycle;
\foreach \i in {0,...,2}{
\draw[line width=0.15mm,color=gray!50] (1/4,\i/16) -- (5/8,\i/16);
\draw[line width=0.15mm,color=gray!50] (1/4+\i/16,0) -- (1/4+\i/16,1/8);
\draw[line width=0.15mm,color=gray!50] (3/8+\i/16,0) -- (3/8+\i/16,1/8);
\draw[line width=0.15mm,color=gray!50] (1/2+\i/16,0) -- (1/2+\i/16,1/8);
}
\foreach \i in {0,...,6}{
\draw[line width=0.15mm,color=gray!50] (11/16,\i/32+11/16) -- (14/16,\i/32+11/16);
\draw[line width=0.15mm,color=gray!50] (\i/32+11/16,11/16) -- (\i/32+11/16,14/16);
}
\foreach \i in {0,...,12}{
\draw[line width=0.15mm,color=gray!50] (0,\i/32+20/32) -- (3/8,\i/32+20/32);
\draw[line width=0.15mm,color=gray!50] (\i/32,5/8) -- (\i/32,1);
}
\foreach \i in {0,...,8}{
\draw[line width=0.3mm] (0,\i/8) -- (1,\i/8);
\draw[line width=0.3mm] (\i/8,0) -- (\i/8,1);
}
\foreach \i in {0,...,6}{
\draw[line width=0.3mm] (0,\i/16) -- (3/8,\i/16);
\draw[line width=0.3mm] (\i/16,0) -- (\i/16,3/8);
}
\foreach \i in {8,...,16}{
\draw[line width=0.3mm] (1/2,\i/16) -- (1,\i/16);
\draw[line width=0.3mm] (\i/16,1/2) -- (\i/16,1);
}
\foreach \i in {0,...,2}{
\draw[line width=0.3mm] (0,\i/16+14/16) -- (1/8,\i/16+14/16);
\draw[line width=0.3mm] (0,\i/16+12/16) -- (1/8,\i/16+12/16);
\draw[line width=0.3mm] (\i/16,3/4) -- (\i/16,1);
}
\foreach \i in {0,...,2}{
\draw[line width=0.3mm] (0,\i/16+6/16) -- (1/8,\i/16+6/16);
\draw[line width=0.3mm] (\i/16,3/8) -- (\i/16,4/8);
}
\foreach \i in {0,...,2}{
\draw[line width=0.3mm] (0,\i/32) -- (1/16,\i/32);
\draw[line width=0.3mm] (0,\i/32+2/32) -- (1/16,\i/32+2/32);
\draw[line width=0.3mm] (\i/32,0) -- (\i/32,2/16);
}
\foreach \i in {26,...,32}{
\draw[line width=0.3mm] (13/16,\i/32) -- (1,\i/32);
\draw[line width=0.3mm] (\i/32,13/16) -- (\i/32,1);
}
\foreach \i in {0,...,2}{
\draw[line width=0.3mm] (0,\i/32+30/32) -- (1/16,\i/32+30/32);
\draw[line width=0.3mm] (\i/32,15/16) -- (\i/32,1);
}
\foreach \i in {62,...,64}{
\draw[line width=0.3mm] (62/64,\i/64) -- (1,\i/64);
\draw[line width=0.3mm] (\i/64,62/64) -- (\i/64,1);
}
\draw[line width=1.0mm,color = blue] (5/8,0) -- (5/8,1/8) -- (2/8,1/8) -- (2/8,0) -- cycle;
\draw[line width=1.0mm,color = blue]  (5/8+1/16,5/8+1/16) -- (14/16,5/8+1/16) -- (14/16,14/16) -- (5/8+1/16,14/16)
-- cycle;
\draw[line width=1.0mm,color = blue]  (0,5/8) -- (3/8,5/8) -- (3/8,1) -- (0,1) -- cycle;
\end{tikzpicture}

(a)

\includegraphics[width=0.51\textwidth]{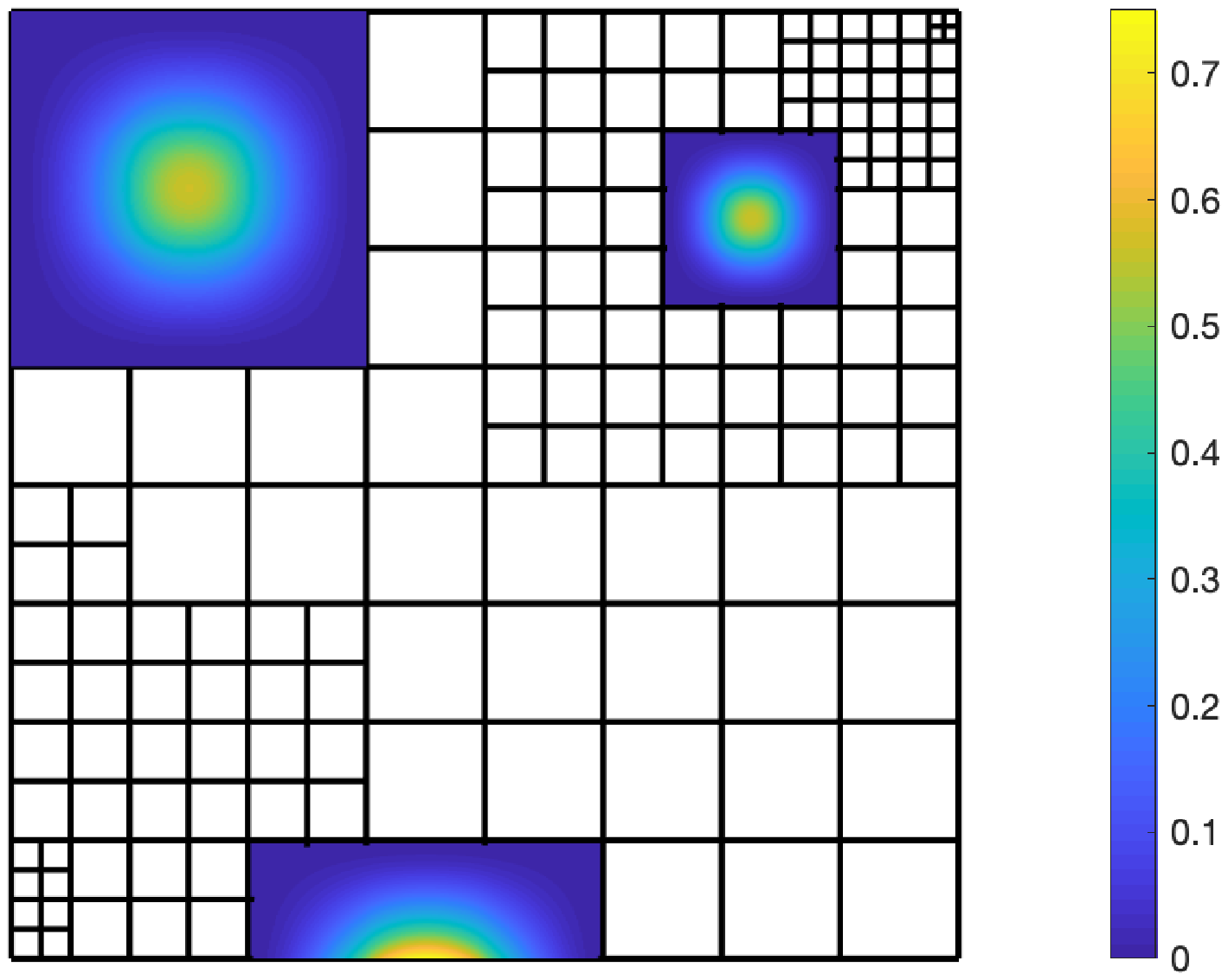}
\hspace{-7.5mm}
\includegraphics[width=0.51\textwidth]{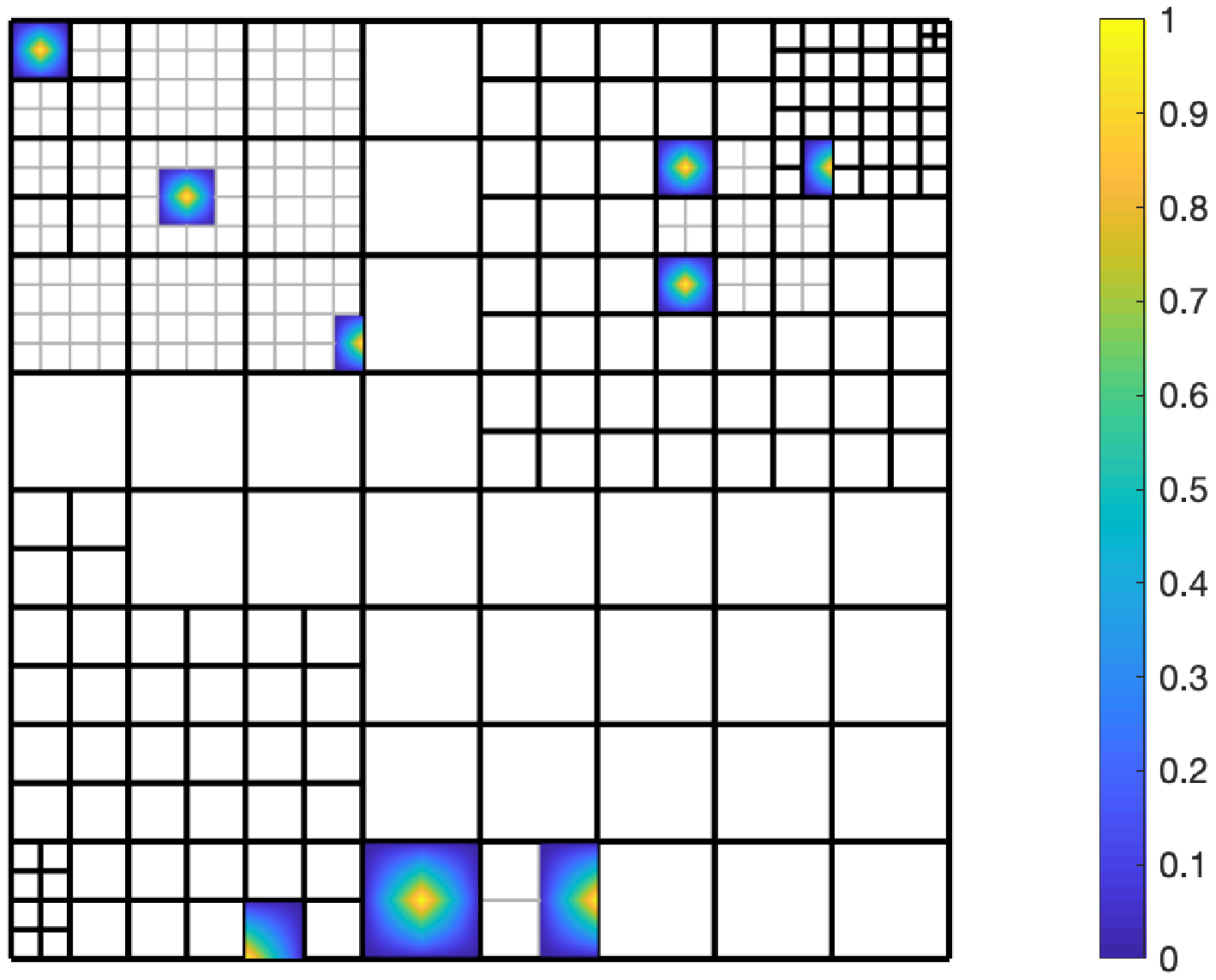}

\vspace{-8mm}
(b) \hspace{70mm} (c)

\caption{\label{fig:partition}
A two-dimensional hierarchical mesh $\widehat\TT_\coarse$ (of Figure~\ref{fig:hiermesh}) is depicted in black.
Assuming that all interior knots have multiplicity $1$, in (a), the support of three hierarchical B-splines of degree $p=2$ is highlighted in blue; two hierarchical B-splines of level $0$ (left and bottom), and one hierarchical B-spline of level $1$ (right).
The corresponding local tensor meshes $\widehat\TT_\ver$ are indicated in grey.
In (b), these three hierarchical B-splines are depicted. For $\overline p=2$, these functions also belong (up to scaling) to the partition of unity $\set{\psi_\ver\circ\F}{\ver\in\VV_\coarse}$ defined in Section~\ref{sec:hierarchical partitions} below.
Finally, in (c), for each of them, three functions of the corresponding local partition of unity $\set{\psi_\vertt\circ \F}{\vertt\in \VV_\coarse^\ver}$ defined in Section~\ref{sec:hierarchical partitions} below over the local tensor meshes are depicted.}
\end{center}
\end{figure}
%%%%%%%

The hierarchical basis
and the mesh $\widehat\TT_\coarse$ are compatible in the following sense:
For all $B_\ver$, $\ver\in\II_\coarse$, the corresponding support can be written as union of elements in $\widehat\TT_{\uni{\level(B_\ver)}}$, i.e.,
\begin{align*}
\supp(B_\ver)=\bigcup\set{\widehat T\in\widehat\TT_{\uni{\level(B_\ver)}}}{\widehat T\subseteq\supp(B_\ver)}.
\end{align*}
Each such element $\widehat T\in\widehat\TT_{\uni{\level(B_\ver)}}$ with $\widehat T\subseteq\supp(B_\ver)\subseteq \widehat\Omega_\coarse^{\level(B_\ver)}$ satisfies that $\widehat T\in\widehat\TT_\coarse$ or $\widehat T\subseteq\widehat\Omega_\coarse^{\level(B_\ver)+1}$.
In either case, we see that $\widehat T$
can be written as union of elements in $\widehat\TT_\coarse$ with level greater or equal to $\level(B_\ver)$.
Altogether, we have that
\begin{align}\label{eq:supp elements}
\supp(B_\ver)=\bigcup_{\ell\ge \level(B_\ver)}\set{\widehat T\in\widehat\TT_\coarse\cap\widehat\TT_{{\rm uni}(\ell)}}{\widehat T\subseteq\supp(B_\ver)}.
\end{align}
Moreover, $\supp(B_\ver)$ must contain at least one element of $\level(B_\ver)$.
Otherwise one would get the contradiction $\supp(B_\ver)\subseteq\widehat\Omega_\coarse^{\level(B_\ver)+1}$.

We now present the following characterization of hierarchical splines, which has been verified  in \cite[Section~3]{sm16},
\begin{align}\label{eq:hier in pardom defined}
\begin{split}
{\mathbb{S}}^{\p}(\KK_0,m,\widehat\TT_\coarse)=\big\{\widehat v_\coarse:(0,1)^d\to\R:\widehat v_\coarse|_{(0,1)^d\setminus \widehat\Omega_\coarse^{\ell+1}}\in  \mathbb{S}^{\p}(\KK_{\uni{\ell}})|_{(0,1)^d\setminus \widehat\Omega_\coarse^{\ell+1}} \text{ for all }\ell\in \N_0\big\}.
\end{split}
\end{align}
In particular, each hierarchical spline is a $\widehat{\TT}_\coarse$-piecewise tensor-product polynomial of degree $\p$.
Put into words, hierarchical splines are coarse splines on coarse mesh elements, and they are fine splines on fine mesh elements.

Finally, we say that a hierarchical mesh $\widehat\TT_\coarse$ is \textit{finer} than another hierarchical mesh $\widehat\TT_H$ if $\widehat\TT_\coarse$ is obtained from $\widehat\TT_H$ via iterative dyadic bisection.
Formally, this can be stated as $\widehat\Omega_H^\ell\subseteq\widehat\Omega_\coarse^\ell$ for all $\ell\in\N_0$.
In this case, \eqref{eq:hier in pardom defined}  shows that the corresponding hierarchical spline spaces are nested, i.e.,
\begin{align}\label{eq:hierarchical splines nested}
{\mathbb{S}}^{\p}(\KK_0,m,\widehat\TT_H)\subseteq {\mathbb{S}}^{\p}(\KK_0,m,\widehat\TT_\coarse).
\end{align}
In particular, we see that
\begin{align*}
{\mathbb{S}}^{\p}(\KK_{\uni{0}})\subseteq{\mathbb{S}}^{\p}(\KK_0,m,\widehat\TT_\coarse)
\subseteq{\mathbb{S}}^{\p}(\KK_{\uni{L_\coarse-1}}).
\end{align*}

%%%%%%%%%%%%%%%%%%%%%%%%%%%%%%%%%%%%%%%%%%%%%%%%%%%%%%%%%%%%%%%%%%%%%%%%%%%%%%%%%%%%%
\subsection{Hierarchical splines in the physical domain $\Omega$}
%%%%%%%%%%%%%%%%%%%%%%%%%%%%%%%%%%%%%%%%%%%%%%%%%%%%%%%%%%%%%%%%%%%%%%%%%%%%%%%%%%%%%
If $\widehat\TT_\coarse$ is a hierarchical mesh in the parameter domain $(0,1)^d$, we set
\begin{align*}
\TT_\coarse\eq\set{\F(\widehat T)}{\widehat T\in\widehat\TT_\coarse},
\end{align*}
where $\F$ is the bi-Lipschitz mapping from Section~\ref{sec:setting}.
In this context, $h\in L^\infty(\Omega)$ denotes the mesh size function defined by $h|_T:=\diam(T)$ for all $T\in\TT_\coarse$.
Moreover, we set
\begin{align*}
{\mathbb{S}}^{\p}(\KK_0,m,\TT_\coarse)\eq\set{\widehat v\circ \bm{F}^{-1}}{\widehat v\in {\mathbb{S}}^{\p}(\KK_0,m,\widehat\TT_\coarse)}.
\end{align*}
The conforming ansatz space for the Galerkin discretization~\eqref{eq:Galerkin} is then defined as
\begin{align}
V_\coarse\eq
{\mathbb{S}}^{\p}(\KK_0,m,\TT_\coarse)\cap H_0^1(\Omega).
\end{align}

\begin{remark}\label{rem:parametrization}
In practice, the parametrization $\F$ is often given in terms of non-uniform rational B-splines (NURBS) along with corresponding polynomial degree $\p_{\F}$ and global knot vector $\KK_{\F}$.
To guarantee good approximation properties of $V_\coarse$, in particular good \emph{a priori} estimates, the functions in ${\mathbb{S}}^{\p}(\KK_0,m,\widehat\TT_\coarse)$ should have the same or lower smoothness as $\F$ along the knot lines on $\widehat\Omega$ corresponding to $\KK_{\F}$.
Such a requirement is not needed for the presented \emph{a posteriori} error analysis to hold. 
Moreover, the form of the oscillation terms~\eqref{eq_osc} and~\eqref{eq:oscillations_glob} for reliability and efficiency, respectively, suggests that the approximation property of $V_\coarse$ has minor influence on the size of these terms, see also Remark~\ref{rem:align}.
\end{remark}

%%%%%%%%%%%%%%%%%%%%%%%%%%%%%%%%%%%%%%%%%%%%%%%%%%%%%%%%%%%%%%%%%%%%%%%%%%%%%%%%%%%%%
%%%%%%%%%%%%%%%%%%%%%%%%%%%%%%%%%%%%%%%%%%%%%%%%%%%%%%%%%%%%%%%%%%%%%%%%%%%%%%%%%%%%%
\section{Partitions of unity and patchwise spaces}\label{sec:partitions}
%%%%%%%%%%%%%%%%%%%%%%%%%%%%%%%%%%%%%%%%%%%%%%%%%%%%%%%%%%%%%%%%%%%%%%%%%%%%%%%%%%%%%
%%%%%%%%%%%%%%%%%%%%%%%%%%%%%%%%%%%%%%%%%%%%%%%%%%%%%%%%%%%%%%%%%%%%%%%%%%%%%%%%%%%%%

In this section, we prepare the necessary material for defining and analyzing our equilibrated fluxes in the IGA context later.
We particularly design a partition of unity based on hierarchical B-splines and define continuous- and discrete-level local spaces.

%%%%%%%%%%%%%%%%%%%%%%%%%%%%%%%%%%%%%%%%%%%%%%%%%%%%%%%%%%%%%%%%%%%%%%%%%%%%%%%%%%%%%
\subsection{Partitions of unity based on hierarchical splines}
\label{sec:hierarchical partitions}
%%%%%%%%%%%%%%%%%%%%%%%%%%%%%%%%%%%%%%%%%%%%%%%%%%%%%%%%%%%%%%%%%%%%%%%%%%%%%%%%%%%%%
%We now present a practical construction of the above partitions in our hierarchical splines setting.
We now first construct a partition of unity on $\Omega$ consisting of hierarchical B-splines with the smallest-possible polynomial degree $\overline p$ but sufficient smoothness to be contained in ${\mathbb{S}}^{\p}(\KK_0,m,\TT_\coarse)$. Subsequently, on the local tensor meshes of the support of each of these hierarchical B-splines, we form a partition of unity by piecewise multilinear hat functions with merely $C^0(\Omega)$ continuity.

Let $\overline p\le p$ be a supplementary polynomial degree and let
\begin{align*}
{\overline\KK}_0=(\overline\KK_{1(0)},\dots,\overline\KK_{d(0)})
\end{align*}
 be a fixed  $d$-dimensional vector of $\overline p$-open knot vectors
\begin{align*}
\overline\KK_{i(0)}=(\overline t_{i(0),0},\dots,\overline t_{i(0),\overline N_{i(0)}+\overline p})
\end{align*}
such that $(\overline t_{i(0),\overline p},\dots,\overline t_{i(0),\overline N_{i(0)}})$ is  a subsequence of $(t_{i(0),p},\dots,t_{i(0),N_{i(0)}})$ which is obtained by reducing multiplicities of the latter knots to at least one and
\begin{align*}
0=\overline t_{i(0),0}=\dots=\overline t_{i(0),\overline p}
\quad\text{and}\quad
\overline t_{i(0),\overline N_{i(0)}}=\dots =\overline t_{i(0),\overline N_{i(\coarse)}+\overline p}=1.
\end{align*}
In particular, the  tensor mesh corresponding to $\overline\KK_0$ coincides with the initial tensor-mesh $\widehat\TT_0$ corresponding to $\KK_0$.
To guarantee that
\begin{align}
{\mathbb{S}}^{\overline p}(\overline\KK_{i(0)})\subseteq{\mathbb{S}}^{p}(\KK_{i(0)}),
\end{align}
we further suppose that
\begin{align} \label{eq_ov_p}
p - \# t_{i(0),j} \le \overline p-\overline \#t_{i(0),j}
\end{align}
for all interior knots $t_{i(0),j}$ in $(0,1)$ (which determines the smoothness of the considered splines), where $\overline \#$ denotes the multiplicity within $\overline\KK_{i(0)}$.
Next, we set $\overline\KK_{\uni{0}}\eq\overline\KK_0$ and recursively define  $\overline\KK_{\uni{\ell+1}}$ for $\ell\in\N_0$ as the uniform $h$-refinement of $\overline\KK_{\uni{\ell}}$ with fixed multiplicity $\overline m\in\N$ such that
\begin{align} \label{eq_ov_p_2}
p-m\le \overline p-\overline m.
\end{align}
In words, the knots of $\overline\KK_{\uni{\ell+1}}$ in $[0,1]$ are the knots of $\overline\KK_{\uni{\ell}}$ in $[0,1]$ plus the points $(t_{i(\uni{\ell}),j-1}+t_{i(\uni{\ell}),j})/{2}$ with multiplicity $\overline m$ if $t_{i(\uni{\ell}),j-1}<t_{i(\uni{\ell}),j}$ for $j\in\{\overline p+1,\dots, \overline N_{i(\uni{\ell})}\}$.
If, for example, all interior knots in $\KK_0$ have the same multiplicity $1$ and $m=1$, i.e., the corresponding splines are $C^{p-1}$ along initial as well as new lines, one can only choose $\overline p = p$ and $\overline m =1$.
If all interior knots in $\KK_0$ have the same multiplicity $p$ and $m=p$, i.e., the corresponding splines are only $C^{0}$ along initial as well as new lines, one can choose $\overline p \le p$ and $\overline m \le \overline p$ arbitrarily, which leads us to $\overline p = 1$ and $\overline m=1$.
While the analysis below does not rely on this, we will always choose the smallest-possible polynomial degree $\overline p$ together with $\overline m:=1$.
This is the most sensible choice from a practical as well as theoretical point of view, as the efficiency constant of our estimator depends  on the degree $\overline p$; see Proposition~\ref{thm_ass} and Remark~\ref{rem_p_rob} as well as Proposition~\ref{prop:efficiency} and Remark~\ref{rem:refinement} below.
In Remark~\ref{rem:new}, we further discuss that the choice of $\overline p$ and $\overline m$ has little influence on the data oscillation terms arising in our error estimator.

Let again $\widehat\TT_\coarse$ be a hierarchical mesh with corresponding sets $(\widehat \Omega_\coarse^\ell)_{\ell\in\N_0}$ as in Section~\ref{sec:hierarchical splines parameter}.
Setting $\overline \p\eq(\overline p,\dots,\overline p)$, we define the B-splines $\overline B_{(\uni{\ell},\bm{j},\overline\p)}$, the hierarchical basis $\set{\overline B_\ver}{\ver\in\overline\II_\coarse}$ with index set $\overline\II_\coarse$, the level $\level(\overline B_\ver)$ for $\ver\in\II_\coarse$, and the spanned space of hierarchical splines ${\mathbb{S}}^{\p}(\overline\KK_0,\overline m,\widehat\TT_\coarse)$ as in \eqref{eq:short cHH}--\eqref{eq:hierarchical_splines}.
Again,  \cite[Section~3]{sm16} gives an explicit characterization for the spanned space of hierarchical splines.
Our assumptions on the knot multiplicities, which imply the nestedness $\mathbb{S}^{\overline\p}(\overline\KK_{\uni{\ell}})\subseteq \mathbb{S}^{\p}(\KK_{\uni{\ell}})$ for all $\ell\in\N_0$, thus give that
\begin{align}
{\mathbb{S}}^{\overline\p}({\overline\KK}_0,\overline m,\widehat\TT_\coarse)
\subseteq{\mathbb{S}}^{\p}({\KK}_0, m,\widehat\TT_\coarse).
\end{align}

Now $1\in {\mathbb{S}}^{\overline\p}(\overline\KK_0,\overline m,\widehat\TT_\coarse)$ yields the existence of a partition of unity on the parameter domain
\begin{align*}
1=\sum_{\ver\in\overline\II_\coarse} c_\ver \overline B_\ver.
\end{align*}
One can prove that the coefficients $c_\ver\in\R$ satisfy that
\begin{align}\label{eq_coefs}
    0\le c_\ver\le1
\end{align}
 for all  $\ver\in\overline\II_\coarse$; see, e.g., \cite[Lemma~3.2]{bg16b}.
Consequently, we can define
\begin{align}\label{eq_psi_a_def}
\psi_\ver\eq(c_\ver \overline B_\ver)\circ\F^{-1}\quad\text{for all }\ver\in\VV_\coarse\eq\set{\ver\in\overline\II_\coarse}{c_\ver>0}
\end{align}
and observe that the $\psi_\ver$ form a {\em partition of unity} on the physical domain
\begin{align*}
\sum_{\ver\in\VV_\coarse}\psi_\ver=1 \quad\text{in }\Omega.
\end{align*}
Henceforth, we call $\VV_\coarse$ the set of nodes and $\ver \in \VV_\coarse$ a \emph{node}. For further use, we abbreviate $\omega_\ver \eq {\rm int}(\supp(\psi_\ver))$ as well as $\widehat\omega_\ver\eq\F^{-1}(\omega_\ver)$ for all $\ver\in\VV_\coarse$; we will use the terminology \emph{large patches} for $\omega_\ver$ or $\widehat\omega_\ver$. Figure~\ref{fig:partition} gives an illustrative example.

Below, we will also crucially use a second partition of unity on each large patch $\omega_\ver$. Let $\widehat \TT_\ver$ be the smallest {\em uniform} tensor mesh refinement of $\set{\widehat T\in\widehat\TT_\coarse}{\widehat T\subseteq\supp(B_\ver)}$, i.e.,
\begin{align*}
\widehat \TT_\ver\eq\set{\widehat T\in\widehat\TT_{\uni{\ell_\ver}}}{\widehat T\subseteq\supp(B_\ver)}
\end{align*}
with $\ell_\ver\eq\max\set{\level(\widehat T)}{\widehat T\in \widehat\TT_\coarse \wedge \widehat T\subseteq\supp(B_\ver)}$,
and  $\TT_\ver$ the corresponding mesh of the large patch $\omega_\ver$ in the physical domain; see again Figure~\ref{fig:partition}.
Moreover, let $\mathbb{Q}^{\bm{1}}(\widehat\TT_\ver)$ be the set of all $\widehat\TT_\ver$-piecewise tensor-product polynomials of degree $\bm{1}\eq(1,\dots,1)$ and
\begin{align}\label{eq_PTa}
\mathbb{Q}^{\bm{1}}(\TT_\ver)\eq\set{\widehat v\circ\F^{-1}|_{\omega_{\ver}}}{\widehat v\in \mathbb{Q}^{\bm{1}}(\widehat\TT_\ver)}.
\end{align}
Finally, let $\VV_\coarse^\ver$ be the set of all vertices in the local mesh $\TT_\ver$. We denote by $\psi_\vertt$ the hat function associated with the vertex $\vertt \in \VV_\coarse^\ver$; this is the unique function in $\mathbb{Q}^{\bm{1}}(\TT_\ver)$ $\cap C^0(\oma)$ taking value $1$ in the vertex $\vertt$ and $0$ in all other vertices from $\VV_\coarse^\ver$.
Observe that the $\psi_\vertt$ form a {\em partition of unity} on the large patches $\omega_\ver$
\begin{align*}
\sum_{\vertt\in\VV_\coarse^\ver}\psi_\vertt=1\quad\text{in }\omega_\ver.
\end{align*}
We abbreviate $\omega_\vertt \eq {\rm int}(\supp(\psi_\vertt))$ as well as $\widehat\omega_\vertt\eq\F^{-1}(\omega_\vertt)$ for all $\vertt\in\VV_\coarse^\ver$, for which we use the name \emph{small patches}.
Figure~\ref{fig:partition} gives again an illustrative example.

%%%%%%%%%%%%%%%%%%%%%%%%%%%%%%%%%%%%%%%%%%%%%%%%%%%%%%%%%%%%%%%%%%%%%%%%%%%%%%%%%%%%%
\subsection{Patchwise Sobolev spaces} \label{sec_patc_spaces}
%%%%%%%%%%%%%%%%%%%%%%%%%%%%%%%%%%%%%%%%%%%%%%%%%%%%%%%%%%%%%%%%%%%%%%%%%%%%%%%%%%%%%
For a node $\ver\in\VV_\coarse$, define a local Sobolev space on the large patch $\omega_\ver$ as
\begin{align}\label{eq:H1_a}
H_*^1(\omega_\ver)\eq
\begin{cases}
\set{v\in H^1(\omega_\ver)}{\dual{v}{1}_{\omega_\ver}=0}\quad&\text{if }\psi_\ver\in H_0^1(\Omega),\\
\set{v\in H^1(\omega_\ver)}{v=0\text{ on }
\partial\omega_\ver\setminus\psi_\ver^{-1}(\{0\})
}
\quad&\text{else.}
\end{cases}
\end{align}
This is the mean-value-free subspace of $H^1(\omega_\ver)$ in the interior of $\Omega$, and the trace-free (on that part of $\partial\omega_\ver$ where $\psi_\ver$ is nonzero) subspace of $H^1(\omega_\ver)$ adjacent to the boundary of $\Omega$. For vector-valued functions, we will use
\begin{align}\label{eq:H0 div for a}
\H_0(\div,\omega_\ver)\eq
\begin{cases}
\set{\bv\in \H(\div,\omega_\ver)}{\bv{\cdot}\n_{\omega_\ver}=0 \text{ on }\partial\omega_\ver}&\hspace{-2mm}\text{if }\psi_\ver\in H_0^1(\Omega),\\
\set{\bv\in \H(\div,\omega_\ver)}{\bv{\cdot}\n_{\omega_\ver}=0\text{ on }\partial\omega_\ver\cap(\psi_\ver)^{-1}(\{0\})}&\hspace{-2mm}\text{else},\\
\end{cases}
\end{align}
where $\n_{\omega_\ver}$ denotes the outer normal vector on $\partial\omega_\ver$ and $\bv{\cdot}\n_{\omega_\ver}$ is understood in the appropriate weak sense. These are the normal-component-free subspaces of $\H(\div,\omega_\ver)$: everywhere on $\partial \oma$ in the interior of $\Omega$ and on that part of $\partial\omega_\ver$ where $\psi_\ver$ is zero adjacent to the boundary of $\Omega$.

For a node $\ver\in\VV_\coarse$ and a vertex $\vertt\in\VV_\coarse^\ver$, we also define some spaces on the small patches $\omega_\vertt$. In particular, we let
\begin{align}
L_*^2(\omega_\vertt)\eq
\begin{cases}\label{eq:L2 star}
\set{v\in L^2(\omega_\vertt)}{\dual{v}{1}_{\omega_\vertt}=0}\quad&\text{if } \psi_\ver\psi_\vertt\in H_0^1(\Omega),\\
L^2(\omega_\vertt)\quad&\text{else},
\end{cases}
\end{align}
where $\psi_\vertt$, being defined as function on $\omega_{\ver}$, is identified with its extension by zero onto $\Omega$ (which is in general not in $H^1(\Omega)$, cf. Figure~\ref{fig:partition} (c)).
We also define, as in~\eqref{eq:H1_a} and~\eqref{eq:H0 div for a},
\begin{align}\label{eq:H1_b}
H_*^1(\omega_\vertt)\eq
\begin{cases}
\set{v\in H^1(\omega_{\vertt})}{\dual{v}{1}_{\omega_{\vertt}}=0}\quad&\text{if }\psi_\ver\psi_\vertt\in H_0^1(\Omega),\\
\set{v\in H^1(\omega_{\vertt})}{v=0\text{ on }
\partial\omega_{\vertt}\setminus(\psi_\ver\psi_\vertt)^{-1}(\{0\})}
\quad&\text{else}
\end{cases}
\end{align}
and
\begin{align}\label{eq:H0 div for b}
\H_0(\div,\omega_\vertt)\eq
\begin{cases}
\set{\bv\in \H(\div,\omega_\vertt)}{\bv{\cdot}\n_{\omega_\vertt}=0 \text{ on }\partial\omega_\vertt}&\hspace{-2mm}\text{if }\psi_\ver\psi_\vertt\in H_0^1(\Omega),\\
\set{\bv\in \H(\div,\omega_\vertt)}{\bv{\cdot}\n_{\omega_\vertt}=0\text{ on } \partial\omega_\vertt\\ \hspace{5cm} \cap(\psi_\ver\psi_\vertt)^{-1}(\{0\})}&\hspace{-2mm}\text{else}.\\
\end{cases}
\end{align}
Note that these spaces actually depend on both $\ver$ and $\vertt$, where $\ver$ is omitted in our notation.

Finally, for $\ver\in\VV_\coarse$, $\vertt\in\VV_\coarse^\ver$, and $\omega\in\{\omega_\ver,\omega_\vertt\}$, define the Poincar\'e--Friedrichs constant as the minimal constant $C_{\rm PF}(\omega)>0$ such that
\begin{align}\label{eq_PF}
&\norm{v}{\omega}\le  \diam(\omega) C_{\rm PF}(\omega)\norm{\nabla v}{\omega} \quad \text{for all }v\in H_*^1(\omega).
\end{align}
Note that $C_{\rm PF}(\omega)$ only depends on the shape of $\omega$ and $\partial\omega$ where respectively $\psi_\ver$ or $\psi_\ver\psi_\vertt$ is nonzero; if $\psi_\ver\in H_0^1(\Omega)$ or $\psi_\ver\psi_\vertt\in H_0^1(\Omega)$ and for convex $\omega$, there in particular holds $C_{\rm PF}(\omega) \leq 1/\pi$ (``interior'' cases).
In the other, ``boundary'', cases, $C_{\rm PF}(\omega) \leq 1$ when there exists a unit vector $\bm{m}$ such that the straight semi-line of direction $\bm{m}$ originating at (almost) any point in $\omega$ hits the boundary $\partial\omega$ there where respectively $\psi_\ver$ or $\psi_\ver\psi_\vertt$ is nonzero, cf., e.g., \cite{Voh_Poinc_disc_05, vv12} and the references therein.

%%%%%%%%%%%%%%%%%%%%%%%%%%%%%%%%%%%%%%%%%%%%%%%%%%%%%%%%%%%%%%%%%%%%%%%%%%%%%%%%%%%%%
\subsection{Patchwise discrete subspaces}\label{sec:discrete spaces}
%%%%%%%%%%%%%%%%%%%%%%%%%%%%%%%%%%%%%%%%%%%%%%%%%%%%%%%%%%%%%%%%%%%%%%%%%%%%%%%%%%%%%
For a node $\ver\in\VV_\coarse$, let us define the $H_*^1(\omega_\ver)$-conforming subspace of the mapped piecewise multilinear functions $\mathbb{Q}^{\bm{1}}(\TT_\ver)$ from~\eqref{eq_PTa} as
\begin{align}\label{eq_Vha}
V_\coarse^\ver\eq \mathbb{Q}^{\bm{1}}(\TT_\ver)\cap H_*^1(\omega_\ver).
\end{align}

Define the vector-valued contravariant Piola transform
\begin{align}\label{eq:piola1}
 \bPhi(\cdot)\eq \big(\det (D\F)^{-1}(D\F)(\cdot)\big)\circ\F^{-1}
\end{align}
and the scalar Piola transform by
\begin{align}\label{eq:piola2}
 \widetilde\Phi(\cdot)\eq \big(\det (D\F)^{-1}(\cdot)\big)\circ\F^{-1},
\end{align}
which satisfy the identity
\begin{align}\label{eq:Piola_eq}
\widetilde\Phi\big(\nabla{\cdot}(\cdot)\big) = \Dv\bPhi(\cdot),
\end{align}
see, e.g., \cite[Chapter~9]{eg21}.

For a node $\ver\in\VV_\coarse$ and a vertex $\vertt\in\VV_\coarse^\ver$, define the meshes of the small patches as $\TT_\vertt\eq\set{T\in\TT_\ver}{T\subseteq\overline\omega_\vertt}$, $\widehat\TT_\vertt\eq\set{\F^{-1}(T)}{T\in\TT_\vertt}$; note that the elements $\widehat T\in\widehat\TT_\vertt$ are rectangles for $d=2$ and rectangular cuboids for $d=3$. Let $\mathbb{Q}^{\widetilde\p}(\widehat \TT_\vertt)$ be the space of all $\widehat\TT_\vertt$-piecewise polynomials of some fixed degree $\widetilde\p=(\widetilde p,\dots,\widetilde p)$ in each coordinate, $\widetilde p\ge0$, and let $L_*^2(\widehat\omega_\vertt)$ be defined as in~\eqref{eq:L2 star} with $\omega_\vertt$ replaced by $\widehat\omega_\vertt$. We then define the local spaces
\begin{subequations}\label{eq_loc_Qab}\begin{align}
\widehat Q_\coarse^{\ver,\vertt} & \eq \mathbb{Q}^{\widetilde\p}(\widehat\TT_\vertt)\cap L_*^2(\widehat\omega_\vertt), \label{eq_wQab} \\
Q_\coarse^{\ver,\vertt} & \eq \set{\widetilde \Phi(\widehat q_\coarse)}{\widehat q_\coarse \in \widehat Q_\coarse^{\ver,\vertt}} = \widetilde \Phi(\widehat Q_\coarse^{\ver,\vertt}). \label{eq_Qab}
\end{align}\end{subequations}
Note that since
\[
\widetilde \Phi(L_*^2(\widehat\omega_\vertt))=L^2_*(\omega_\vertt),
\]
$Q_\coarse^{\ver,\vertt}$ is contained in $L_*^2(\omega_\vertt)$, i.e.,
\begin{align}\label{eq:Q subset}
    Q_\coarse^{\ver,\vertt} \subset L_*^2(\omega_\vertt).
\end{align}

The mean-value constraint in~\eqref{eq:L2 star} (when $\psi_\ver\psi_\vertt\in H_0^1(\Omega)$)
%, though,
makes $Q_\coarse^{\ver,\vertt}$ a constrained subspace of mapped piecewise polynomials from $\mathbb{Q}^{\widetilde\p}(\widehat\TT_\vertt)$ scaled by the factor $\det (D\F)^{-1}$.
With the set $\widehat Q_\coarse:=\mathbb{Q}^{\widetilde\p}(\widehat\TT_\coarse)$ of all $\widehat\TT_\coarse$-piecewise polynomials of degree $\widetilde\p$, we also define the global (unconstrained) space via the mapping $\widetilde\Phi$ as
\begin{align*}
Q_\coarse\eq
\set{\widetilde \Phi(\widehat q_\coarse)}{\widehat q_\coarse\in \widehat Q_\coarse} = \widetilde \Phi(\widehat Q_\coarse).
\end{align*}

Let
\begin{align*}
\hspace{-3mm} \RT^{\widetilde\p}(\widehat\TT_\vertt)\eq\begin{cases}
\mathbb{Q}^{\widetilde\p+(1,0)}(\widehat\TT_\vertt)\times \mathbb{Q}^{\widetilde\p+(0,1)}(\widehat\TT_\vertt)\quad&\text{ if } d=2,
 \\
 \mathbb{Q}^{\widetilde\p+(1,0,0)}(\widehat\TT_\vertt)\times \mathbb{Q}^{\widetilde\p+(0,1,0)}(\widehat\TT_\vertt)\times  \mathbb{Q}^{\widetilde\p+(0,0,1)}(\widehat\TT_\vertt)&\text{ if } d=3
 \end{cases}
\end{align*}
be the usual broken (elementwise) Raviart--Thomas space on the rectangular/rectangular cuboid mesh $\widehat\TT_\vertt$, see, \eg, \cite[Section~2.4.1]{Bof_Brez_For_MFEs_13}. We then set
\begin{subequations} \label{eq_Vhab_def} \begin{align}
    \widehat\V_\coarse^{\ver,\vertt} & \eq \RT^{\widetilde\p}(\widehat\TT_\vertt) \cap \H_0(\div,\widehat\omega_\vertt), \label{eq_w_Vhab} \\
    \V_\coarse^{\ver,\vertt} & \eq \set{\bPhi(\widehat\bv_\coarse)}{\widehat\bv_\coarse\in \widehat\V_\coarse^{\ver,\vertt}} = \bPhi(\widehat\V_\coarse^{\ver,\vertt}). \label{eq_Vhab}
\end{align}\end{subequations}
Since
\[
    \H_0(\div,\omega_\vertt) = \bPhi( \H_0(\div,\widehat\omega_\vertt)),
\]
we also have

\begin{align}\label{eq:V subset}
    \V_\coarse^{\ver,\vertt} = \RT^{\widetilde\p}(\TT_\vertt)\cap \H_0(\div,\omega_\vertt),
\end{align}
where $\RT^{\widetilde\p}(\TT_\vertt)$ is the space $\RT^{\widetilde\p}(\widehat\TT_\vertt)$ mapped by the Piola transform
\begin{align}\label{eq_RT_map}
\RT^{\widetilde\p}(\TT_\vertt)\eq\set{\bPhi(\widehat\bv_\coarse)}{\widehat\bv_\coarse\in \RT^{\widetilde\p}(\widehat\TT_\vertt)} =\bPhi(\RT^{\widetilde\p}(\widehat\TT_\vertt)).
\end{align}
Crucially, by construction, see~\cite[Section~2.4.1]{Bof_Brez_For_MFEs_13},
\begin{subequations} \label{eq:divs}
\begin{align} \label{eq:div_wV}
    \Dv \widehat \V_\coarse^{\ver,\vertt} =  \widehat Q_\coarse^{\ver,\vertt},
\end{align}
whereas by the Piola transform identity~\eqref{eq:Piola_eq} and definition~\eqref{eq_Qab}, one also has
\begin{align} \label{eq:div_V}
    \Dv \V_\coarse^{\ver,\vertt} =  Q_\coarse^{\ver,\vertt}.
\end{align}
\end{subequations}

\begin{remark}\label{rem:align}
While it is in principle not relevant for the quality of our {\sl a posteriori} estimator whether the knot lines corresponding to $V_\coarse$ are aligned with those of a NURBS parametrization $\F$ (see Remark~\ref{rem:parametrization}), this is essential for $Q_\coarse^{\ver,\vertt}$ and $\RT^{\widetilde\p}(\TT_\vertt)$.
Indeed, the latter two spaces must exhibit good approximation properties to obtain small oscillation terms~\eqref{eq_osc} and~\eqref{eq:oscillations_glob}.
On the other hand, as $Q_\coarse^{\ver,\vertt}$ and $\RT^{\widetilde\p}(\TT_\vertt)$ are discontinuous (piecewise with respect to the mesh $\TT_\vertt$), they do not see/need the regularity of $\F$ over the knot lines. 
We also refer to Remark~\ref{rem:reliable oscillations} and Remark~\ref{rem:efficient oscillations} for conditions under which the oscillations even vanish.
\end{remark}

%%%%%%%%%%%%%%%%%%%%%%%%%%%%%%%%%%%%%%%%%%%%%%%%%%%%%%%%%%%%%%%%%%%%%%%%%%%%%%%%%%%%%
%%%%%%%%%%%%%%%%%%%%%%%%%%%%%%%%%%%%%%%%%%%%%%%%%%%%%%%%%%%%%%%%%%%%%%%%%%%%%%%%%%%%%
\section{Abstract assumptions}\label{sec:assumptions}
%%%%%%%%%%%%%%%%%%%%%%%%%%%%%%%%%%%%%%%%%%%%%%%%%%%%%%%%%%%%%%%%%%%%%%%%%%%%%%%%%%%%%
%%%%%%%%%%%%%%%%%%%%%%%%%%%%%%%%%%%%%%%%%%%%%%%%%%%%%%%%%%%%%%%%%%%%%%%%%%%%%%%%%%%%%

In this section, we attempt to describe as clearly as possible the underlying principles of a posteriori error analysis by equilibrated fluxes in the present context. For this purpose, we identify six abstract assumptions under which our subsequent analysis can be carried out. We then immediately verify these assumptions in the particular context of Sections~\ref{sec:setting}--\ref{sec:partitions}.

As for the general setting, we merely need to assume:

\begin{assumption}
\label{ass:setting}
$\widehat\Omega$ is an open bounded connected Lipschitz domain in $\R^d$, $\F$ is a bi-Lipschitz mapping, $\Omega \eq \F(\widehat\Omega)$, and $V_\coarse$ is an arbitrary subspace of $H_0^1(\Omega)$.
\end{assumption}

As for the partitions of unity, the minimalist assumptions are (note that we do not require that the partitions are non-negative, i.e., $\psi_\ver,\psi_\vertt\ge0$):

\begin{assumption}
\label{ass:partition of unity}
There is a finite index set $\VV_\coarse$ and functions $\psi_\ver$ such that
\begin{subequations}
\begin{align}
\set{\psi_\ver}{\ver\in\VV_\coarse}\subset W^{1,\infty}(\Omega)
\end{align}
form a partition of unity over $\Omega$ in the sense that
\begin{align}\label{eq:partition of unity}
\sum_{\ver\in\VV_\coarse}\psi_\ver=1 \quad\text{in }\Omega.
\end{align}
The interior $\omega_\ver$ of the support of any $\psi_\ver$ is a connected Lipschitz domain with $|\omega_\ver|>0$.
Moreover,
\begin{align}\label{eq:beta in H01}
\set{\psi_\ver}{\ver\in\VV_\coarse}\cap H_0^1(\Omega)\subset V_\coarse.
\end{align}
\end{subequations}
\end{assumption}

\begin{assumption}
\label{ass:partition of unity2}
For any node $\ver\in\VV_\coarse$, there is a finite set of vertices $\VV_\coarse^\ver$ and functions $\psi_\vertt$ such that
\begin{subequations}
\begin{align}\label{eq_psi_b}
\set{\psi_\vertt}{\vertt\in\VV_\coarse^\ver}\subset W^{1,\infty}(\omega_\ver)
\end{align}
form a partition of unity over $\omega_\ver$ in that
\begin{align}\label{eq:partition of unity2}
\sum_{\vertt\in\VV_\coarse^\ver}\psi_\vertt=1\quad\text{in }\omega_\ver.
\end{align}
The interior $\omega_\vertt$ of the support of any $\psi_\vertt$ is a connected Lipschitz domain with $|\omega_\vertt|>0$.
Moreover, for $H_*^1(\omega_\ver)$ given by~\eqref{eq:H1_a}, all $\psi_\vertt$ such that $\psi_\ver\psi_\vertt\in H_0^1(\Omega)$, where $\psi_{\vertt}$ is identified with its extension by zero onto $\Omega$,  are contained up to additive constants in a finite-dimensional subspace $V_h^\ver\subset H_*^1(\omega_\ver)$, i.e.,
\begin{align}\label{eq:b in H01}
\set{\psi_\vertt}{\vertt\in\VV_\coarse^\ver\wedge \psi_\ver\psi_\vertt\in H_0^1(\Omega)}\subset
\begin{cases}
\set{v_\coarse+c}{v_\coarse \in V_\coarse^\ver, c\in\R}\quad &\text{if }\psi_\ver\in H_0^1(\Omega)\\
V_\coarse^\ver\quad&\text{else}.
\end{cases}
\end{align}
\end{subequations}
\end{assumption}

\begin{remark}
While in theory, there is almost no connection between the functions $\psi_\ver$ and the functions $\psi_\vertt$, in practice, $\psi_\ver$ are mapped piecewise polynomials with high smoothness on some local mesh and the corresponding $\psi_\vertt$ are mapped piecewise polynomials with low smoothness on the same mesh, see Sections~\ref{sec:inex_iga} and~\ref{sec:hierarchical partitions}.
We stress that the choice of these partitions of unity determines the quality of the equilibrated flux estimator through the oscillations terms in Propositions~\ref{prop:reliable} and~\ref{prop:efficiency}.
\end{remark}

Next, we recall the space $H_*^1(\omega_\vertt)$ from~\eqref{eq:H1_b}, the Poincar\'e--Friedrichs inequality~\eqref{eq_PF}, and assume the following set of estimates:
\begin{assumption}\label{ass:CPF}
There exist generic positive constants $C_1,\dots,C_6>0$ such that for all $\ver\in\VV_\coarse$ and $\vertt\in\VV_\coarse^\ver$, it holds that
\begin{subequations}
\begin{align}
\label{eq:beta bounded}
& \norm{\psi_\ver}{\infty,\omega_\ver}\le C_1,
\\
\label{eq:CPFcont beta}
&\norm{\nabla \psi_\ver}{\infty,\omega_\ver} \diam(\omega_\ver) C_{\rm PF}(\omega_\ver)\le C_2,
\\
\label{eq:CPFcont beta b}
&\norm{\nabla \psi_\ver}{\infty,\omega_\vertt}\diam(\omega_\vertt)C_{\rm PF}(\omega_\vertt)\le C_3,
\\
\label{eq:number of bs}
&\sup_{\bx \in \omega_\ver} \#\set{\vertt'\in\VV_\coarse^\ver}{\bx \in \omega_{\vertt'}}\le C_4,
\\
\label{eq:b bounded}
& \norm{\psi_\vertt}{\infty,\omega_\vertt}\le C_5,
\\
\label{eq:CPFcont b}
&\norm{\nabla \psi_\vertt}{\infty,\omega_\vertt} \diam(\omega_\vertt) C_{\rm PF}(\omega_\vertt) \le C_6.
\end{align}
\end{subequations}
\end{assumption}

Recall the space $\H_0(\div,\omega_\vertt)$ from~\eqref{eq:H0 div for b}, the subspace $L_*^2(\omega_\vertt)$ of $L^2(\omega_\vertt)$ containing functions with mean value zero if $\psi_\ver\psi_\vertt\in H_0^1(\Omega)$ from~\eqref{eq:L2 star}, and the similar space $L_*^2(\widehat\omega_\vertt)$ in the parameter domain. Recall also the contravariant Piola transform $\bPhi$ from~\eqref{eq:piola1}.
For local flux equilibration, we will rely on discrete spaces $Q_\coarse^{\ver,\vertt}$ and $\V_\coarse^{\ver,\vertt}$ satisfying the following:

\begin{assumption}\label{ass:QV subsets}
For any node $\ver\in\VV_\coarse$ and any vertex $\vertt\in\VV_\coarse^\ver$, there are finite-dimensional subspaces
\begin{subequations}
\begin{align}\label{eq:QV subsets}
 Q_\coarse^{\ver,\vertt}\subset L_*^2(\omega_\vertt) \quad\text{and}\quad\V_\coarse^{\ver,\vertt} \subset \H_0(\div,\omega_\vertt)
\end{align}
satisfying the compatibility condition
\begin{align}\label{eq:QV compatibility}
\begin{split}
 \Dv \V_\coarse^{\ver,\vertt} & =  Q_\coarse^{\ver,\vertt}.
%&\text{with}\\
% \widehat Q_\coarse^{\ver,\vertt} \eq \set{q_\coarse \circ \F}{q_\coarse\in Q_\coarse^{\ver,\vertt}} & \text{ and }
% \widehat\V_\coarse^{\ver,\vertt} \eq \set{\bPhi^{-1}(\bv_\coarse)}{\bv_\coarse\in \V_\coarse^{\ver,\vertt}}.
\end{split}
\end{align}
Moreover, we suppose the existence of a global space $Q_\coarse\subset L^2(\Omega)$
such that
\begin{align}\label{eq:Q inclusion}
Q_\coarse|_{\omega_\vertt} \subseteq Q_{\coarse,\mathrm{c}}^{\ver,\vertt}\eq
\begin{cases}
\set{q_\coarse+\widetilde\Phi(c)}{q_\coarse\in Q_\coarse^{\ver,\vertt},c\in\R}\quad&\text{if }\psi_\ver\psi_\vertt\in H_0^1(\Omega),\\
Q_\coarse^{\ver,\vertt}\quad&\text{else.}
\end{cases}
\end{align}
\end{subequations}
\end{assumption}

\begin{remark}\label{rem_Q_spaces}
Using the Piola transforms $\bPhi$ and $\widetilde\Phi$ from~\eqref{eq:piola1}--\eqref{eq:piola2}, let
\begin{align}\label{eq_Q_h_ab_mean}
\widehat Q_\coarse^{\ver,\vertt} \eq \widetilde\Phi^{-1}(Q_\coarse^{\ver,\vertt}) \subset L_*^2(\widehat\omega_\vertt) =\widetilde\Phi^{-1}(L^2_*(\omega_\vertt))
\end{align}
and
\begin{align}
 \widehat\V_\coarse^{\ver,\vertt} \eq \bPhi^{-1}(\V_\coarse^{\ver,\vertt}) \subset \H_0(\div,\widehat\omega_\vertt)=\bPhi^{-1}( \H_0(\div,\omega_\vertt)).
\end{align}
From~\eqref{eq:QV compatibility} and~\eqref{eq:Piola_eq}, we in particular have
\[
    \Dv \widehat \V_\coarse^{\ver,\vertt} =  \widehat Q_\coarse^{\ver,\vertt},
\]as in Section~\ref{sec:discrete spaces}.
%Thus, for any $\widehat\bv_\coarse \in \widehat\V_\coarse^{\ver,\vertt}$, there holds $(\nabla {\cdot} \widehat\bv_\coarse, 1)_{\widehat\omega_\vertt} = 0$ if $\psi_\ver\psi_\vertt\in H_0^1(\Omega)$, so that~\eqref{eq:QV compatibility} implies that functions in $\widehat Q_\coarse^{\ver,\vertt}$ have mean value zero when $\psi_\ver \psi_\vertt\in H_0^1(\Omega)$, \ie,
%In contrast, $Q_\coarse^{\ver,\vertt}=\set{\widehat q_\coarse\circ\F^{-1}}{\widehat q_\coarse\in \widehat Q_\coarse^{\ver,\vertt}}$ does not satisfy the mean value condition if $\psi_\ver\psi_\vertt\in H_0^1(\Omega)$ and thus is only a subspace of $L^2(\omega_\vertt)$ and not of $L_*^2(\omega_\vertt)$ (unless the mapping $\F$ is affine).
%The space $Q_\coarse^{\ver,\vertt}$, though, is a constrained space, as it maps the constraint~\eqref{eq_Q_h_ab_mean}.
%The space $Q_{\coarse,\mathrm{c}}^{\ver,\vertt}$ from~\eqref{eq:Q inclusion} then contains no constraint, independently of whether $\psi_\ver\psi_\vertt\in H_0^1(\Omega)$ or not.
The role of the global space $Q_\coarse\subset L^2(\Omega)$ will be prominent below: please note that it is related to neither the node $\ver$, nor to the vertex $\vertt$; this forces the local spaces $Q_{\coarse,\mathrm{c}}^{\ver,\vertt}$ to contain patch-independent ``base-blocks''.
In practice, $Q_\coarse|_{\omega_\vertt} = Q_{\coarse,\mathrm{c}}^{\ver,\vertt}$ only for uniform mesh refinement but not, for example, in the setting of Figure~\ref{fig:partition}, where a strict inclusion holds, since the local spaces $Q_{\coarse,\mathrm{c}}^{\ver,\vertt}$ stem from the local uniform fine meshes $\TT_\ver$ (in grey in Figure~\ref{fig:partition}), whereas $Q_\coarse$ is related to the (mapped) hierarchical mesh $\TT_\coarse$ (in black in Figure~\ref{fig:partition}).
\end{remark}

Finally, we will essentially employ the following finite dimension to infinite dimension extension property.

\begin{assumption}\label{ass:min RT to cont}
There exist a generic constant $C_{\rm st}\geq 1$, as well as a superspace $\RT_\coarse^{\ver,\vertt} \subset [L^2(\omega_\vertt)]^d$ verifying $\V_\coarse^{\ver,\vertt} =\RT_\coarse^{\ver,\vertt} \cap \H_0(\div,\omega_\vertt)$
for all nodes $\ver\in\VV_\coarse$ and vertices $\vertt\in\VV_\coarse^\ver$, such that
\begin{align}\label{eq:min RT to cont}
\min_{\substack{\bv_\coarse\in\V_\coarse^{\ver,\vertt}\\
 \Dv\bv_\coarse=g_\coarse}}
\norm{\bv_\coarse + \ttau_\coarse}{\omega_\vertt}
\le
C_{\rm st} \min_{\substack{\bv \in \H_0(\div,\omega_\vertt)\\
\Dv\bv=g_\coarse}}
\norm{\bv+ \ttau_\coarse}{\omega_\vertt}
\end{align}
for all $g_\coarse\in \Dv\V_\coarse^{\ver,\vertt} =Q_\coarse^{\ver,\vertt}$ and all $\ttau_\coarse \in \RT_\coarse^{\ver,\vertt}$.
\end{assumption}

We now verify that the above assumptions are satisfied in our IGA context:

\begin{proposition}\label{thm_ass} Assumptions~\ref{ass:setting}--\ref{ass:min RT to cont} are satisfied in the context of Sections~\ref{sec:setting}--\ref{sec:partitions} with the choice $\RT_\coarse^{\ver,\vertt} = \RT^{\widetilde\p}(\TT_\vertt)$.
 In particular, the constants in Assumption~\ref{ass:CPF} can be taken as $C_1=1$, $C_4=2^d$, $C_5=1$ and such that $C_2$, $C_3$, $C_6$, and $C_{\rm st}$ only depend on the shapes in $\widehat \TT_{\uni{0}}$ as well as the mapping $\F$ via $\max\{\norm{D\F}{\infty,\widehat\Omega},\norm{(D\F)^{-1}}{\infty,\widehat\Omega}\}$; $C_2$ and $C_3$ additionally depend on the supplementary polynomial degree $\overline p$ from~\eqref{eq_ov_p}--\eqref{eq_ov_p_2}.
\end{proposition}

\begin{proof} Assumptions~\ref{ass:setting}--\ref{ass:partition of unity2} are immediately satisfied with the choices made in Sections~\ref{sec:setting}--\ref{sec:partitions}, in particular fixing the space $V_\coarse^\ver$ following~\eqref{eq_Vha}.

We now verify Assumption~\ref{ass:CPF}. Thanks to~\eqref{eq_scaling} and~\eqref{eq_coefs}, inequality~\eqref{eq:beta bounded} is satisfied with the constant $C_1=1$, whereas~\eqref{eq:number of bs} and~\eqref{eq:b bounded} hold easily with respectively $C_4=2^d$ and $C_5=1$. Next, \eqref{eq:CPFcont beta}, \eqref{eq:CPFcont beta b}, and~\eqref{eq:CPFcont b} follow from the facts that $\norm{\nabla\psi_\ver}{\infty,\omega_\ver} \lesssim \overline p \,\diam(\omega_\ver)^{-1}$, $\diam(\omega_\vertt) \leq \diam(\omega_\ver)$, and $\norm{\nabla\psi_\vertt}{\infty,\omega_\vertt} \lesssim \diam(\omega_\vertt)^{-1}$, see, e.g., \cite[Equation~(2.7)]{bbsv14}, where we also use that $C_{\rm PF}(\omega_\ver), C_{\rm PF}(\omega_\vertt)\lesssim 1$ (see, e.g., \cite[Corollary~2.2]{vv12}).
Here, $A\lesssim B$ means that $A\le C B$ for a hidden constant $C>0$ depending only on the shapes of the elements in $\widehat\TT_{\uni{0}}$ as well as on $\max\{\norm{D\F}{\infty,\widehat\Omega},\norm{(D\F)^{-1}}{\infty,\widehat\Omega}\}$.

We now turn to Assumption~\ref{ass:QV subsets}. Properties~\eqref{eq:QV subsets} and~\eqref{eq:Q inclusion} are trivially satisfied for the spaces defined in Section~\ref{sec:discrete spaces}, see in particular~\eqref{eq:Q subset} and~\eqref{eq:V subset}, whereas the compatibility condition~\eqref{eq:QV compatibility} follows from the construction requirement of the Raviart--Thomas space in the parameter domain~\eqref{eq:div_wV} by~\eqref{eq:Piola_eq} and definition~\eqref{eq_Qab}, see~\eqref{eq:div_V}.

We finally address Assumption~\ref{ass:min RT to cont}. For the spaces defined in Section~\ref{sec:discrete spaces}, where we take $\RT_\coarse^{\ver,\vertt} = \RT^{\widetilde\p}(\TT_\vertt)$, this result is proved in the parameter domain
in~\cite[Theorems~5 and~7]{Brae_Pill_Sch_p_rob_09} in two space dimensions; in three space dimensions, one needs to rely instead on~\cite{Cost_Daug_Demk_ext_08}, cf.~\cite[Theorem~2.5 and Corollary~3.3]{ev20} building on~\cite{Cost_McInt_Bog_Poinc_10, Demk_Gop_Sch_ext_III_12}.
We give a proof in the physical domain in Appendix~\ref{sec:extension}, where the resulting constant $C_{\rm st}$ depends only on the shapes of the elements in $\widehat\TT_{\uni{0}}$ and $\max\{\norm{D\F}{\infty,\widehat\Omega},\norm{(D\F)^{-1}}{\infty,\widehat\Omega}\}$.
We stress that $C_{\rm st}$ is independent of the polynomial degree $\widetilde p$.
\end{proof}

\begin{remark}\label{rem_p_rob} Recall that the supplementary polynomial degree $\overline p$ from Section~\ref{sec:hierarchical partitions} depends on the considered smoothness but not necessarily on the polynomial degree $p$. In this sense, $C_2, C_3$ in Proposition~\ref{thm_ass} are independent of the polynomial degree $p$. We recall in particular that $\overline p = 1$ can be taken for $C^{0}$ splines, and in general $\overline p = k+1$ for $C^k$ splines. We admit, though, that $\overline p = p$ for $C^{p-1}$ splines, see the discussion in Section~\ref{sec:hierarchical partitions}. \end{remark}

\begin{remark}\label{rem:naive_approach}
We mention that the abstract Assumptions~\ref{ass:setting}--\ref{ass:QV subsets} are also satisfied in the setting of Sections~\ref{sec:standard_equi} and~\ref{sec:standard_iga}.
Indeed, one can simply choose one function $\psi_\vertt:=1$ on each patch $\omega_\ver$ so that $\omega_\vertt = \omega_\ver$, $V_\coarse^\ver := \{0\}$, $\V_\coarse^{\ver,\vertt}:= \RT^{\widetilde \p}(\TT_\ver) \cap \H_0(\div,\omega_\ver)$ with $\RT^{\widetilde \p}(\TT_\ver)$ from \eqref{eq:RT1}, $Q_\coarse^{\ver,\vertt} := \mathbb{Q}^{\widetilde \p}(\TT_\ver) \cap L^2_*(\omega_\ver) $ and $Q_\coarse := \mathbb{Q}^{\widetilde \p}(\TT_\coarse)$. For Assumption~\ref{ass:min RT to cont}, in turn, to our knowledge, there only is a rigorous proof in the setting of Section~\ref{sec:standard_equi} with ``small'' vertex patches $\oma$, and not in Section~\ref{sec:standard_iga} with ``large'' vertex patches $\oma$.
\end{remark}

%%%%%%%%%%%%%%%%%%%%%%%%%%%%%%%%%%%%%%%%%%%%%%%%%%%%%%%%%%%%%%%%%%%%%%%%%%%%%%%%%%%%%
%%%%%%%%%%%%%%%%%%%%%%%%%%%%%%%%%%%%%%%%%%%%%%%%%%%%%%%%%%%%%%%%%%%%%%%%%%%%%%%%%%%%%
\section{Inexpensive equilibration}\label{sec:equilibration}
%%%%%%%%%%%%%%%%%%%%%%%%%%%%%%%%%%%%%%%%%%%%%%%%%%%%%%%%%%%%%%%%%%%%%%%%%%%%%%%%%%%%%
%%%%%%%%%%%%%%%%%%%%%%%%%%%%%%%%%%%%%%%%%%%%%%%%%%%%%%%%%%%%%%%%%%%%%%%%%%%%%%%%%%%%%
Let the abstract assumptions of Section~\ref{sec:assumptions} be satisfied, let $u$ solve~\eqref{eq:variational}, and let $u_\coarse$ solve~\eqref{eq:Galerkin}.
The partition of unity functions $\psi_\ver$ from Assumption~\ref{ass:partition of unity}, in our setting the hierarchical B-splines $\psi_\ver$ of polynomial degree $\overline \p=(\overline p,\dots,\overline p)$ and multiplicity $\overline m$ from~\eqref{eq_psi_a_def}, lead to
\begin{align}\label{eq:Galerkin for beta}
\dual{\nabla u_\coarse}{\nabla \psi_\ver}_{\omega_\ver} =\dual{f}{\psi_\ver}_{\omega_\ver}\quad\text{for all }\ver\in \VV_\coarse \text{ with }\psi_\ver\in H_0^1(\Omega),
\end{align}
which is an immediate consequence of~\eqref{eq:Galerkin} and~\eqref{eq:beta in H01}. Thus, in view of~\eqref{eq:variational},
\begin{align}\label{eq:GO}
\dual{\nabla (u-u_\coarse)}{\nabla \psi_\ver}_{\omega_\ver} =0\quad\text{for all }\ver\in \VV_\coarse \text{ with }\psi_\ver\in H_0^1(\Omega).
\end{align}
This orthogonality is sufficient to localize the error $\norm{\nabla (u-u_\coarse)}{\Omega}$ (or, equivalently, the dual norm of the residual) over the large patches $\omega_\ver$, see~\cite{Blech_Mal_Voh_loc_res_NL_20} and the references therein, and then a flux equilibration can be easily devised. The issue, however, is that this leads to an expensive equilibration with {\em higher-order} (related to the polynomial degree $p$) mixed finite spaces on the {\em large patches} $\omega_\ver$, see Section~\ref{sec:standard_iga}.
Our goal below is to design a much {\em less expensive equilibration}, with some inexpensive (typically piecewise multilinear) solve on the large patches $\omega_\ver$ followed by higher-order (related to $p$) mixed finite element solves on the {\em small patches} $\omega_\vertt$ only, extending the construction from Section~\ref{sec:inex_iga} to the present general setting. In order to achieve it, we crucially rely on the discrete patchwise spaces from Section~\ref{sec:assumptions}/Section~\ref{sec:discrete spaces}.

%%%%%%%%%%%%%%%%%%%%%%%%%%%%%%%%%%%%%%%%%%%%%%%%%%%%%%%%%%%%%%%%%%%%%%%%%%%%%%%%%%%%%
\subsection{Inexpensive (lowest-order) residual lifting on the large patches $\omega_\ver$}
%%%%%%%%%%%%%%%%%%%%%%%%%%%%%%%%%%%%%%%%%%%%%%%%%%%%%%%%%%%%%%%%%%%%%%%%%%%%%%%%%%%%%

Our first step is to construct a discrete residual function $r_\coarse^\ver\in V_\coarse^\ver$, where $V_\coarse^\ver$ is the local space from Assumption~\ref{ass:partition of unity2}. Recall that in the context of hierarchical B-splines, $V_\coarse^\ver$ merely consists of mapped piecewise multilinear functions, see~\eqref{eq_PTa} and~\eqref{eq_Vha}.
\begin{definition}\label{def_ra}
For all nodes $\ver\in\VV_\coarse$, let $r_\coarse^\ver\in V_\coarse^\ver$ be such that
\begin{align}\label{eq:rh beta}
\dual{\nabla r_\coarse^\ver}{\nabla v_\coarse}_{\omega_\ver}
=\dual{f}{v_\coarse \psi_\ver}_{\omega_\ver}- \dual{\nabla u_\coarse}{\nabla(v_\coarse \psi_\ver)}_{\omega_\ver}
\quad\text{for all }v_\coarse\in V_\coarse^\ver.
\end{align}
\end{definition}

Figure~\ref{fig:procedure}, steps 1)--3), gives an illustration for
\eqref{eq:rh beta} in the spline setting, which is an {\em inexpensive} scalar-valued lowest-order local problem on the larger patches $\omega_\ver$, lifting the $\psi_\ver$-weighted residual, cf.~\cite{Cars_Funk_full_rel_FEM_00, Brae_Pill_Sch_p_rob_09, ev15, Blech_Mal_Voh_loc_res_NL_20}. We note that~\eqref{eq:rh beta} is a finite-dimensional version of the problem: find $r^\ver\in H_*^1(\omega_\ver)$ such that
\begin{align}\label{eq:r beta}
\dual{\nabla r^\ver}{\nabla v}_{\omega_\ver}
=\dual{f}{v\psi_\ver}_{\omega_\ver}- \dual{\nabla u_\coarse}{\nabla(v\psi_\ver)}_{\omega_\ver}
\quad\text{for all }v\in H_*^1(\omega_\ver).
\end{align}
This can be equivalently written as: find $r^\ver\in H_*^1(\omega_\ver)$ such that
\begin{align}\label{eq:r beta_2}
\dual{\nabla r^\ver + \psi_\ver \nabla u_\coarse}{\nabla v}_{\omega_\ver}
=\dual{f\psi_\ver - \nabla u_\coarse {\cdot} \nabla \psi_\ver}{v}_{\omega_\ver}
\quad\text{for all }v\in H_*^1(\omega_\ver).
\end{align}
From~\eqref{eq:r beta_2}, we see that $-(\nabla r^\ver + \psi_\ver \nabla u_\coarse)$ lies in $\H_0(\div,\omega_\ver)$ with the divergence equal to $f \psi_\ver - \nabla u_\coarse {\cdot} \nabla \psi_\ver$.
It can be also characterized by
\begin{align}\label{eq:r beta_3}
\argmin{\substack{\bv \in \H_0(\div,\omega_\ver)\\
\Dv \bv = f\psi_\ver -\nabla u_\coarse{\cdot}\nabla \psi_\ver}}
\norm{\bv + \psi_\ver \nabla u_\coarse}{\omega_\ver},
\end{align}
which is the infinite-dimensional paradigm of~\eqref{eq:sigma_a}.

%%%%%%%%%%%%%%%%%%%%%%%%%%%%%%%%%%%%%%%%%%%%%%%%%%%%%%%%%%%%%%%%%%%%%%%%%%%%%%%%%%%%%
\subsection{Projection operators for general bi-Lipschitz mappings $\F$} \label{sec_proj}
%%%%%%%%%%%%%%%%%%%%%%%%%%%%%%%%%%%%%%%%%%%%%%%%%%%%%%%%%%%%%%%%%%%%%%%%%%%%%%%%%%%%%

In order to proceed, we need to define some projection operators. Recall the local space $Q_\coarse^{\ver,\vertt}$ from Assumption~\ref{ass:QV subsets} as well as $\widehat Q_\coarse^{\ver,\vertt}$ from~\eqref{eq_Q_h_ab_mean}.
%From assumptions~\eqref{eq:QV subsets} and~\eqref{eq:QV compatibility}, there holds
%\begin{align*}
%\Dv\V_\coarse^{\ver,\vertt}=\widetilde \Phi(\widehat Q_\coarse^{\ver,\vertt}) := \widetilde Q_\coarse^{\ver,\vertt} \subset L_*^2(\omega_\vertt).
%\end{align*}
For a given $g\in L^2(\omega_\vertt)$, we define $\Upsilon_{Q_\coarse^{\ver,\vertt}}g \in Q_\coarse^{\ver,\vertt}$ as the Petrov--Galerkin projection of $g$ into $Q_\coarse^{\ver,\vertt}$ with the test space $\widehat Q_\coarse^{\ver,\vertt}\circ\F^{-1} = \set{\widehat q_\coarse \circ \F^{-1}}{\widehat q_\coarse\in \widehat Q_\coarse^{\ver,\vertt}}$, i.e.,
\begin{align}\label{eq:projection_Q_a_b_char}
 \dual{\Upsilon_{Q_\coarse^{\ver,\vertt}}g}{q_\coarse}_{\omega_\vertt} = \dual{g}{q_\coarse}_{\omega_\vertt} \quad\text{for all }q_\coarse\in \widehat Q_\coarse^{\ver,\vertt}\circ\F^{-1}.
\end{align}
When the mapping $\F$ is affine, the scaling factor $\det (D\F)^{-1}$ in~\eqref{eq:piola2} is constant, so that the trial space $Q_\coarse^{\ver,\vertt}$ and the test space $\widehat Q_\coarse^{\ver,\vertt}\circ\F^{-1}$ become the same. In such a case, $\Upsilon_{Q_\coarse^{\ver,\vertt}}$ is simply the $L^2(\omega_\vertt)$-orthogonal (Galerkin) projection.
Let
\begin{align*}
    \widehat Q_{\coarse,\mathrm{c}}^{\ver,\vertt} & \eq \set{\widetilde \Phi^{-1}(q_\coarse)}{q_\coarse\in Q_{\coarse,\mathrm{c}}^{\ver,\vertt}} = \widetilde \Phi^{-1}(Q_{\coarse,\mathrm{c}}^{\ver,\vertt}), \\
    \widehat Q_\coarse & \eq \set{\widetilde \Phi^{-1}(q_\coarse)}{q_\coarse\in Q_\coarse}=\widetilde \Phi^{-1}(Q_\coarse).
\end{align*}
For $g\in L^2(\omega_\vertt)$ and $g\in L^2(\Omega)$, respectively,
we define
$\Upsilon_{Q_{\coarse,\mathrm{c}}^{\ver,\vertt}}g \in Q_{\coarse,\mathrm{c}}^{\ver,\vertt}$
and
$\Upsilon_{Q_{\coarse}}g \in Q_{\coarse}$
analogously, i.e.,
\begin{subequations}\label{eq:projections_Q}\begin{align}
 \dual{\Upsilon_{Q_{\coarse,\mathrm{c}}^{\ver,\vertt}}g}{q_\coarse}_{\omega_\vertt} & = \dual{g}{q_\coarse}_{\omega_\vertt} \quad\text{for all }q_\coarse\in \widehat Q_{\coarse,\mathrm{c}}^{\ver,\vertt}\circ\F^{-1}, \label{eq:projection_Q_a_b_c}\\
 \dual{\Upsilon_{Q_\coarse}g}{q_\coarse}_{\Omega} & = \dual{g}{q_\coarse}_{\Omega} \quad\text{for all }q_\coarse\in \widehat Q_\coarse\circ\F^{-1}. \label{eq:projection_Q}
\end{align} \end{subequations}

%%%%%%%%%%%%%%%%%%%%%%%%%%%%%%%%%%%%%%%%%%%%%%%%%%%%%%%%%%%%%%%%%%%%%%%%%%%%%%%%%%%%%
\subsection{Equilibrated flux on the small patches $\omega_\vertt$}
%%%%%%%%%%%%%%%%%%%%%%%%%%%%%%%%%%%%%%%%%%%%%%%%%%%%%%%%%%%%%%%%%%%%%%%%%%%%%%%%%%%%%

Let $r_\coarse^\ver$ be given by Definition~\ref{def_ra} and recall the projection operator $\Upsilon_{Q_\coarse^{\ver,\vertt}}$ from Section~\ref{sec_proj}. Then our equilibrated flux on the small patches $\omega_\vertt$ is given by:
\begin{definition}\label{def_flux_ab}
For all nodes $\ver\in\VV_\coarse$ and all vertices $\vertt\in\VV_\coarse^\ver$, let
\begin{align}\label{eq:sigma beta b}
\ssigma_\coarse^{\ver,\vertt}\eq\argmin{\substack{\bv_\coarse\in\V_\coarse^{\ver,\vertt}\\
\Dv\bv_\coarse=\Upsilon_{Q_\coarse^{\ver,\vertt}}(f\psi_\ver \psi_\vertt-\nabla u_\coarse{\cdot}\nabla(\psi_\ver \psi_\vertt) -\nabla r_\coarse^\ver{\cdot}\nabla \psi_\vertt)}}
\norm{\bv_\coarse+\psi_\vertt(\psi_\ver \nabla u_\coarse+\nabla r_\coarse^\ver)}{\omega_\vertt}.
\end{align}
\end{definition}
Figure~\ref{fig:procedure}, steps 4)--5), gives an illustration in the spline setting. From~\eqref{eq:r beta_2}, we know that $-(\nabla r^\ver + \psi_\ver \nabla u_\coarse)$ lies in $\H_0(\div,\omega_\ver)$ with the divergence equal to $f \psi_\ver - \nabla u_\coarse {\cdot} \nabla \psi_\ver$. Thus, its cut-off by $\psi_\vertt$, $-\psi_\vertt(\psi_\ver \nabla u_\coarse + \nabla r^\ver)$, lies in $\H_0(\div,\omega_\vertt)$ with the divergence equal to $f\psi_\ver \psi_\vertt-\nabla u_\coarse{\cdot}\nabla(\psi_\ver \psi_\vertt) -\nabla r^\ver{\cdot}\nabla \psi_\vertt$. Similarly to~\eqref{eq:r beta_3}, it can be characterized implicitly by
\begin{align} \label{eq:r beta_2_dual}
\argmin{\substack{\bv \in \H_0(\div,\omega_\vertt)\\
\Dv \bv = f\psi_\ver \psi_\vertt-\nabla u_\coarse{\cdot}\nabla(\psi_\ver \psi_\vertt) -\nabla r^\ver{\cdot}\nabla \psi_\vertt}}
\norm{\bv + \psi_\vertt ( \psi_\ver \nabla u_\coarse + \nabla r^\ver)}{\omega_\vertt}.
\end{align}
The flux $\ssigma_\coarse^{\ver,\vertt}$ from~\eqref{eq:sigma beta b} is then its discrete approximation.

The following lemma guarantees the existence and uniqueness of the minimizer of Definition~\ref{def_flux_ab}:

\begin{lemma}\label{lem:minimizer}
There exists a unique minimizer $\ssigma_\coarse^{\ver,\vertt}$ of~\eqref{eq:sigma beta b}.
\end{lemma}
\begin{proof}
Let us abbreviate
\begin{align}\label{eq_g_def}
    g \eq f\psi_\ver \psi_\vertt-\nabla u_\coarse{\cdot}\nabla(\psi_\ver \psi_\vertt) -\nabla r_\coarse^\ver{\cdot}\nabla \psi_\vertt
\end{align}
and
\[
\ttau\eq\psi_\vertt(\psi_\ver \nabla u_\coarse+\nabla r_\coarse^\ver).
\]
By definition of $\Upsilon_{Q_\coarse^{\ver,\vertt}}$, $\Upsilon_{Q_\coarse^{\ver,\vertt}}g \in Q_\coarse^{\ver,\vertt} =  \nabla{\cdot}\V_\coarse^{\ver,\vertt}$, so that the minimization set is nonempty, and existence and uniqueness follow by standard convexity arguments, see, \eg, \cite[Proposition~1.2]{Eke_Tem_con_anal_99}.
In more details, problem~\eqref{eq:sigma beta b} is equivalent to finding $\ssigma_\coarse^{\ver,\vertt}$ with $\nabla{\cdot}\ssigma_\coarse^{\ver,\vertt}=\Upsilon_{Q_\coarse^{\ver,\vertt}}g$ such that
\begin{align}\label{eq:computation of sigma}
\dual{\ssigma_\coarse^{\ver,\vertt}}{\bv_\coarse}_{\omega_\vertt}= -\dual{\ttau}{\bv_\coarse}_{\omega_\vertt}\quad\text{for all } \bv_\coarse\in\V_\coarse^{\ver,\vertt}\text{ with }\nabla{\cdot}\bv_\coarse=0,
\end{align}
which is a square linear system. Existence and uniqueness thus follow when $\ssigma_\coarse^{\ver,\vertt}=0$ for zero data. Let thus $g=0$ and $\ttau=\nv$. Since $\Upsilon_{Q_\coarse^{\ver,\vertt}}0=0$ follows from~\eqref{eq:projection_Q_a_b_char}, we can take $\bv_\coarse=\ssigma_\coarse^{\ver,\vertt}$ in~\eqref{eq:computation of sigma}, which implies  $\norm{\ssigma_\coarse^{\ver,\vertt}}{\omega_\vertt}=0$ and thus $\ssigma_\coarse^{\ver,\vertt}=0$.
\end{proof}

Recalling~\eqref{eq:projection_Q_a_b_char} and~\eqref{eq:projection_Q_a_b_c}, the following is an important extension of Lemma~\ref{lem:minimizer}:

\begin{lemma}\label{lem:easier projection}
The projection $\Upsilon_{Q_\coarse^{\ver,\vertt}}$ in~\eqref{eq:sigma beta b} can be replaced by $\Upsilon_{Q_{\coarse,\mathrm{c}}^{\ver,\vertt}}$, i.e.,
\begin{align}\label{eq:a_b_c}
(\Upsilon_{Q_\coarse^{\ver,\vertt}}-\Upsilon_{Q_{\coarse,\mathrm{c}}^{\ver,\vertt}})(f\psi_\ver \psi_\vertt-\nabla u_\coarse{\cdot}\nabla(\psi_\ver \psi_\vertt) -\nabla r_\coarse^\ver{\cdot}\nabla \psi_\vertt) = 0.
\end{align}
\end{lemma}
\begin{proof}
From~\eqref{eq:Q inclusion}, the spaces $Q_\coarse^{\ver,\vertt}$ and $Q_{\coarse,\mathrm{c}}^{\ver,\vertt}$ only differ if $\psi_\ver\psi_\vertt\in H_0^1(\Omega)$, the difference only being transformed constants that remove the constraint.
Let thus $\psi_\ver\psi_\vertt\in H_0^1(\Omega)$ and recall the notation~\eqref{eq_g_def}. From~$\V_\coarse^{\ver,\vertt} \subset \H_0(\div,\omega_\vertt)$ (recall~\eqref{eq:H0 div for b}) and~\eqref{eq:sigma beta b}, $0=\dual{\Dv\ssigma_\coarse^{\ver,\vertt}}{1}_{\omega_\vertt} = \dual{\Upsilon_{Q_\coarse^{\ver,\vertt}}g}{1}_{\omega_\vertt}$.
Taking into account~\eqref{eq:projection_Q_a_b_char} and~\eqref{eq:projection_Q_a_b_c},
we thus only need to verify the Neumann compatibility condition $\dual{g}{1}_{\omega_\vertt}=0$.
Hence, test $g$ with $1$ and obtain that
\begin{align}\label{eq:aux1:minimizer}
\dual{g}{1}_{\omega_\vertt}&=\dual{f\psi_\ver}{\psi_\vertt}_{\omega_\vertt}
-\dual{\nabla u_\coarse{\cdot}\nabla\psi_\ver}{\psi_\vertt}_{\omega_\vertt} - \dual{\psi_\ver\nabla u_\coarse}{\nabla \psi_\vertt}_{\omega_\vertt}
- \dual{\nabla r_\coarse^\ver}{\nabla \psi_\vertt}_{\omega_\vertt}.
\end{align}
We consider two cases.

First, let $\psi_\ver\in H_0^1(\Omega)$.
Then, we use that $\supp(\psi_\vertt)\subseteq\supp(\psi_\ver)$ from~\eqref{eq_psi_b}, \eqref{eq:rh beta} with the test function $v_\coarse=\psi_\vertt-|\omega_\ver|^{-1}\dual{\psi_\vertt}{1}_{\omega_\ver}$ having zero mean value on $\omega_\ver$  (which is possible due to~\eqref{eq:b in H01}), and~\eqref{eq:Galerkin for beta}  to see for the last term in~\eqref{eq:aux1:minimizer} that
\begin{eqnarray*}
\dual{\nabla r_\coarse^\ver}{\nabla \psi_\vertt}_{\omega_\vertt}&=& \dual{\nabla r_\coarse^\ver}{\nabla (\psi_\vertt-|\omega_\ver|^{-1}\dual{\psi_\vertt}{1}_{\omega_\ver})}_{\omega_\ver}
\\
 &\reff{eq:rh beta}=&
\dual{f\psi_\ver-\nabla u_\coarse{\cdot}\nabla\psi_\ver}{\psi_\vertt-|\omega_\ver|^{-1}\dual{\psi_\vertt}{1}_{\omega_\ver}}_{\omega_\ver} - \dual{\psi_\ver\nabla u_\coarse}{\nabla \psi_\vertt}_{\omega_\ver}
\\
&\reff{eq:Galerkin for beta}=&
\dual{f\psi_\ver-\nabla u_\coarse{\cdot}\nabla\psi_\ver}{\psi_\vertt}_{\omega_\vertt} - \dual{\psi_\ver\nabla u_\coarse}{\nabla \psi_\vertt}_{\omega_\vertt}.
\end{eqnarray*}
With~\eqref{eq:aux1:minimizer}, we then see that $\dual{g}{1}_{\omega_\vertt}=0$.

It remains to consider $\psi_\ver\not\in H_0^1(\Omega)$.
In this case, we obtain again with $\supp(\psi_\vertt)\subseteq\supp(\psi_\ver)$ and~\eqref{eq:rh beta} with the test function $v_\coarse=\psi_\vertt$ (which is possible due to~\eqref{eq:b in H01}) that
\begin{align*}
\dual{\nabla r_\coarse^\ver}{\nabla \psi_\vertt}_{\omega_\vertt}&= \dual{\nabla r_\coarse^\ver}{\nabla \psi_\vertt}_{\omega_\ver}
\reff{eq:rh beta}=
\dual{f\psi_\ver-\nabla u_\coarse{\cdot}\nabla\psi_\ver}{\psi_\vertt}_{\omega_\ver} - \dual{\psi_\ver\nabla u_\coarse}{\nabla \psi_\vertt}_{\omega_\ver}
\\
&=\dual{f\psi_\ver-\nabla u_\coarse{\cdot}\nabla\psi_\ver}{\psi_\vertt}_{\omega_\vertt} - \dual{\psi_\ver\nabla u_\coarse}{\nabla \psi_\vertt}_{\omega_\vertt}.
\end{align*}
With~\eqref{eq:aux1:minimizer}, we again see that $\dual{g}{1}_{\omega_\vertt}=0$, which concludes the proof.
\end{proof}

%%%%%%%%%%%%%%%%%%%%%%%%%%%%%%%%%%%%%%%%%%%%%%%%%%%%%%%%%%%%%%%%%%%%%%%%%%%%%%%%%%%%%
\subsection{Equilibrated flux on the large patches and the final equilibrated flux}
%%%%%%%%%%%%%%%%%%%%%%%%%%%%%%%%%%%%%%%%%%%%%%%%%%%%%%%%%%%%%%%%%%%%%%%%%%%%%%%%%%%%%

As in the simplified setting of Section~\ref{sec:inex_iga}, our equilibrated flux on the large patches $\omega_\ver$ and the final equilibrated flux are given by:

\begin{definition}\label{def_flux_a} From~\eqref{eq:sigma beta b}, define the patchwise fluxes, for all $\ver\in\VV_\coarse$,
\begin{subequations}\label{eq_flux}
\begin{align}\label{eq_flux_a}
 \ssigma_\coarse^\ver\eq\sum_{\vertt\in\VV_\coarse^\ver} \ssigma_\coarse^{\ver,\vertt}
\end{align}
and the equilibrated flux
\begin{align}\label{eq_flux_tot}
\ssigma_\coarse\eq\sum_{\ver\in\VV_\coarse}\ssigma_\coarse^\ver.
\end{align}
\end{subequations}
\end{definition}

\begin{remark}\label{rem2:naive_approach}
With the definitions of Remark~\ref{rem:naive_approach}, the discrete residual function $r_\coarse^\ver$ from~\eqref{eq:rh beta} is simply zero, so that, up to the projection~$\Upsilon_{Q_\coarse^{\ver,\vertt}}$, the local equilibrated flux $\ssigma_\coarse^{\ver,\vertt}$ from~\eqref{eq:sigma beta b} coincides with $\ssigma_\coarse^\ver$ from~\eqref{eq_flux_a} and \eqref{eq:sigma_a}.
\end{remark}

We first discuss the contributions $\ssigma_\coarse^\ver$, which can be seen as discrete approximations of $-(\nabla r^\ver + \psi_\ver \nabla u_\coarse)$ from~\eqref{eq:r beta_3}. Recalling~\eqref{eq:H0 div for a}, we have.

\begin{lemma}\label{lem:equiflux in beta}
For all $\ver\in\VV_\coarse$, it holds that $\ssigma_\coarse^\ver\in \H_{0}(\div,\omega_\ver)$ and
\begin{align}\label{eq:equiflux in beta}
\dual{\nabla{\cdot}\ssigma_\coarse^\ver}{q_\coarse}_{\omega_\ver}= \dual{f\psi_\ver-\nabla u_\coarse{\cdot}\nabla\psi_\ver}{q_\coarse}_{\omega_\ver}\quad\text{for all }q_\coarse\in \widehat Q_\coarse\circ\F^{-1}.
\end{align}
\end{lemma}
\begin{proof}
By Definition~\ref{def_flux_ab}, we have $\ssigma_\coarse^{\ver,\vertt}\in \V_\coarse^{\ver,\vertt}\subset \H_0(\div,\omega_\vertt)$ for all $\vertt\in\VV_\coarse^\ver$. If $\psi_\ver\in H_0^1(\Omega)$, then $\psi_\ver\psi_\vertt\in H_0^1(\Omega)$, and $\ssigma_\coarse^{\ver,\vertt}{\cdot}\n_{\omega_\vertt}=0$ on $\partial\omega_\vertt$ from~\eqref{eq:H0 div for b}, so that $\ssigma_\coarse^{\ver,\vertt}{\cdot}\n_{\omega_\ver}=0$ on $\partial\omega_\ver$ by virtue of $\supp(\psi_\vertt)\subseteq\supp(\psi_\ver)$, which we assume in~\eqref{eq_psi_b}. If $\psi_\ver\not\in H_0^1(\Omega)$, then $\psi_\ver\psi_\vertt\in H_0^1(\Omega)$ may still hold, in which case the previous reasoning applies. If $\psi_\ver\not\in H_0^1(\Omega)$ and $\psi_\ver\psi_\vertt\not\in H_0^1(\Omega)$, then the homogeneous Neumann (no flow) boundary condition only applies on $\partial\omega_\vertt\cap(\psi_\ver\psi_\vertt)^{-1}(\{0\})$ in~\eqref{eq:H0 div for b}, and, in the sum of contributions over all $\vertt\in\VV_\coarse^\ver$, on $\partial\omega_\ver\cap(\psi_\ver)^{-1}(\{0\})$. Thus $\ssigma_\coarse^\ver\in \H_0(\div,\omega_\ver)$.

To see~\eqref{eq:equiflux in beta}, we use again definitions~\eqref{eq:sigma beta b} and~\eqref{eq_flux_a}, along with the partition of unity property~\eqref{eq:partition of unity2} and~\eqref{eq:a_b_c}. Let $q_\coarse\in \widehat Q_\coarse\circ\F^{-1}$ be fixed. Then
\begin{align*}
\dual{\nabla{\cdot}\ssigma_\coarse^\ver}{q_\coarse}_{\omega_\ver}
& =
\sum_{\vertt\in\VV_\coarse^\ver} \dual{\nabla{\cdot}\ssigma_\coarse^{\ver,\vertt}}{q_\coarse}_{\omega_\vertt}  \\
& \refff{eq:sigma beta b}{eq:a_b_c}=\sum_{\vertt\in\VV_\coarse^\ver}
\dual{\Upsilon_{Q_{\coarse,\mathrm{c}}^{\ver,\vertt}}(f\psi_\ver \psi_\vertt -\nabla u_\coarse{\cdot}\nabla(\psi_\ver \psi_\vertt)-\nabla r_\coarse^\ver{\cdot}\nabla \psi_\vertt)}{q_\coarse}_{\omega_\vertt}\\
& \refff{eq:projection_Q_a_b_c}{eq:Q inclusion}=\sum_{\vertt\in\VV_\coarse^\ver}
\dual{f\psi_\ver \psi_\vertt -\nabla u_\coarse{\cdot}\nabla(\psi_\ver \psi_\vertt)-\nabla r_\coarse^\ver{\cdot}\nabla \psi_\vertt}{q_\coarse}_{\omega_\vertt}\\
& \reff{eq:partition of unity2}= \dual{f\psi_\ver-\nabla u_\coarse{\cdot}\nabla\psi_\ver}{q_\coarse}_{\omega_\ver},
\end{align*}
where we have crucially used Lemma~\ref{lem:easier projection} (on the second line) and the inclusion property~\eqref{eq:Q inclusion} (on the third line).
\end{proof}

We now show that $\ssigma_\coarse$ is indeed an equilibrated flux:

\begin{lemma}
\label{lem:equiflux}
It holds that $\ssigma_\coarse\in \H(\div,\Omega)$ with
\begin{align}\label{eq:equiflux}
\dual{\Dv\ssigma_\coarse}{q_\coarse}_{\Omega} = \dual{f}{q_\coarse}_{\Omega} \quad\text{for all }q_\coarse\in \widehat Q_\coarse\circ\F^{-1},
\end{align}
or equivalently, by virtue of~\eqref{eq:projection_Q}, $\Upsilon_{Q_\coarse}(f-\Dv\ssigma_\coarse)=0$.
\end{lemma}

\begin{proof}
Definition~\ref{def_flux_a}, \eqref{eq:H0 div for a}, and Lemma~\ref{lem:equiflux in beta} immediately imply that $\ssigma_\coarse\in \H(\div,\Omega)$.
To see the second point, we use again Lemma~\ref{lem:equiflux in beta} and the partition of unity property~\eqref{eq:partition of unity}, giving
\begin{align*}
\dual{\nabla{\cdot}\ssigma_\coarse}{q_\coarse}_{\Omega}\reff{eq_flux_tot}=\sum_{\ver\in\VV_\coarse}\dual{\nabla{\cdot}\ssigma_\coarse^\ver}{q_\coarse}_{\omega_\ver}
\reff{eq:equiflux in beta}=\sum_{\ver\in\VV_\coarse}\dual{f\psi_\ver-\nabla u_\coarse{\cdot}\nabla \psi_\ver}{q_\coarse}_{\omega_\ver}
\reff{eq:partition of unity}=\dual{f}{q_\coarse}_{\Omega},
\end{align*}
which concludes the proof.
\end{proof}

%%%%%%%%%%%%%%%%%%%%%%%%%%%%%%%%%%%%%%%%%%%%%%%%%%%%%%%%%%%%%%%%%%%%%%%%%%%%%%%%%%%%%
%%%%%%%%%%%%%%%%%%%%%%%%%%%%%%%%%%%%%%%%%%%%%%%%%%%%%%%%%%%%%%%%%%%%%%%%%%%%%%%%%%%%%
\section{A posteriori error estimates}\label{sec:a_post}
%%%%%%%%%%%%%%%%%%%%%%%%%%%%%%%%%%%%%%%%%%%%%%%%%%%%%%%%%%%%%%%%%%%%%%%%%%%%%%%%%%%%%
%%%%%%%%%%%%%%%%%%%%%%%%%%%%%%%%%%%%%%%%%%%%%%%%%%%%%%%%%%%%%%%%%%%%%%%%%%%%%%%%%%%%%

We are now ready to present our a posteriori error estimates.

%%%%%%%%%%%%%%%%%%%%%%%%%%%%%%%%%%%%%%%%%%%%%%%%%%%%%%%%%%%%%%%%%%%%%%%%%%%%%%%%%%%%%
\subsection{Reliability}
%%%%%%%%%%%%%%%%%%%%%%%%%%%%%%%%%%%%%%%%%%%%%%%%%%%%%%%%%%%%%%%%%%%%%%%%%%%%%%%%%%%%%

With the dual norm
\begin{align*}
\norm{\cdot}{H^{-1}(\Omega)}\eq\sup_{\substack{v\in H_0^1(\Omega)\\ \norm{\nabla v}{\Omega}=1}}\dual{\cdot}{v}_\Omega,
\end{align*}
one immediately  gets the following reliability result:

\begin{proposition}\label{prop:reliable_abs} Let the abstract assumptions of Section~\ref{sec:assumptions} be satisfied, let $u$ solve~\eqref{eq:variational}, and let $u_\coarse$ solve~\eqref{eq:Galerkin}. Let the equilibrated flux be given by Definitions~\ref{def_ra}, \ref{def_flux_ab}, and~\ref{def_flux_a}. Then
\begin{align}\label{eq:reliability}
\norm{\nabla (u-u_\coarse)}{\Omega}\le \norm{\ssigma_\coarse+\nabla u_\coarse}{\Omega} +\norm{(1-\Upsilon_{Q_\coarse})(f-\nabla{\cdot}\ssigma_\coarse)}{H^{-1}(\Omega)}.
\end{align}
\end{proposition}
\begin{proof}
The weak solution definition~\eqref{eq:variational}, the fact that $\ssigma_\coarse\in \H(\div,\Omega)$, the Green theorem, and Lemma~\ref{lem:equiflux} show
\begin{eqnarray*}
\norm{\nabla(u-u_\coarse)}{\Omega}&=&\sup_{\substack{v\in H_0^1(\Omega)\\ \norm{\nabla v}{\Omega}=1}} \big( \dual{f}{v}_\Omega - \dual{\nabla u_\coarse}{\nabla v}_\Omega\big)
\\
&=&\sup_{\substack{v\in H_0^1(\Omega)\\ \norm{\nabla v}{\Omega}=1}} \big( \dual{\nabla{\cdot}\ssigma_\coarse}{v}_\Omega - \dual{\nabla u_\coarse}{\nabla v}_\Omega+\dual{f-\Dv\ssigma_\coarse}{v}_\Omega\big)
\\
&=&\sup_{\substack{v\in H_0^1(\Omega)\\ \norm{\nabla v}{\Omega}=1}} \big(-\dual{\ssigma_\coarse+\nabla u_\coarse}{\nabla v}_\Omega +\dual{(1-\Upsilon_{Q_\coarse})(f-\nabla{\cdot}\ssigma_\coarse)}{v}_\Omega\big)
\\
&\le& \norm{\ssigma_\coarse+\nabla u_\coarse}{\Omega} +\norm{(1-\Upsilon_{Q_\coarse})(f-\nabla{\cdot}\ssigma_\coarse)}{H^{-1}(\Omega)}.
\end{eqnarray*}
\end{proof}

Let $\norm{\cdot}{2}$ denote the spectral norm of a square matrix.
In the particular situation of Section~\ref{sec:discrete spaces}, the $\norm{\cdot}{H^{-1}(\Omega)}$ in~\eqref{eq:reliability} can be further estimated by a computable weighted $L^2$-norm, forming a data oscillation term.

\begin{proposition}\label{prop:reliable}
Suppose that $\widehat Q_\coarse \circ\F^{-1}$ contains the space of piecewise constants with respect to some mesh $\TT_\coarse$ of $\Omega$, as is the case in Section~\ref{sec:discrete spaces}.
Assume that for all $T\in\TT_\coarse$,  $\widehat T\eq\F^{-1}(T)$ is convex.
Then, there holds that
\begin{subequations}
\begin{align}\label{eq:reliability2}
 \norm{\nabla (u-u_\coarse)}{\Omega}\le \norm{\ssigma_\coarse+\nabla u_\coarse}{\Omega} + \osc_\coarse^{\rm rel},
\end{align}
where
\begin{align} \label{eq_osc}
\osc_\coarse^{\rm rel} \eq \Bigg(\sum_{T\in\TT_\coarse} \osc_\coarse^{\rm rel}(T)^2\Bigg)^{1/2} \text{ with }\,
\osc_\coarse^{\rm rel}(T)\eq\frac{C_{\rm rel}}{\pi}\,\diam(\widehat T)  \norm{(1-\Upsilon_{Q_\coarse})(f-\nabla{\cdot}\ssigma_\coarse)}{T}
\end{align}
and
\begin{align}
C_{\rm rel} \eq \norm{\det(D\F)}{\infty,\widehat \Omega}^{1/2}\, \norm{\det(D\F)^{-1}}{\infty,\widehat \Omega} ^{1/2} \,\sup_{\widehat x\in\widehat \Omega}\norm{D\F(\widehat x)}{2}.
\end{align}
\end{subequations}
\end{proposition}

\begin{proof}
Let $v\in H_0^1(\Omega)$ with $\norm{\nabla v}{\Omega}=1$ and let $v_\coarse$ be its $L^2(\Omega)$-orthogonal projection onto the space of $\TT_\coarse$-piecewise constants.
Recall from Lemma~\ref{lem:equiflux} that $\Upsilon_{Q_\coarse}(f-\nabla{\cdot}\ssigma_\coarse)=0$.
Hence, the Cauchy--Schwarz inequality and the Poincar\'e inequality as in~\eqref{eq_PF} (with constant $C_{\rm P}(T)$) show that
\begin{align*}
 &\dual{(1-\Upsilon_{Q_\coarse})(f-\nabla{\cdot}\ssigma_\coarse)}{v}_\Omega
 = \dual{f-\nabla{\cdot}\ssigma_\coarse}{v-v_\coarse}_\Omega
 = \sum_{T\in\TT_\coarse}  \dual{f-\nabla{\cdot}\ssigma_\coarse}{v-v_\coarse}_T
 \\
 &\quad\le \sum_{T\in\TT_\coarse}  \norm{f-\nabla{\cdot}\ssigma_\coarse}{T}\,\norm{v-v_\coarse}{T}
 \le \sum_{T\in\TT_\coarse}  \norm{f-\nabla{\cdot}\ssigma_\coarse}{T}\, \diam(T)C_{\rm P}(T)\norm{\nabla v}{T}
 \\
 &\quad \le \Bigg(\sum_{T\in\TT_\coarse} \diam(T)^2 C_{\rm P}(T)^2 \norm{f-\nabla{\cdot}\ssigma_\coarse}{T}^2\Bigg)^{1/2}.
\end{align*}
According to~\cite[Corollary~2.2]{vv12}, it further holds for $\widehat T\eq\F^{-1}(T)$ that
\begin{align*}
 \diam(T)C_{\rm P}(T) \le \diam(\widehat T)\, C_{\rm P}(\widehat T)\, \norm{\det(D\F)}{\infty,\widehat T}^{1/2}\, \norm{\det(D\F)^{-1}}{\infty,\widehat T} ^{1/2}
  \,\sup_{\widehat x\in\widehat T}\norm{D\F(\widehat x)}{2},
\end{align*}
where $C_{\rm P}(\widehat T)$ denotes the Poincar\'e constant of $\widehat T$, which is smaller or equal to $1/\pi$ as we assumed that $\widehat T$ is convex.
\end{proof}

\begin{remark}\label{rem:reliable oscillations}
If there also holds the converse inclusion in~\eqref{eq:Q inclusion}, i.e., $Q_{\coarse,\mathrm{c}}^{\ver,\vertt}\subseteq Q_\coarse$ (where elements in $Q_{\coarse,\mathrm{c}}^{\ver,\vertt}$ are extended by zero outside of $\omega_\vertt$), then $\nabla{\cdot}\ssigma_\coarse\in Q_\coarse$ and thus $(1-\Upsilon_{Q_\coarse})\Dv\ssigma_\coarse=0$.
For the spaces of Section~\ref{sec:discrete spaces}, the converse inclusion is satisfied for uniform refinements $\TT_\coarse$ of $\TT_0$, but not in general.
In this case, the second term in~\eqref{eq:reliability2} is, for smooth $f$, of order $\mathcal{O}(h^{\widetilde p+2})$, where $h$ denotes the maximal diameter in $\TT_\coarse$, cf. \cite[Equation~(3.12b)]{Dol_Ern_Voh_hp_16} and~\cite[Remark~3.6]{ev15}.
\end{remark}

%%%%%%%%%%%%%%%%%%%%%%%%%%%%%%%%%%%%%%%%%%%%%%%%%%%%%%%%%%%%%%%%%%%%%%%%%%%%%%%%%%%%%
\subsection{Efficiency}
%%%%%%%%%%%%%%%%%%%%%%%%%%%%%%%%%%%%%%%%%%%%%%%%%%%%%%%%%%%%%%%%%%%%%%%%%%%%%%%%%%%%%

To prove efficiency, we will crucially rely on the patchwise Sobolev spaces from Section~\ref{sec_patc_spaces} and Assumption~\ref{ass:min RT to cont}. We will employ the residual function $r^\ver\in H_*^1(\omega_\ver)$ from~\eqref{eq:r beta}.
The next proposition states local efficiency of the equilibrated flux estimator from Proposition~\ref{prop:reliable}:

\begin{proposition}\label{prop:efficient}
Let $\TT_\coarse$ be a mesh of $\Omega$, as is the case in Section~\ref{sec:discrete spaces}.
For all $T \in \TT_\coarse$, it holds that
\begin{subequations}
\begin{align}
\begin{split}
\norm{\ssigma_\coarse + \nabla u_\coarse}{T}
\le   \sum_{\substack{\ver\in\VV_\coarse\\|\omega_\ver\cap T|>0}} \Big(2\sqrt{1+(C_1+C_2)^2}\, C_4 C_{\rm veff}  \norm{\nabla (u-u_\coarse)}{\omega_\ver}  \\
+ \sqrt{2 C_4}C_{\rm veff} \, \osc_\coarse^{\rm eff}(\omega_\ver,T)\Big),
\end{split}
\end{align}
where
\begin{align}
\notag
&\osc_\coarse^{\rm eff}(\omega_\ver,T)\eq \Bigg(\sum_{\substack{\vertt\in\VV_\coarse^\ver\\ |\omega_\vertt\cap T|>0}}
\big(\diam(\omega_\vertt)^2  C_{\rm PF}(\omega_\vertt)^2 \norm{(1-\Upsilon_{Q_{\coarse,\mathrm{c}}^{\ver,\vertt}})(f\psi_\ver\psi_\vertt-\nabla u_\coarse{\cdot}\nabla(\psi_\ver \psi_\vertt))}{\omega_\vertt} ^2
\\
\label{eq:oscillations}
&\hspace{75mm}+\norm{(1-\Pi_{\RT_\coarse^{\ver,\vertt}}) (\psi_\ver\psi_\vertt\nabla u_\coarse) }{\omega_\vertt}^2  \big)\Bigg)^{1/2}
\end{align}
and $\Pi_{\RT_\coarse^{\ver,\vertt}}:[L^2(\omega_\vertt)]^{d}\to \RT_\coarse^{\ver,\vertt}$ denotes the $L^2$-orthogonal projection.
With $C_{\F}:=\norm{\det(D\F)}{\infty,\widehat\Omega}\,\norm{\det(D\F)^{-1}}{\infty,\widehat\Omega}$, the constant $C_{\rm veff}$ is explicitly given as
\begin{align}
C_{\rm veff} := 2 C_{\rm st}\max\big\{C_3 C_5 + C_1 (C_5+C_6), \,C_5+(C_5+C_6) + C_{\F} C_6\big\}.
\end{align}
\end{subequations}
In the setting of Sections~\ref{sec:setting}--\ref{sec:partitions}, all involved constants depend themselves only on the space dimension $d$, the polynomial degree $\overline p$ from Section~\ref{sec:hierarchical partitions} (which itself depends only on the considered smoothness, see Remark~\ref{rem_p_rob}), and $\max\{\norm{D\F}{\infty,\widehat\Omega},\norm{(D\F)^{-1}}{\infty,\widehat\Omega}\}$.
They do \emph{not} depend on the polynomial degrees $p$ and $\widetilde p$.
\end{proposition}

\begin{proof}
We prove the assertion in seven steps.

\noindent{\bf Step~1:}
Definition~\eqref{eq_flux_tot}, the partition of unity property~\eqref{eq:partition of unity}, and the triangle inequality show that
\begin{align*}
\norm{\ssigma_\coarse+\nabla u_\coarse }{T}&\refff{eq_flux_tot}{eq:partition of unity}=\Bigg\|\sum_{\substack{\ver\in\VV_\coarse\\|\omega_\ver\cap T|>0}}\big(\ssigma_\coarse^\ver+\psi_\ver\nabla u_\coarse\big)\Bigg\|_T
\le \sum_{\substack{\ver\in\VV_\coarse\\|\omega_\ver\cap T|>0}} \norm{\ssigma_\coarse^\ver+\psi_\ver\nabla u_\coarse}{T}.
\end{align*}

\noindent{\bf Step~2:}
Next, we bound each summand separately.
Let $\ver\in\VV_\coarse$ with $|\omega_\ver\cap T|>0$.
Then, definition~\eqref{eq_flux_a} and the partition of unity property~\eqref{eq:partition of unity2} together with the Cauchy--Schwarz inequality give that
\begin{align}\begin{split}
\norm{\ssigma_\coarse^\ver+\psi_\ver\nabla u_\coarse}{T}& \refff{eq_flux_a}{eq:partition of unity2}= \Bigg\|\sum_{\substack{\vertt\in\VV_\coarse^\ver\\ |\omega_\vertt\cap T|>0}} \big(\ssigma_\coarse^{\ver,\vertt}+\psi_\ver \psi_\vertt\nabla u_\coarse\big)\Bigg\|_T\\
&\reff{eq:number of bs}\le \sqrt{C_4}
\Bigg(\sum_{\substack{\vertt\in\VV_\coarse^\ver\\ |\omega_\vertt\cap T|>0}}\norm{\ssigma_\coarse^{\ver,\vertt} + \psi_\ver \psi_\vertt\nabla u_\coarse}{T}^2\Bigg)^{1/2}\\
&\reff{eq:b bounded}\le \sqrt{2 C_4} \Bigg(\!\!\sum_{\substack{\vertt\in\VV_\coarse^\ver\\ |\omega_\vertt\cap T|>0}}\big(\norm{\ssigma_\coarse^{\ver,\vertt}+\psi_\vertt(\psi_\ver \nabla u_\coarse + \nabla r_\coarse^\ver)}{T}^2+ C_5^2\norm{\nabla r_\coarse^\ver}{\omega_\vertt}^2\big)\!\!\Bigg)^{1/2}.
\end{split}
\label{eq:after C1}
\end{align}

\noindent{\bf Step~3:}
We estimate the first summand of~\eqref{eq:after C1}.
To this end, let $\vertt\in\VV_\coarse^\ver$ with $|\omega_\vertt\cap T|>0$.
As in Lemma~\ref{lem:minimizer}, we abbreviate $g \eq f\psi_\ver \psi_\vertt-\nabla u_\coarse{\cdot}\nabla(\psi_\ver \psi_\vertt) -\nabla r_\coarse^\ver{\cdot}\nabla \psi_\vertt$ as well as $\ttau\eq\psi_\vertt(\psi_\ver \nabla u_\coarse + \nabla r_\coarse^\ver)$.
Then, definition~\eqref{eq:sigma beta b} (or its equivalent formulation~\eqref{eq:computation of sigma}, which allows to stick in the projector $\Pi_{\RT_\coarse^{\ver,\vertt}}$) with~$\V_\coarse^{\ver,\vertt}\subseteq\RT_\coarse^{\ver,\vertt}$ together with~\eqref{eq:min RT to cont} give that
\begin{align*}
&\norm{\ssigma_\coarse^{\ver,\vertt}+\psi_\vertt(\psi_\ver \nabla u_\coarse + \nabla r_\coarse^\ver)}{T}
\le
\norm{\ssigma_\coarse^{\ver,\vertt}+\ttau}{\omega_\vertt}
\\
&\qquad\,\reff{eq:sigma beta b}=\,
\min_{\substack{\bv_\coarse\in\V_\coarse^{\ver,\vertt}\\
 \Dv\bv_\coarse=\Upsilon_{Q_\coarse^{\ver,\vertt}}(g)}}
\norm{\bv_\coarse + \Pi_{\RT_\coarse^{\ver,\vertt}} (\ttau)}{\omega_\vertt}
\reff{eq:min RT to cont}\le
C_{\rm st} \min_{\substack{\bv \in \H_0(\div,\omega_\vertt)\\
\Dv\bv=\Upsilon_{Q_\coarse^{\ver,\vertt}}(g)}}
\norm{\bv+ \Pi_{\RT_\coarse^{\ver,\vertt}} (\ttau) }{\omega_\vertt}
\\
&\qquad\reff{eq:b bounded}\le
C_{\rm st} \Bigg(\min_{\substack{\bv \in \H_0(\div,\omega_\vertt)\\
\Dv\bv=\Upsilon_{Q_\coarse^{\ver,\vertt}}(g)}}
\norm{\bv+ \ttau }{\omega_\vertt}
+ \norm{(1-\Pi_{\RT_\coarse^{\ver,\vertt}}) (\psi_\ver\psi_\vertt\nabla u_\coarse) }{\omega_\vertt}
+ C_5\norm{\nabla r_\coarse^\ver}{\omega_\vertt}\Bigg).
\end{align*}
Let $r^{\ver,\vertt} \in H_*^1(\omega_\vertt)$ solve
\[
\dual{\nabla r^{\ver,\vertt}}{\nabla v}_{\omega_\vertt}
=\dual{\Upsilon_{Q_\coarse^{\ver,\vertt}}(g)}{v}_{\omega_\vertt}- \dual{\ttau}{\nabla v}_{\omega_\vertt}
\quad\text{for all }v\in H_*^1(\omega_\vertt).
\]
Then a standard primal-dual equivalence gives, as in, e.g., \cite[Corollary~3.6]{ev20},
\begin{align*}
\min_{\substack{\bv \in \H_0(\div,\omega_\vertt)\\
\Dv\bv=\Upsilon_{Q_\coarse^{\ver,\vertt}}(g)}}
\norm{\bv+ \ttau }{\omega_\vertt}
= \norm{\nabla r^{\ver,\vertt}}{\omega_\vertt} =  \sup_{\substack{v\in H_*^1(\omega_\vertt)\\ \norm{\nabla v}{\omega_\vertt}=1}}
\big(\dual{\Upsilon_{Q_\coarse^{\ver,\vertt}} (g)}{v}_{\omega_\vertt} -\dual{\ttau}{\nabla v}_{\omega_\vertt}\big).
\end{align*}
Thus,
\begin{align}
& \min_{\substack{\bv \in \H_0(\div,\omega_\vertt)\\
\Dv\bv=\Upsilon_{Q_\coarse^{\ver,\vertt}}(g)}}
\norm{\bv+ \ttau }{\omega_\vertt}
=  \sup_{\substack{v\in H_*^1(\omega_\vertt)\\ \norm{\nabla v}{\omega_\vertt}=1}}
\big( -\dual{\ttau}{\nabla v}_{\omega_\vertt}
+\dual{g}{v}_{\omega_\vertt}
-\dual{(1-\Upsilon_{Q_\coarse^{\ver,\vertt}})g}{v}_{\omega_\vertt}\big)
\notag
\\
&\,\,\,\,=\,\,\,\, \sup_{\substack{v\in H_*^1(\omega_\vertt)\\ \norm{\nabla v}{\omega_\vertt}=1}}
\big(\dual{f\psi_\ver-\nabla u_\coarse {\cdot}\nabla \psi_\ver}{v\psi_\vertt}_{\omega_\vertt} - \dual{\psi_\ver\nabla u_\coarse + \nabla r_\coarse^\ver}{\nabla(v\psi_\vertt)}_{\omega_\vertt}
\label{eq:without osc}
\\
&\qquad\qquad-\dual{(1-\Upsilon_{Q_\coarse^{\ver,\vertt}})(f\psi_\ver\psi_\vertt - \nabla u_\coarse{\cdot}\nabla (\psi_\ver \psi_\vertt) -  \nabla r_\coarse^\ver{\cdot}\nabla \psi_\vertt)}{v}_{\omega_\vertt}\big).
\label{eq:with osc}
\end{align}

\noindent{\bf Step~4:} Recalling~\eqref{eq:H1_b}, we note that for all $v\in H_*^1(\omega_\vertt)$,
\begin{align*}
\dual{f\psi_\ver}{v\psi_\vertt}_{\omega_\vertt} & = \dual{f}{\underbrace{v\psi_\ver\psi_\vertt}_{\in H_0^1(\Omega)}}_{\omega_\vertt} \reff{eq:variational}=\dual{\nabla u}{\nabla (v\psi_\ver \psi_\vertt)}_{\omega_\vertt}\\
& = \dual{\nabla u {\cdot}\nabla \psi_\ver}{v\psi_\vertt}_{\omega_\vertt} + \dual{\psi_\ver \nabla u}{\nabla (v\psi_\vertt)}_{\omega_\vertt}.
\end{align*}
With the Poincar\'e--Friedrichs inequality~\eqref{eq_PF} and Assumption~\ref{ass:CPF}, the latter equality allows us to further estimate the term in~\eqref{eq:without osc}
\begin{eqnarray*}
&&\hspace{-12mm}\sup_{\substack{v\in H_*^1(\omega_\vertt)\\ \norm{\nabla v}{\omega_\vertt}=1}}
\big(\dual{f\psi_\ver-\nabla u_\coarse{\cdot}\nabla \psi_\ver}{v\psi_\vertt}_{\omega_\vertt} - \dual{\psi_\ver\nabla u_\coarse + \nabla r_\coarse^\ver}{\nabla(v\psi_\vertt)}_{\omega_\vertt}\big)
\\
&=&\sup_{\substack{v\in H_*^1(\omega_\vertt)\\ \norm{\nabla v}{\omega_\vertt}=1}}
\big(\dual{\nabla (u-u_\coarse){\cdot}\nabla \psi_\ver}{v\psi_\vertt}_{\omega_\vertt} + \dual{\psi_\ver\nabla (u-u_\coarse) - \nabla r_\coarse^\ver}{\nabla (v\psi_\vertt)}_{\omega_\vertt}\big)
\\
&\le& \norm{\nabla (u-u_\coarse)}{\omega_\vertt} \norm{\nabla \psi_\ver}{\infty,\omega_\vertt} \sup_{\substack{v\in H_*^1(\omega_\vertt)\\ \norm{\nabla v}{\omega_\vertt}=1}} \norm{v\psi_\vertt}{\omega_\vertt}
\\
&&\quad+ \big(\norm{\psi_\ver}{\infty,\omega_\vertt} \norm{\nabla (u-u_\coarse)}{\omega_\vertt}+ \norm{\nabla r_\coarse^\ver}{\omega_\vertt}\big) \sup_{\substack{v\in H_*^1(\omega_\vertt)\\ \norm{\nabla v}{\omega_\vertt}=1}} \norm{\nabla (v\psi_\vertt)}{\omega_\vertt}
\\
&\le& \norm{\nabla (u-u_\coarse)}{\omega_\vertt} \big(\norm{\nabla\psi_\ver}{\infty,\omega_\vertt} \norm{\psi_\vertt}{\infty,\omega_\vertt}\diam(\omega_\vertt) C_{\rm PF}(\omega_\vertt) \big)
\\
&&\quad+ \big(\norm{\psi_\ver}{\infty,\omega_\vertt} \norm{\nabla (u-u_\coarse)}{\omega_\vertt}+ \norm{\nabla r_\coarse^\ver}{\omega_\vertt}\big)\big(\norm{\nabla \psi_\vertt}{\infty,\omega_\vertt}\diam(\omega_\vertt) C_{\rm PF}(\omega_\vertt)+\norm{\psi_\vertt}{\infty,\omega_\vertt}\big)
\\
&\stackrel{\ref{ass:CPF}}\le&
(C_3 C_5 + C_1 (C_5+C_6))\norm{\nabla(u-u_\coarse)}{\omega_\vertt}
+ (C_5+C_6) \norm{\nabla r_\coarse^\ver}{\omega_\vertt}.
\end{eqnarray*}

\noindent{\bf Step~5:}
To estimate the term in~\eqref{eq:with osc},
we use that $\Upsilon_{Q_\coarse^{\ver,\vertt}}g = \Upsilon_{Q_{\coarse,\mathrm{c}}^{\ver,\vertt}}g$ from Lemma~\ref{lem:easier projection}, \eqref{eq:piola2}, \eqref{eq:projection_Q_a_b_c}, the Poincar\'e--Friedrichs inequality~\eqref{eq_PF}, and~\eqref{eq:CPFcont b} to infer that
\begin{eqnarray*}
&& \hspace{-15mm}\sup_{\substack{v\in H_*^1(\omega_\vertt)\\ \norm{\nabla v}{\omega_\vertt}=1}}
\dual{(1-\Upsilon_{Q_\coarse^{\ver,\vertt}})(f\psi_\ver\psi_\vertt - \nabla u_\coarse{\cdot}\nabla (\psi_\ver \psi_\vertt) -  \nabla r_\coarse^\ver{\cdot}\nabla \psi_\vertt)}{v}_{\omega_\vertt}
\\
&\refff{eq:piola2}{eq:projection_Q_a_b_c}\le &
\big(\norm{(1-\Upsilon_{Q_{\coarse,\mathrm{c}}^{\ver,\vertt}})(f\psi_\ver\psi_\vertt-\nabla u_\coarse{\cdot}\nabla(\psi_\ver \psi_\vertt))}{\omega_\vertt}
\\
&&+\norm{\det(D\F)}{\infty,\widehat\omega_\vertt}\norm{\det(D\F)^{-1}}{\infty,\widehat\omega_\vertt} \norm{\nabla r_\coarse^\ver}{\omega_\vertt} \norm{\nabla \psi_\vertt}{\infty,\omega_\vertt} \big)
\diam(\omega_\vertt)  C_{\rm PF}(\omega_\vertt)
\\
&\reff{eq:CPFcont b}\le&\diam(\omega_\vertt)  C_{\rm PF}(\omega_\vertt) \norm{(1-\Upsilon_{Q_{\coarse,\mathrm{c}}^{\ver,\vertt}})(f\psi_\ver\psi_\vertt-\nabla u_\coarse{\cdot}\nabla(\psi_\ver \psi_\vertt))}{\omega_\vertt}
+ C_{\F}C_6\norm{\nabla r_\coarse^\ver}{\omega_\vertt}.
\end{eqnarray*}
Together with Steps~3--4 and the Cauchy--Schwarz inequality, we conclude that
\begin{align}\label{eq_patch_eff} \begin{split}
&\norm{\ssigma_\coarse^{\ver,\vertt}+\psi_\vertt(\psi_\ver \nabla u_\coarse + \nabla r_\coarse^\ver)}{T}^2
\\
&\quad \le C_{\rm veff}^2\Big(\norm{\nabla (u-u_\coarse)}{\omega_\vertt}^2 +  \norm{\nabla r_\coarse^\ver}{\omega_\vertt}^2
+ \norm{(1-\Pi_{\RT_\coarse^{\ver,\vertt}}) (\psi_\ver\psi_\vertt\nabla u_\coarse) }{\omega_\vertt}^2
\\
&\qquad + \diam(\omega_\vertt)^{2}  C_{\rm PF}(\omega_\vertt)^{2} \norm{(1-\Upsilon_{Q_{\coarse,\mathrm{c}}^{\ver,\vertt}})(f\psi_\ver\psi_\vertt-\nabla u_\coarse{\cdot}\nabla(\psi_\ver \psi_\vertt))}{\omega_\vertt}^2\Big).
\end{split} \end{align}

\noindent{\bf Step~6:}
As a final auxiliary step, we bound $\norm{\nabla r_\coarse^\ver}{\omega_\ver}$.
Since $r_\coarse^\ver$ from~\eqref{eq:rh beta} is the Galerkin approximation of $r^\ver$ from~\eqref{eq:r beta}, and crucially employing~\eqref{eq:H1_a}, we see with the variational formulation~\eqref{eq:variational} that
\begin{eqnarray*}
\norm{\nabla r_\coarse^\ver}{\omega_\ver}
&\le& \norm{\nabla r^\ver}{\omega_\ver}
\refff{eq:r beta}{eq:H1_a}=\sup_{\substack{v\in H_*^1(\omega_\ver)\\ \norm{\nabla v}{\omega_\ver}=1 }}\big(\dual{f}{\underbrace{v\psi_\ver}_{\in H_0^1(\Omega)}}_{\omega_\ver}-\dual{\nabla u_\coarse}{\nabla(v\psi_\ver)}_{\omega_\ver}\big)
\\
&\reff{eq:variational}=&\sup_{\substack{v\in H_*^1(\omega_\ver)\\ \norm{\nabla v}{\omega_\ver}=1 }} \dual{\nabla(u- u_\coarse)}{\nabla(v\psi_\ver)}_{\omega_\ver}
\le \norm{\nabla(u- u_\coarse)}{\omega_\ver} \sup_{\substack{v\in H_*^1(\omega_\ver)\\ \norm{\nabla v}{\omega_\ver}=1 }} \norm{\nabla(v \psi_\ver)}{\omega_\ver}.
\end{eqnarray*}
Relying on the Poincar\'e inequality~\eqref{eq_PF} and with~\eqref{eq:CPFcont beta} and~\eqref{eq:beta bounded}, we can  bound the supremum
\begin{align} \label{CPa}
 \sup_{\substack{v\in H_*^1(\omega_\ver)\\ \norm{\nabla v}{\omega_\ver}=1 }} \norm{\nabla(v\psi_\ver)}{\omega_\ver}
 &\le  \sup_{\substack{v\in H_*^1(\omega_\ver)\\ \norm{\nabla v}{\omega_\ver}=1 }} \big(\norm{\nabla\psi_\ver}{\infty,\omega_\ver}\norm{v}{\omega_\ver} + \norm{\psi_\ver}{\infty,\omega_\ver}\norm{\nabla v}{\omega_\ver}\big)\refff{eq:beta bounded}{eq:CPFcont beta}\le C_1+C_2.
\end{align}

\noindent{\bf Step~7:}
Putting all steps together, also using that $C_5 \leq C_{\rm veff}$, we obtain that
\begin{align*}
& \norm{\ssigma_\coarse+\nabla u_\coarse}{T} \\
&\stackrel{1,2}\le \sum_{\substack{\ver\in\VV_\coarse\\|\omega_\ver\cap T|>0}} \!\!\!
\sqrt{2 C_4} \Bigg(\!\!\sum_{\substack{\vertt\in\VV_\coarse^\ver\\ |\omega_\vertt\cap T|>0}}\big(\norm{\ssigma_\coarse^{\ver,\vertt} + \psi_\vertt (\psi_\ver\nabla u_\coarse + \nabla r_\coarse^\ver)}{T}^2+C_5^2\norm{\nabla r_\coarse^\ver}{\omega_\vertt}^2\big)\!\Bigg)^{1/2} \\
&\stackrel{\eqref{eq_patch_eff}}\le
\sqrt{2 C_4} \sum_{\substack{\ver\in\VV_\coarse\\|\omega_\ver\cap T|>0}} \!\!\! \Bigg(\!\! \sum_{\substack{\vertt\in\VV_\coarse^\ver\\ |\omega_\vertt\cap T|>0}}   \Big(C_{\rm veff}^2\big(\norm{\nabla (u-u_\coarse)}{\omega_\vertt}^2
+\norm{\nabla r_\coarse^\ver}{\omega_\vertt}^2 \\
& \quad + \norm{(1-\Pi_{\RT_\coarse^{\ver,\vertt}}) (\psi_\ver\psi_\vertt\nabla u_\coarse) }{\omega_\vertt}^2
\\
&\quad + \diam(\omega_\vertt)^2  C_{\rm PF}(\omega_\vertt)^2 \norm{(1-\Upsilon_{Q_{\coarse,\mathrm{c}}^{\ver,\vertt}})(f\psi_\ver\psi_\vertt-\nabla u_\coarse{\cdot}\nabla(\psi_\ver \psi_\vertt))}{\omega_\vertt}^2\big) + C_5^2\norm{\nabla r_\coarse^\ver}{\omega_\vertt}^2\Big)\!\Bigg)^{1/2}
\\
&\refff{eq:number of bs}{eq:oscillations} \le
2\, C_4 C_{\rm veff}\sum_{\substack{\ver\in\VV_\coarse\\|\omega_\ver\cap T|>0}} \!\!\!
\big(\norm{\nabla (u-u_\coarse)}{\omega_\ver}^2
+\norm{\nabla r_\coarse^\ver}{\omega_\ver}^2\big)^{1/2}
+ \sqrt{2 C_4} C_{\rm veff} \sum_{\substack{\ver\in\VV_\coarse\\|\omega_\ver\cap T|>0}} \!\!\! \osc_\coarse^{\rm eff}(\omega_\ver,T)
\\
&\stackrel{6}\le
2\sqrt{1+(C_1+C_2)^2}\, C_4 C_{\rm veff}\sum_{\substack{\ver\in\VV_\coarse\\|\omega_\ver\cap T|>0}} \!\!\!
\norm{\nabla (u-u_\coarse)}{\omega_\ver}
+\sqrt{2 C_4} C_{\rm veff} \sum_{\substack{\ver\in\VV_\coarse\\|\omega_\ver\cap T|>0}} \!\!\! \osc_\coarse^{\rm eff}(\omega_\ver,T),
\end{align*}
which concludes the proof.
\end{proof}

\begin{remark}\label{rem:efficient oscillations}
In the situation of Section~\ref{sec:partitions}, $\osc_\coarse^{\rm eff}(\omega_\ver,T)$ of~\eqref{eq:oscillations} vanishes if $\F$ is affine, $\TT_\coarse$ is a uniform refinement of $\TT_0$, $f$ is a $\TT_{\coarse}$-piecewise polynomial of some degree $\bm{q}=(q,\dots,q)$ with $q\ge0$, and $\widetilde p \ge \max\{q+\overline p +1, p+\overline p+1 \}$.  Note that $\widetilde p \ge p+\overline p+1$ in particular implies that $\psi_\ver \psi_\vertt\nabla u_\coarse+\psi_\vertt\nabla r_\coarse^\ver \in \RT^{\widetilde\p}(\TT_\vertt)$ that we take for $\RT_\coarse^{\ver,\vertt}$.
\end{remark}

We now turn towards the global efficiency. In order to achieve robustness with respect to the strength of the hierarchical refinement (the number of hanging nodes), we do not straightforwardly use the element-related result of Proposition~\ref{prop:efficient} but rather resort to its patch-related variant.
With an assumption on the maximal overlap by the patches $\omega_\ver$ (not limiting the strength of the hierarchical refinement, see Remark~\ref{rem:refinement}), our global efficiency result is:

\begin{proposition}\label{prop:efficiency}
Let $\TT_\coarse$ be a mesh of $\Omega$, as is the case in Section~\ref{sec:discrete spaces}.
Let $C_{\rm over}>0$ be a constant such that
\begin{subequations}
\begin{align}\label{eq:overlap}
\sup_{\bx \in \Omega} \#\set{\ver\in\VV_\coarse}{\bx \in \omega_\ver}\le C_{\rm over}.
\end{align}
Then, there holds that
\begin{align}\label{eq:efficiency}\begin{split}
\norm{\ssigma_\coarse + \nabla u_\coarse}{\Omega}^2
\le  4 C_{\rm over}^2 C_4^2 C_{\rm veff}^2 ((C_1+C_2)^2+1) \norm{\nabla (u-u_\coarse)}{\Omega}^2
 \\
+ 4 C_{\rm over} C_4 C_{\rm veff}^2 \sum_{\ver\in\VV_\coarse}\osc_\coarse^{\rm eff}(\omega_\ver)^2,
\end{split}\end{align}
where
\begin{align}\label{eq:oscillations_glob}
\begin{split}
\osc_\coarse^{\rm eff}(\omega_\ver)^2\eq \sum_{\vertt\in\VV_\coarse^\ver}
\big(\diam(\omega_\vertt)^2  C_{\rm PF}(\omega_\vertt)^2 \norm{(1-\Upsilon_{Q_{\coarse,\mathrm{c}}^{\ver,\vertt}})(f\psi_\ver\psi_\vertt-\nabla u_\coarse{\cdot}\nabla(\psi_\ver \psi_\vertt))}{\omega_\vertt} ^2
\\
+\norm{(1-\Pi_{\RT_\coarse^{\ver,\vertt}}) (\psi_\ver\psi_\vertt\nabla u_\coarse) }{\omega_\vertt}^2  \big).
\end{split}
\end{align}
\end{subequations}
In the setting of Sections~\ref{sec:setting}--\ref{sec:partitions}, all involved constants except of $C_{\rm over}$ depend themselves only on the space dimension $d$, the polynomial degree $\overline p$ from Section~\ref{sec:hierarchical partitions} (which itself depends only on the considered smoothness, see Remark~\ref{rem_p_rob}), and $\max\{\norm{D\F}{\infty,\widehat\Omega},\norm{(D\F)^{-1}}{\infty,\widehat\Omega}\}$.
They do \emph{not} depend on the polynomial degrees $p$ and $\widetilde p$. $C_{\rm over}$ is discussed in Remark~\ref{rem:refinement} below.
\end{proposition}

\begin{proof}
Proceeding as in the proof of Proposition~\ref{prop:efficient} while relying on Definition~\ref{def_flux_a} and the partitions of unity~\eqref{eq:partition of unity} and~\eqref{eq:partition of unity2} together with the finite overlap assumptions~\eqref{eq:number of bs} and~\eqref{eq:overlap}, we see
\begin{align*}
 \norm{\ssigma_\coarse + \nabla u_\coarse}{\Omega}^2
 & \refff{eq:partition of unity}{eq_flux_tot}= \Bigg\|\sum_{\ver\in\VV_\coarse}\big(\ssigma_\coarse^\ver+\psi_\ver\nabla u_\coarse\big)\Bigg\|_\Omega^2 \reff{eq:overlap}\leq C_{\rm over} \sum_{\ver\in\VV_\coarse} \norm{\ssigma_\coarse^\ver+\psi_\ver\nabla u_\coarse}{\omega_\ver}^2 \\
 & \refff{eq:partition of unity2}{eq_flux_a}= C_{\rm over} \sum_{\ver\in\VV_\coarse} \Bigg\|\sum_{\vertt\in\VV_\coarse^\ver} \big(\ssigma_\coarse^{\ver,\vertt}+\psi_\ver \psi_\vertt\nabla u_\coarse\big)\Bigg\|_{\omega_\ver}^2 \\
 & \quad \reff{eq:number of bs}\leq C_{\rm over} C_4 \sum_{\ver\in\VV_\coarse} \sum_{\vertt\in\VV_\coarse^\ver} \norm{\ssigma_\coarse^{\ver,\vertt}+\psi_\ver \psi_\vertt\nabla u_\coarse}{\omega_\vertt}^2.
\end{align*}
We now employ~\eqref{eq_patch_eff} plus the triangle and the Cauchy--Schwarz inequalities. Also using~\eqref{eq:b bounded} and the fact that $C_5 \leq C_{\rm veff}$, this leads to
\[
\norm{\ssigma_\coarse + \nabla u_\coarse}{\Omega}^2 \leq C_{\rm over} C_4 4 C_{\rm veff}^2 \sum_{\ver\in\VV_\coarse} \Bigg(\sum_{\vertt\in\VV_\coarse^\ver}
\big(\norm{\nabla (u-u_\coarse)}{\omega_\vertt}^2
+\norm{\nabla r_\coarse^\ver}{\omega_\vertt}^2\big)
+ \osc_\coarse^{\rm eff}(\omega_\ver)^2\Bigg),
\]
and we are left to treat the first two terms. The finite overlap assumptions~\eqref{eq:number of bs} and~\eqref{eq:overlap} again imply
\[
\sum_{\ver\in\VV_\coarse} \sum_{\vertt\in\VV_\coarse^\ver}
\norm{\nabla (u-u_\coarse)}{\omega_\vertt}^2 \reff{eq:number of bs}\leq C_4 \sum_{\ver\in\VV_\coarse} \norm{\nabla (u-u_\coarse)}{\omega_\ver}^2 \reff{eq:overlap}\leq C_{\rm over} C_4 \norm{\nabla (u-u_\coarse)}{\Omega}^2.
\]
Reasoning similarly and also employing the two estimates from Step~6 of the proof of Proposition~\ref{prop:efficient}, we see
\begin{align*}
\sum_{\ver\in\VV_\coarse} \sum_{\vertt\in\VV_\coarse^\ver}
\norm{\nabla r_\coarse^\ver}{\omega_\vertt}^2 & \reff{eq:number of bs}\leq C_4 \sum_{\ver\in\VV_\coarse} \norm{\nabla r_\coarse^\ver}{\omega_\ver}^2 \leq C_4 (C_1+C_2)^2 \sum_{\ver\in\VV_\coarse} \norm{\nabla(u- u_\coarse)}{\omega_\ver}^2 \\
& \reff{eq:overlap}\leq C_{\rm over} C_4 (C_1+C_2)^2 \norm{\nabla (u-u_\coarse)}{\Omega}^2,
\end{align*}
which altogether gives~\eqref{eq:efficiency}.
\end{proof}

\begin{remark}\label{rem:refinement}
For (scaled) hierarchical B-splines $\psi_\ver$ of degree $\overline p$ on $\overline{\mathcal{H}}$-admissible (i.e., $\mathcal{H}$-admissible with respect to $\overline{p}$, $\overline\KK_{0}$, and $\overline{m}$) meshes $\widehat\TT_\coarse$ of class $\mu$ introduced along with suitable refinement algorithms in~\cite{ghp17} for $\mu=2$ and in~\cite{bgv18} for $\mu\ge2$, the upper bound $C_{\rm over}$ from~\eqref{eq:overlap} only depends on the polynomial degree $\overline p$ (which itself only depends on the considered smoothness), the initial knot vector $\overline\KK_0$, the multiplicity $\overline{m}$, and the grading parameter $\mu$, allowing for meshes with arbitrarily many hanging nodes. Meshes satisfying assumption~\eqref{eq:overlap} with the constant $C_{\rm over}$ only depending on the polynomial degree $\overline p$ are considered in the numerics Section~\ref{sec:numerics} below.
\end{remark}

\begin{remark} \label{rem:new}
We further mention that the degree $\overline p$ and the smoothness $\overline p - \overline m$ for (scaled) hierarchical B-splines $\psi_\ver$, or more precisely the approximation power of the spanned space, do not directly affect the oscillations~\eqref{eq:oscillations_glob} and thus the efficiency of the estimator.
%Indeed, it is rather important that the $\psi_\ver$ can be well approximated.
Similarly, the chosen $\overline p$ and $\overline m$ do not directly affect the oscillations~\eqref{eq_osc} and thus the reliability of the estimator. 
Indeed, as discussed in Remarks~\ref{rem:align} and~\ref{rem:efficient oscillations}, the presence and size of data oscillations is rather induced by the approximation properties of the spaces $Q_\coarse^{\ver,\vertt}$ and $\RT^{\widetilde\p}(\TT_\vertt)$ which are discontinuous (piecewise with respect to $\TT_\vertt$).
\end{remark}

%%%%%%%%%%%%%%%%%%%%%%%%%%%%%%%%%%%%%%%%%%%%%%%%%%%%%%%%%%%%%%%%%%%%%%%%%%%%%%%%%%%%%%
%%%%%%%%%%%%%%%%%%%%%%%%%%%%%%%%%%%%%%%%%%%%%%%%%%%%%%%%%%%%%%%%%%%%%%%%%%%%%%%%%%%%%%
\section{Numerical experiments}\label{sec:numerics}
%%%%%%%%%%%%%%%%%%%%%%%%%%%%%%%%%%%%%%%%%%%%%%%%%%%%%%%%%%%%%%%%%%%%%%%%%%%%%%%%%%%%%%
%%%%%%%%%%%%%%%%%%%%%%%%%%%%%%%%%%%%%%%%%%%%%%%%%%%%%%%%%%%%%%%%%%%%%%%%%%%%%%%%%%%%%%

We consider problem~\eqref{eq:Poisson} on the quarter ring depicted in Figure~\ref{fig:qring},
\begin{align*}
 \Omega\eq\set{r(\cos(\varphi),\sin(\varphi))}{r\in(1/2,1)\wedge\varphi\in(0,\pi/2)}
\end{align*}
with NURBS parametrization $\F$ as in \cite[Section~6.3]{ghp17},
and prescribe the exact solution
\begin{align} \label{eq_sol}
 u(x,y)=xy\sin(4\pi(x^2+y^2)).
\end{align}
For polynomial degrees $p\in\{1,\dots,5\}$ and multiplicities $m\in\{1,p\}$, we define the initial knot vectors by
\begin{align*}
 \KK_{1(0)}\eq\KK_{2(0)} \eq \big(\underbrace{0,\dots,0}_{(p+1) \text{-times}}, \underbrace{1/2,\dots,1/2}_{m \text{-times}}, \underbrace{1,\dots,1}_{(p+1) \text{-times}}\big),
 \end{align*}
leading to piecewise $p$-degree polynomials with $C^{p-m}$ smoothness.
As the corresponding polynomial degree $\overline p$ for the partition of unity by the $\psi_\ver$ in Section~\ref{sec:hierarchical partitions}, we choose $\overline  p\eq p+1-m$
and the corresponding multiplicity $\overline m \eq \overline p - p +m=1$ following~\eqref{eq_ov_p_2}, so that the $\psi_\ver$ are mapped piecewise $\overline  p$-degree polynomials of class $C^{\overline p- 1}$.
The polynomial degree $\widetilde p$ for the flux equilibration in Section~\ref{sec:discrete spaces} is chosen in $\{p+1,p+2\}$.
Following Remark~\ref{rem:efficient oscillations}, ignoring temporarily the source term $f$, this would only imply $\osc_\coarse^{\rm eff}(\omega_\ver,T) = \osc_\coarse^{\rm eff}(\omega_\ver) = 0$ in Propositions~\ref{prop:efficient} and~\ref{prop:efficiency} if $\F$ was affine (which is not the case here) and $\widetilde p \ge p+\overline p+1$ (which is only the case here for $m=p$, \ie, $C^0$, but not higher-smoothness splines; note that in the case of maximum $C^{p-1}$ smoothness, we should use $\widetilde p \ge p+\overline p+1 = 2p +1$ theoretically, but we merely employ $\widetilde p = p +1$ or $\widetilde p = p +2$ numerically). Nevertheless, both choices $\widetilde p \in \{p+1,p+2\}$ seem to perform numerically well in the considered test case also for high-smoothness cases, up to $m =1$, corresponding to $C^{p-1}$ splines.

\begin{figure}[t]
\begin{center}
\includegraphics[width=0.2\textwidth]{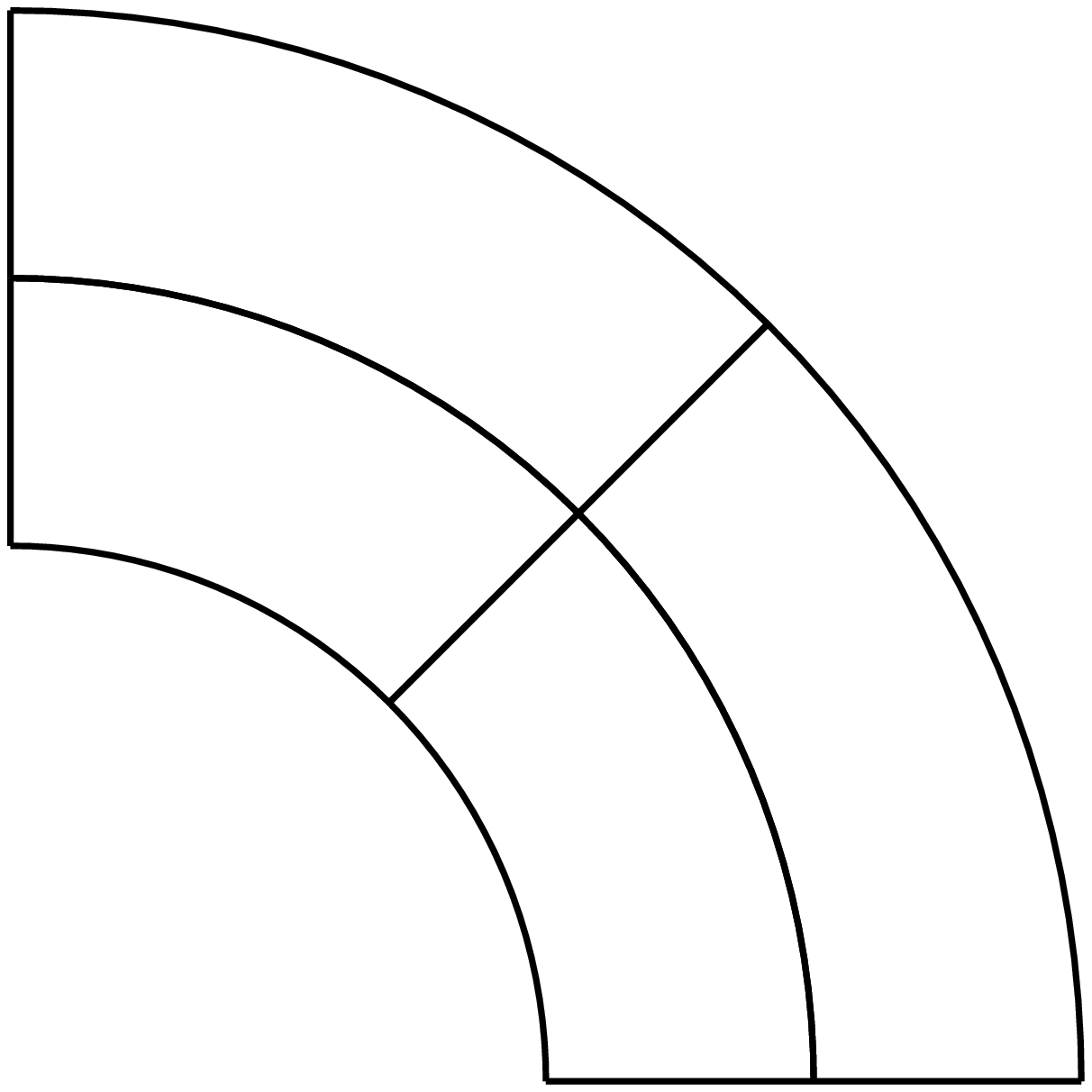}
\includegraphics[width=0.2\textwidth]{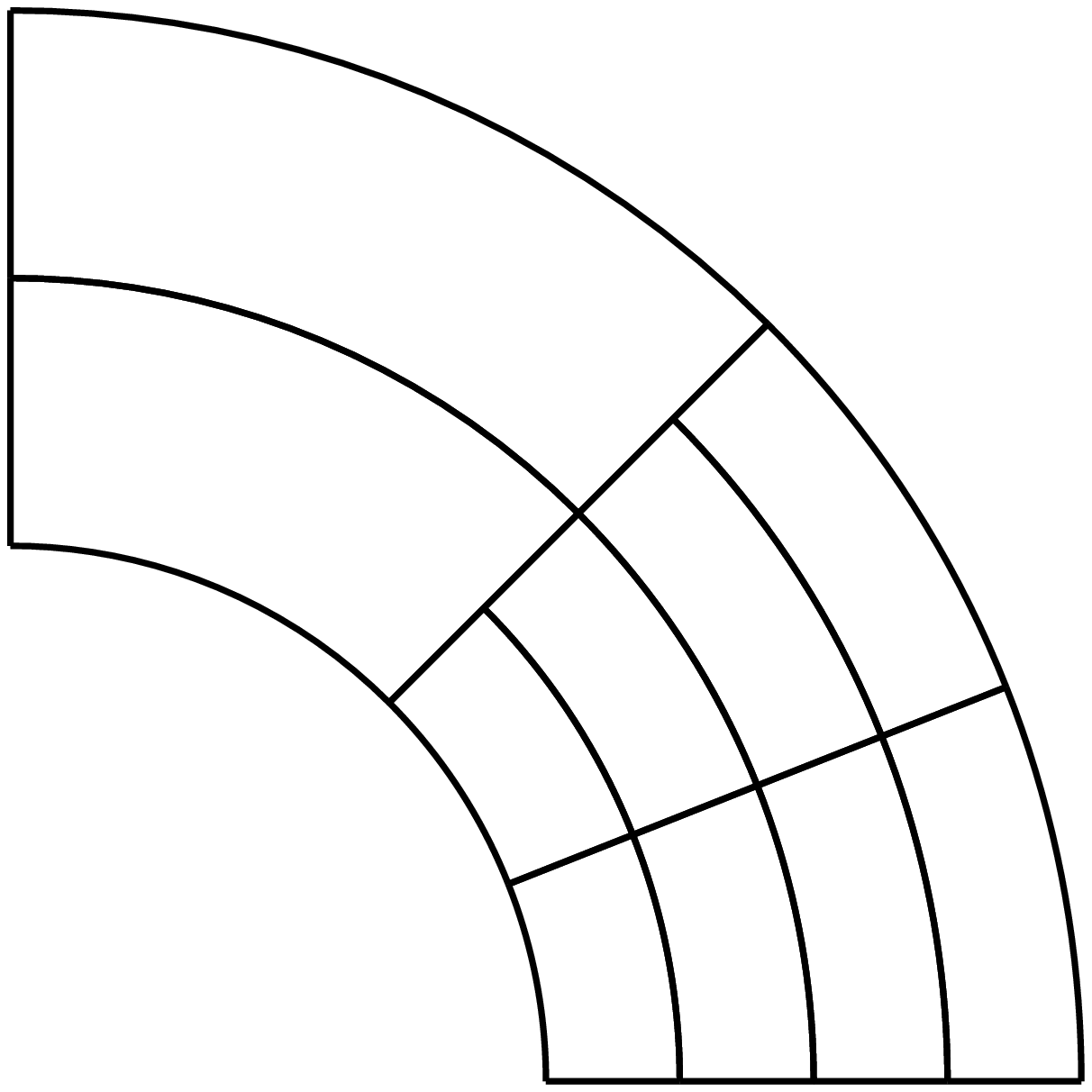}
\includegraphics[width=0.2\textwidth]{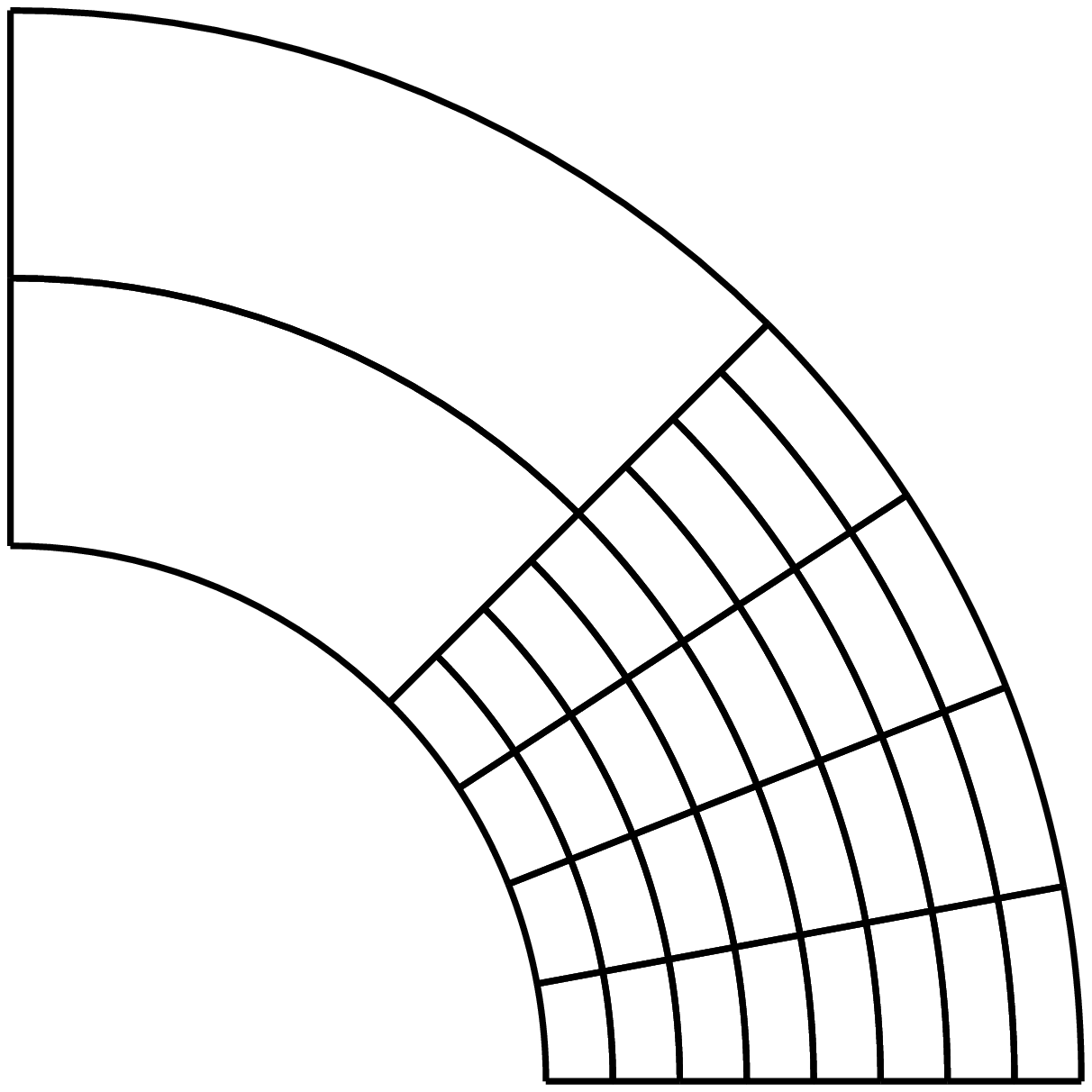}
\includegraphics[width=0.2\textwidth]{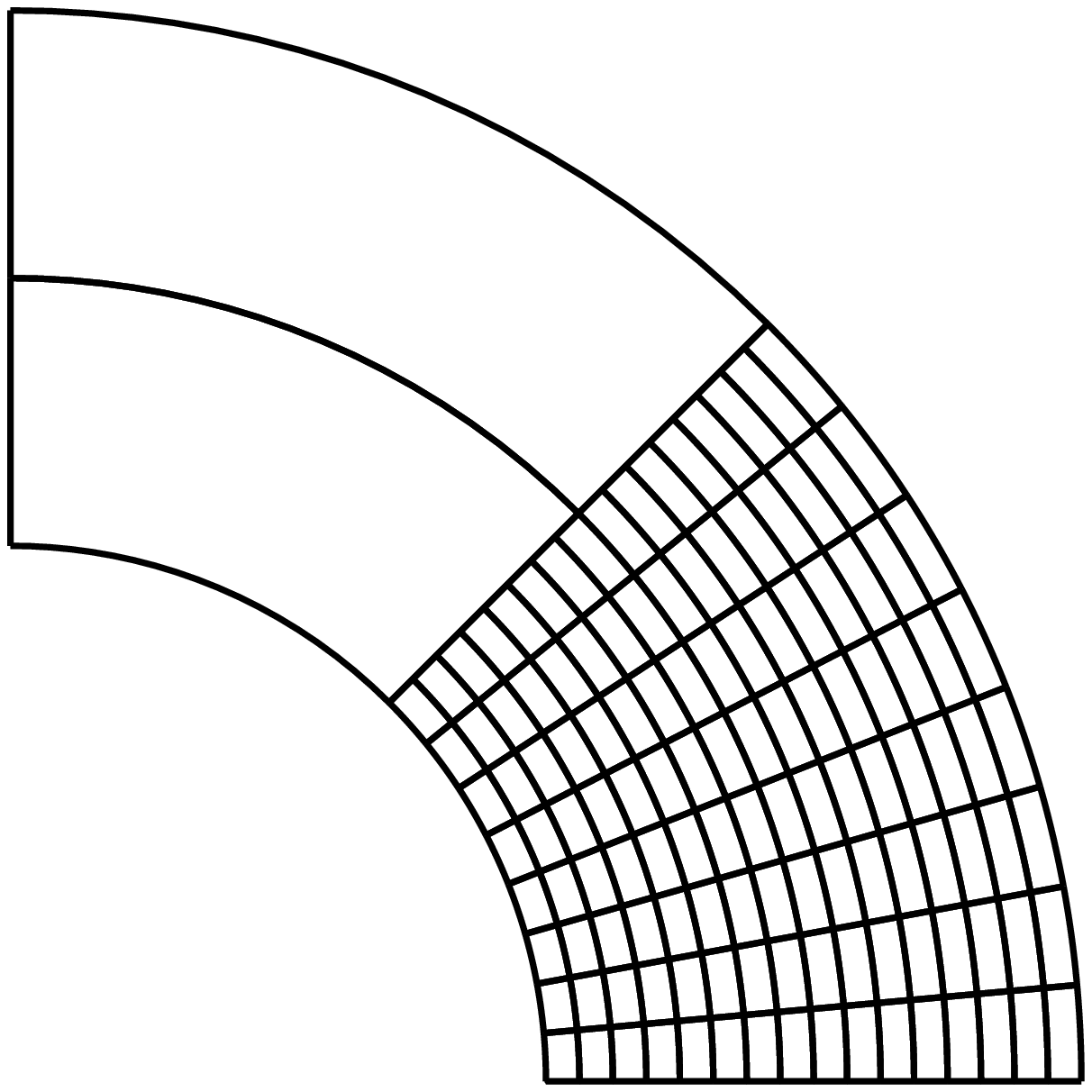}
\\
\vspace{5mm}
\includegraphics[width=0.2\textwidth]{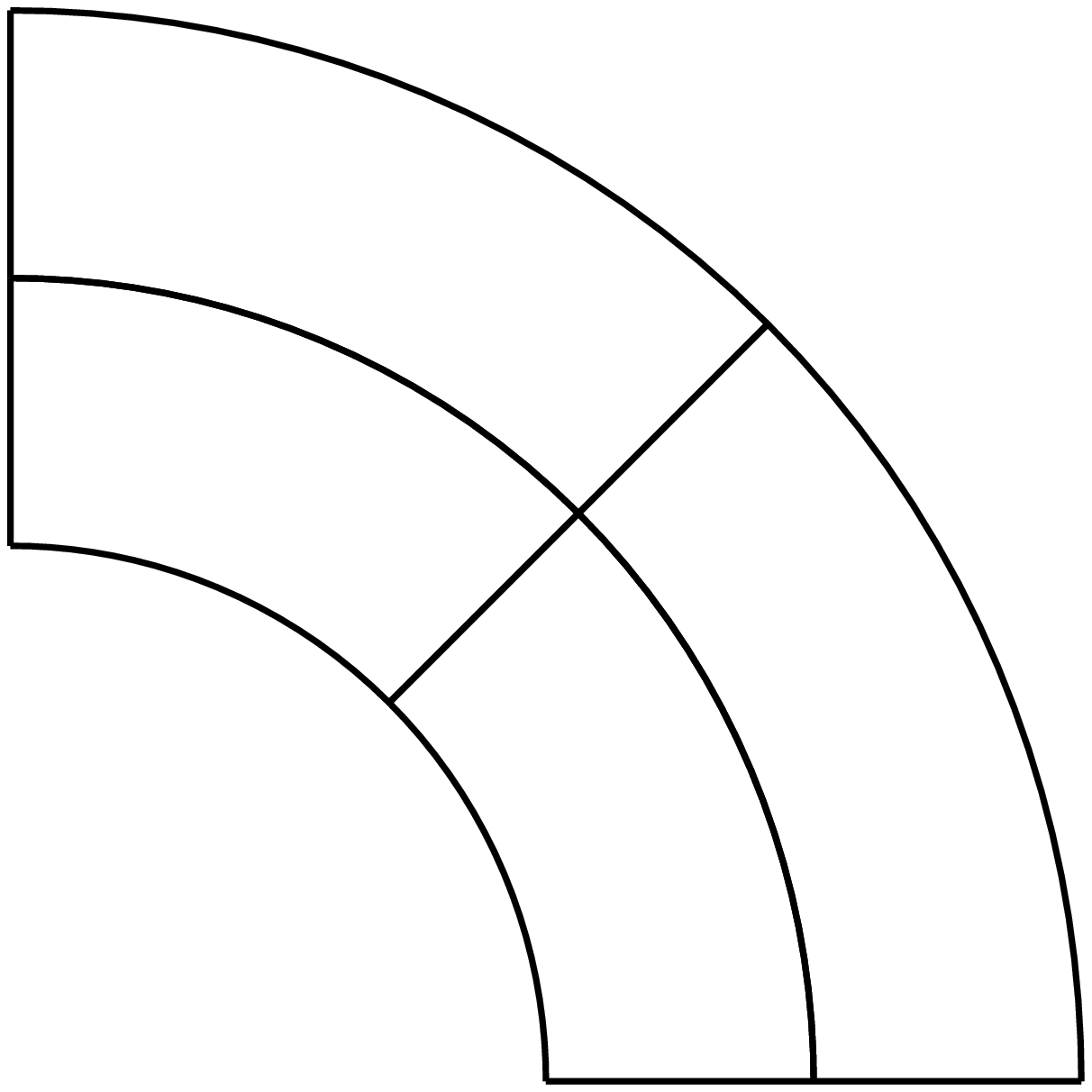}
\includegraphics[width=0.2\textwidth]{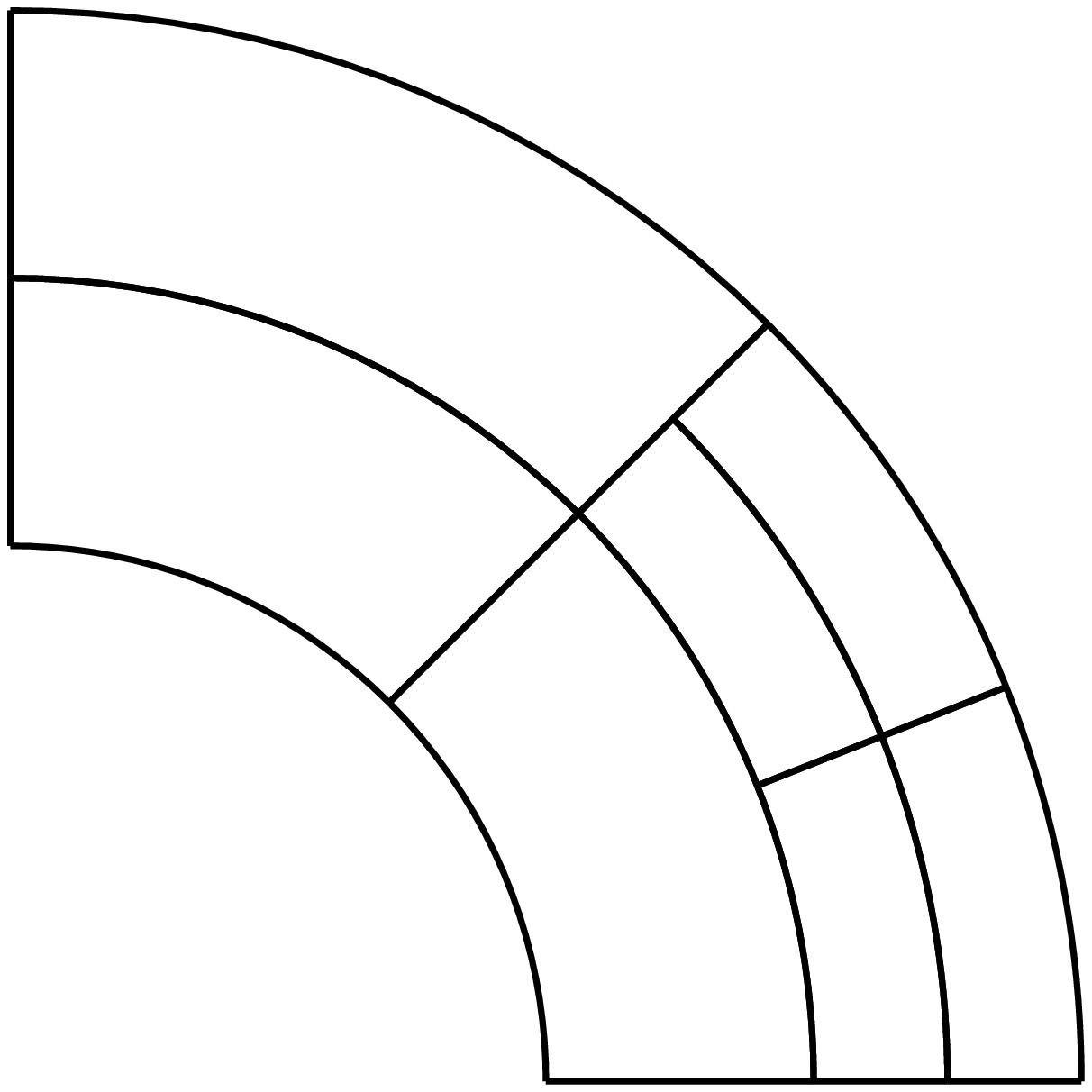}
\includegraphics[width=0.2\textwidth]{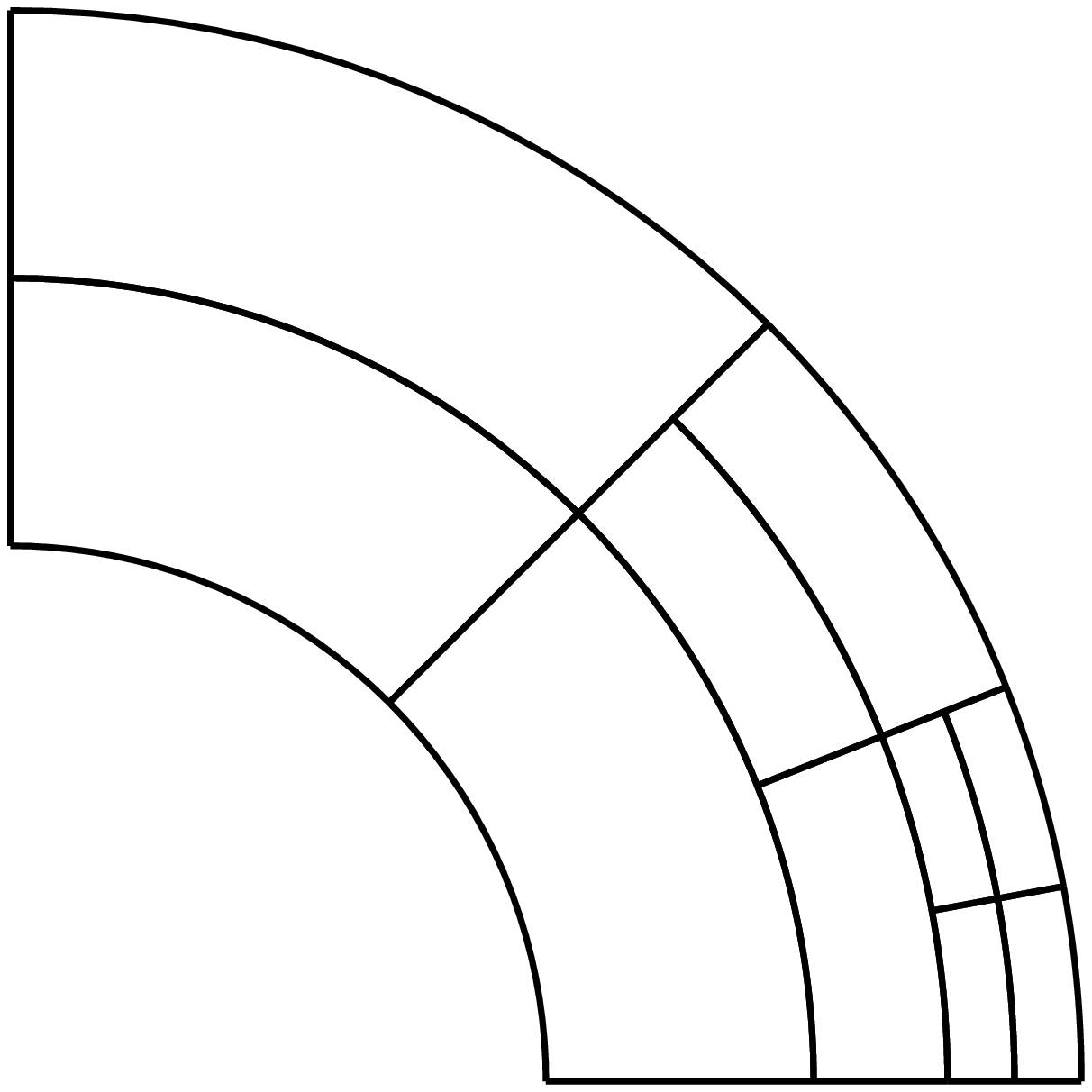}
\includegraphics[width=0.2\textwidth]{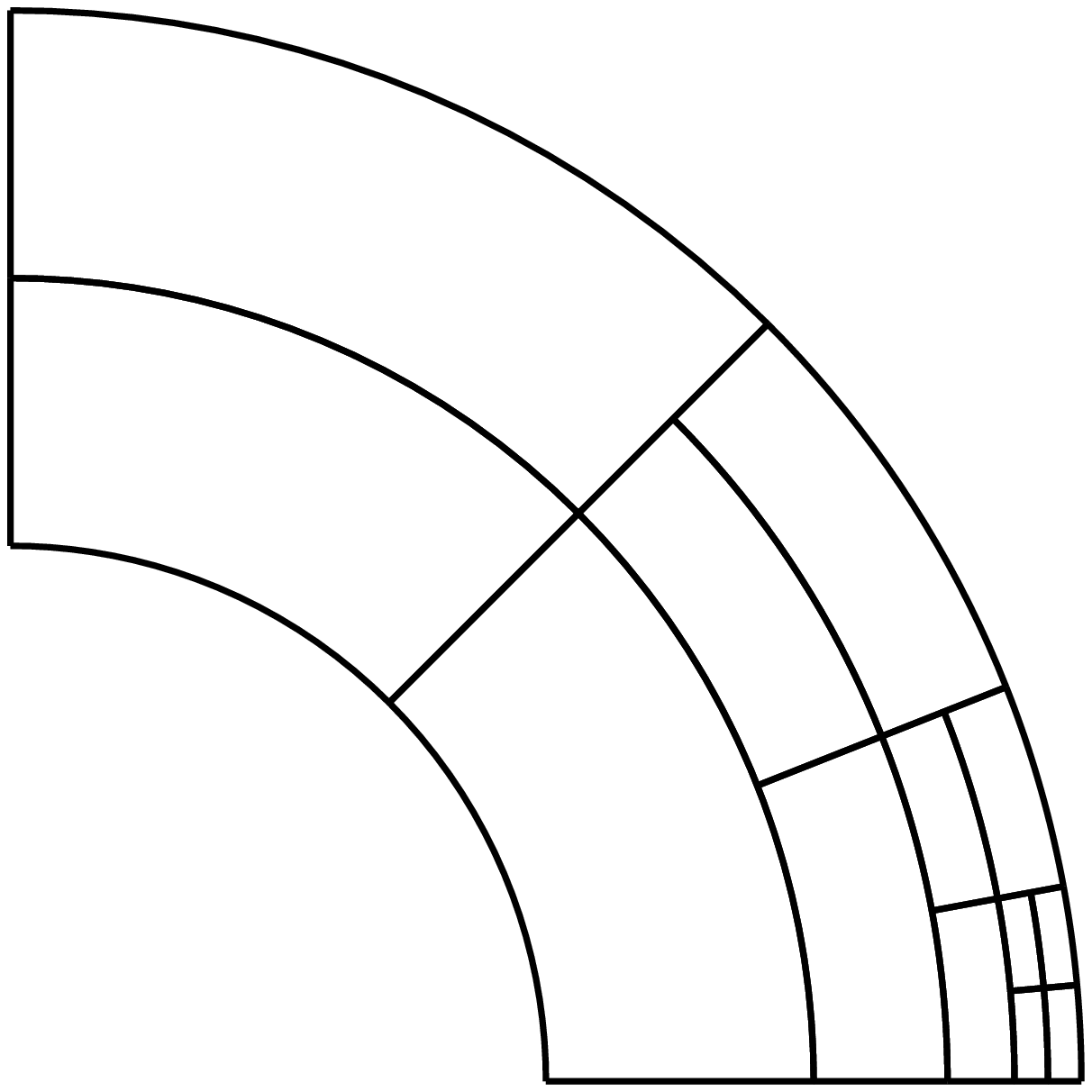}
\end{center}
\caption{\label{fig:qring}
The domain $\Omega$ considered in Section~\ref{sec:numerics} along with the initial mesh $\TT_0$ and the first three refinements: in the right half of $\Omega$ (leading to arbitrary many hanging nodes) (top) and towards the point $(1,0)$ (violating the bounded patch overlap condition~\eqref{eq:overlap}) (bottom).}
\end{figure}

We consider four different refinements of the initial mesh:
1) uniform refinement, where in each step, all elements in the parameter domain are bisected in both directions; 2) adaptive refinement, where in each step, a minimal set of elements $\mathcal{M}_\coarse\subseteq\TT_\coarse$ is marked via the D\"orfler marking
\begin{align*}
 \theta \sum_{T\in\TT_\coarse} \big(\norm{\ssigma_\coarse+\nabla u_\coarse}{T}^2 +  \osc_\coarse^{\rm rel}(T)^2\big)
 \le \sum_{T\in\mathcal{M}_\coarse} \big(\norm{\ssigma_\coarse+\nabla u_\coarse}{T}^2 +  \osc_\coarse^{\rm rel}(T)^2\big)
\end{align*}
with $\theta=0.5$ and subsequently refined via the refinement strategy from~\cite{ghp17} (see also Remark~\ref{rem:refinement});
3) artificial refinement enforcing an arbitrary number of hanging nodes, where in each step, all elements in the parameter domain that are contained in $[0,1/2]\times[0,1]$ are bisected in both directions (see Figure~\ref{fig:qring}, top);
4) artificial refinement enforcing an arbitrary number of overlapping patches $\omega_\ver$, where in each step, the element in the parameter domain containing the point $(0,0)$, which is mapped onto $(1,0)$ under $\F$, is bisected in both directions (see Figure~\ref{fig:qring}, bottom).
In each case, new knots have multiplicity $m$.

The resulting effectivity indices $\frac{\norm{\ssigma_\coarse+\nabla u_\coarse}{\Omega}}{\norm{\nabla(u-u_\coarse)}{\Omega}}$ and $\frac{\norm{\ssigma_\coarse+\nabla u_\coarse}{\Omega} + \osc_\coarse^{\rm rel}}{\norm{\nabla(u-u_\coarse)}{\Omega}}$ as function of the number of mesh elements $N$ in $\TT_\coarse$ are displayed in Figures~\ref{fig:uniform}--\ref{fig:local}.
Recall from Propositions~\ref{prop:efficient} and~\ref{prop:efficiency} that the efficiency constant in~\eqref{eq:efficiency} may theoretically depend on the space dimension $d$, the polynomial degree $\overline p$ (which itself depends on the considered smoothness, see Remark~\ref{rem_p_rob}), $\max\{\norm{D\F}{\infty,\widehat\Omega},$ $\norm{(D\F)^{-1}}{\infty,\widehat\Omega}\}$, and the overlap constant $C_{\rm over}$ from~\eqref{eq:overlap} (which itself only depends on $\overline p$ for the first three refinements types, see Remark~\ref{rem:refinement}, but grows unboundedly in the fourth case).
At least in this example, though, the dependence on $\overline p$ is \emph{not} observed, and the equilibrated flux estimator $\norm{\ssigma_\coarse+\nabla u_\coarse}{\Omega}$ (plus oscillation terms) seems to be not only, as proven, robust with respect to the polynomial degrees $p$ and $\widetilde p$, but also with respect to the smoothness incarnated in $\overline p$.
The increase of the effectivity indices on adaptively refined meshes with $\widetilde p = p + 1$ in the right part of Figure~\ref{fig:adaptive} is only because of data oscillation, as we discuss below.
Moreover, as theoretically shown in Proposition~\ref{prop:efficiency}, we also numerically observe in Figure~\ref{fig:hanging} the robustness with respect to the strength of the hierarchical refinement (number of hanging nodes).
For the fourth refinement, \eqref{eq:overlap} is not uniformly satisfied, which, however, is again not reflected in the resulting efficiency constants.
Note however that the computation of the equilibrated flux becomes very expensive in this case, as there are nodes $\ver$ whose patches $\omega_\ver$ coincide with the entire computational domain $\Omega$, so that the corresponding meshes $\TT_\ver$ are uniform refinements of $\TT_0$ up to the same level as the element containing the point $(1,0)$.
Additional numerical experiments were carried out for an exact solution $u(x,y)=\sin(2\pi x)\sin(2\pi y)$ on the square (not displayed), with similar results.

The ``data oscillation'' terms $\osc_\coarse^{\rm rel}$ from~\eqref{eq_osc} are displayed separately in Figure~\ref{fig:oscs}, again as function of the number of mesh elements $N$ in $\TT_\coarse$. For the present smooth solution~\eqref{eq_sol}, one expects the error $\norm{\nabla(u-u_\coarse)}{\Omega}$ to decay as $\mathcal{O}(h^{p}) \approx \mathcal{O}(N^{-p/2})$ for uniform mesh refinement. Recall from Remark~\ref{rem:reliable oscillations} that $\osc_\coarse^{\rm rel}$ are expected to decay in this case as $\mathcal{O}(h^{\widetilde p+2})$, i.e., as $\mathcal{O}(N^{-p/2-3/2})$ or $\mathcal{O}(N^{-p/2-2})$ for respectively $\widetilde p = p+1$ and $\widetilde p = p+2$. Figure~\ref{fig:oscs} only concerns adaptive mesh refinement, where $\norm{\nabla(u-u_\coarse)}{\Omega}$ is still expected to decay as $\mathcal{O}(N^{-p/2})$. We do not have here theoretical indications for $\osc_\coarse^{\rm rel}$, but we observe at least $\mathcal{O}(N^{-p/2})$ in most cases, even though the chosen $\widetilde p$ is theoretically inappropriate for higher smoothness as discussed above.
For uniform mesh refinement (not displayed), we indeed observe $\mathcal{O}(h^{\widetilde p+2})$, following Remark~\ref{rem:reliable oscillations}.

\begin{figure}[t]
\begin{center}
\includegraphics[width=0.49\textwidth]{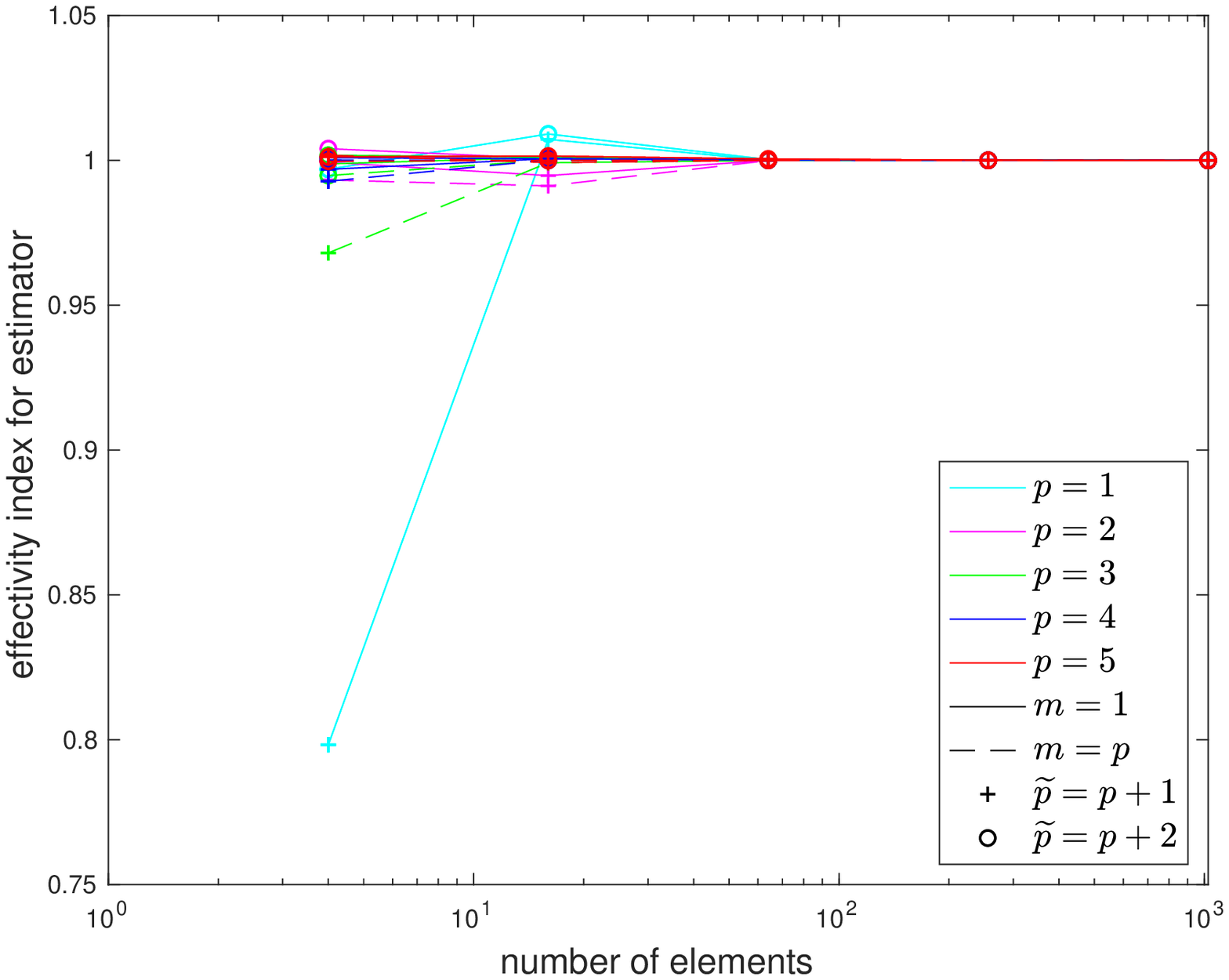}
\includegraphics[width=0.49\textwidth]{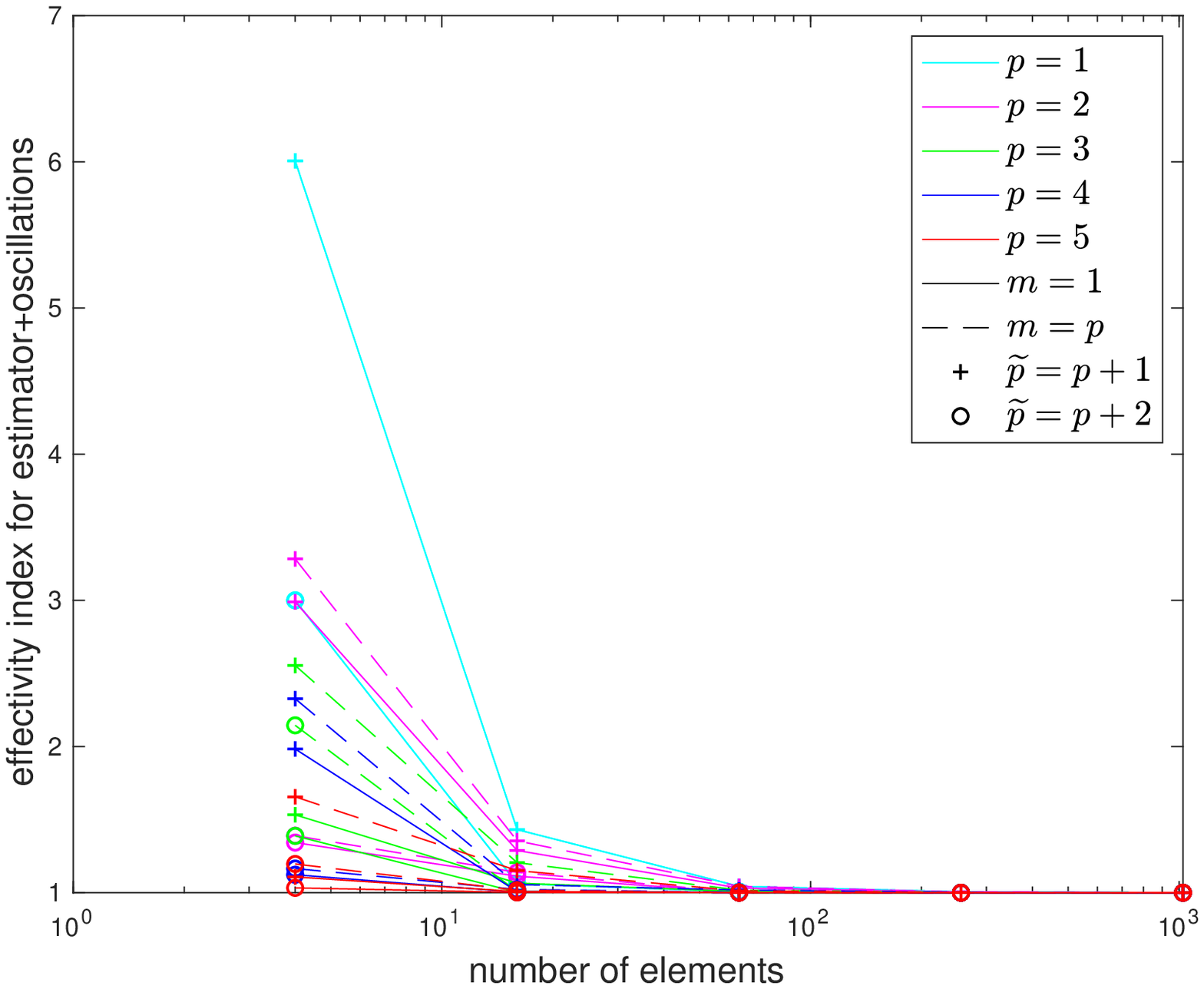}
\end{center}
\caption{\label{fig:uniform}
Effectivity indices $\frac{\norm{\ssigma_\coarse+\nabla u_\coarse}{\Omega}}{\norm{\nabla(u-u_\coarse)}{\Omega}}$ (left) and $\frac{\norm{\ssigma_\coarse+\nabla u_\coarse}{\Omega} + \osc_\coarse^{\rm rel}}{\norm{\nabla(u-u_\coarse)}{\Omega}}$  (right) corresponding to the problem of Section~\ref{sec:numerics} with uniform mesh refinement, polynomial degrees $p\in\{1,\dots,5\}$, multiplicities $m\in\{1,p\}$.}
\end{figure}

\begin{figure}[t]
\begin{center}
\includegraphics[width=0.49\textwidth]{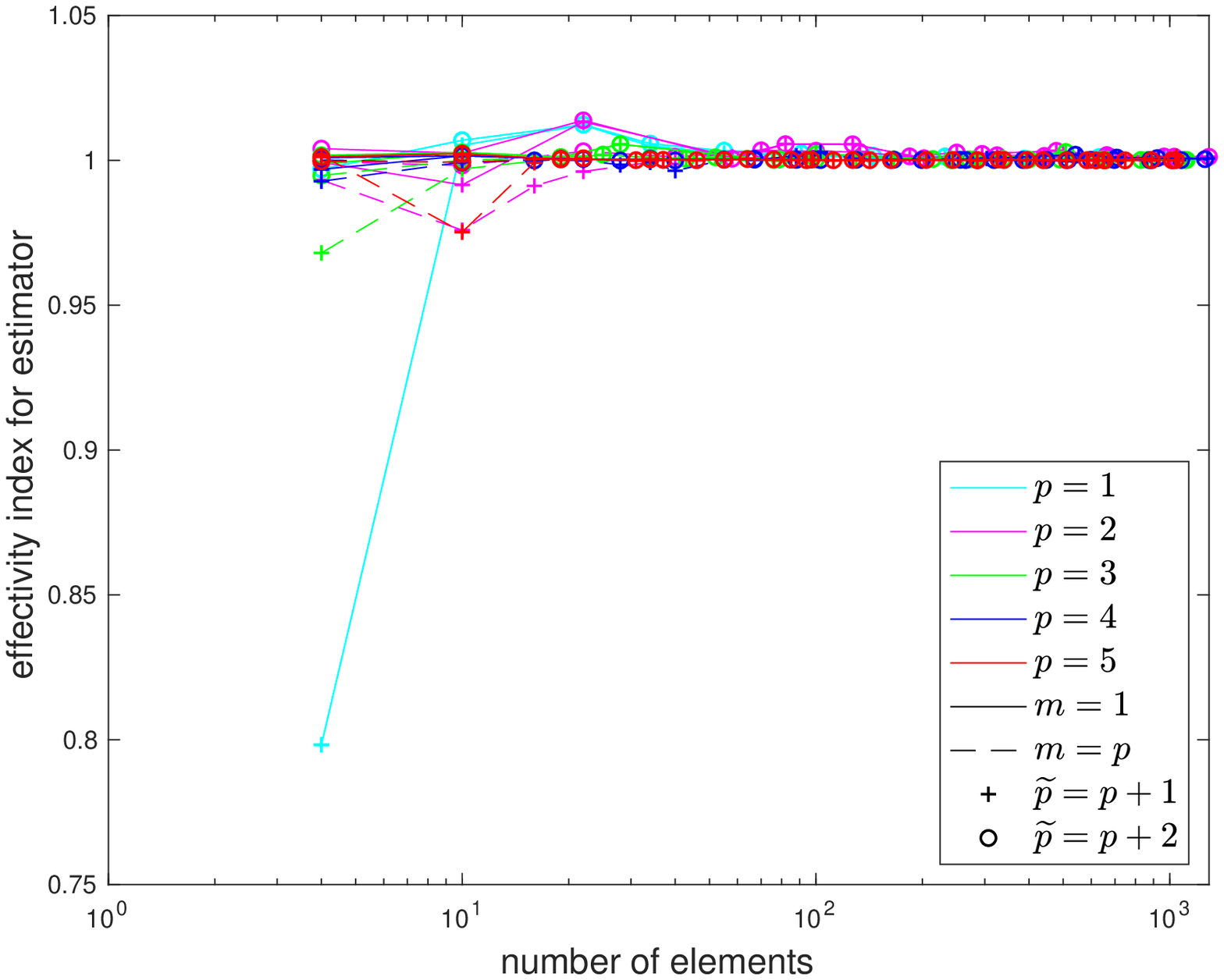}
\includegraphics[width=0.49\textwidth]{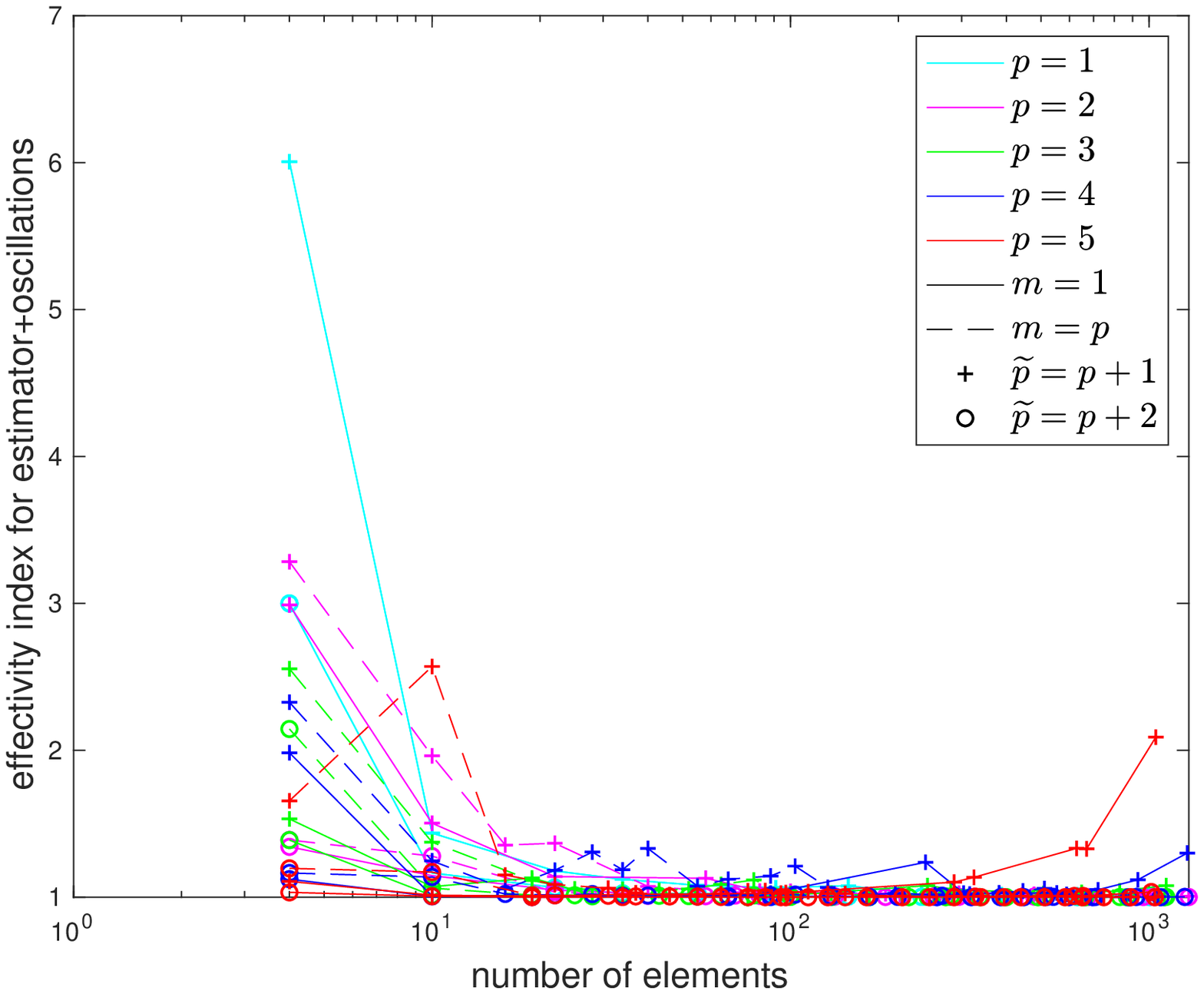}
\end{center}
\caption{\label{fig:adaptive}
Effectivity indices $\frac{\norm{\ssigma_\coarse+\nabla u_\coarse}{\Omega}}{\norm{\nabla(u-u_\coarse)}{\Omega}}$ (left) and $\frac{\norm{\ssigma_\coarse+\nabla u_\coarse}{\Omega} + \osc_\coarse^{\rm rel}}{\norm{\nabla(u-u_\coarse)}{\Omega}}$ (right) corresponding to the problem of Section~\ref{sec:numerics} with adaptive mesh refinement, polynomial degrees $p\in\{1,\dots,5\}$, multiplicities $m\in\{1,p\}$.}
\end{figure}

\begin{figure}[t]
\begin{center}
\includegraphics[width=0.49\textwidth]{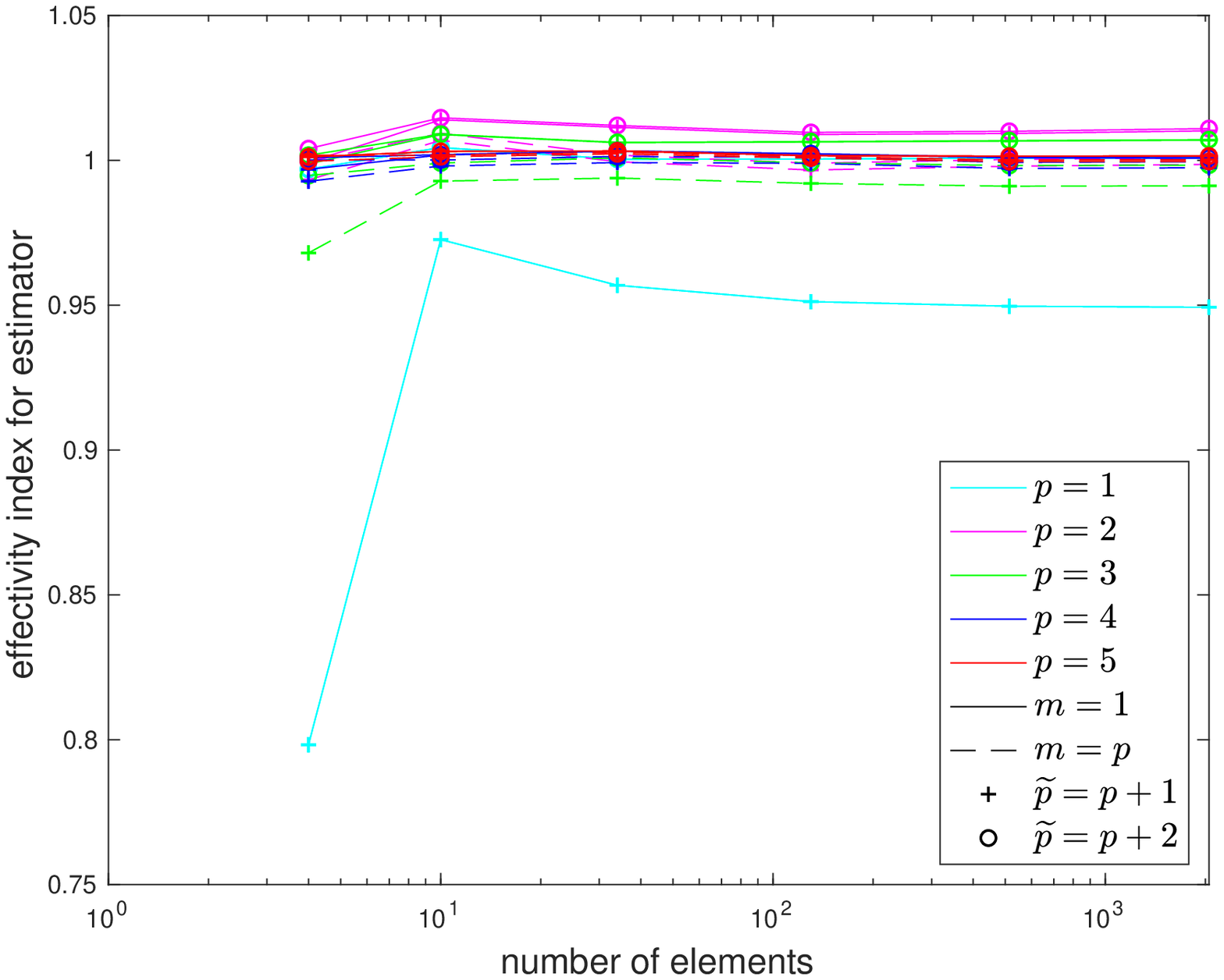}
\includegraphics[width=0.49\textwidth]{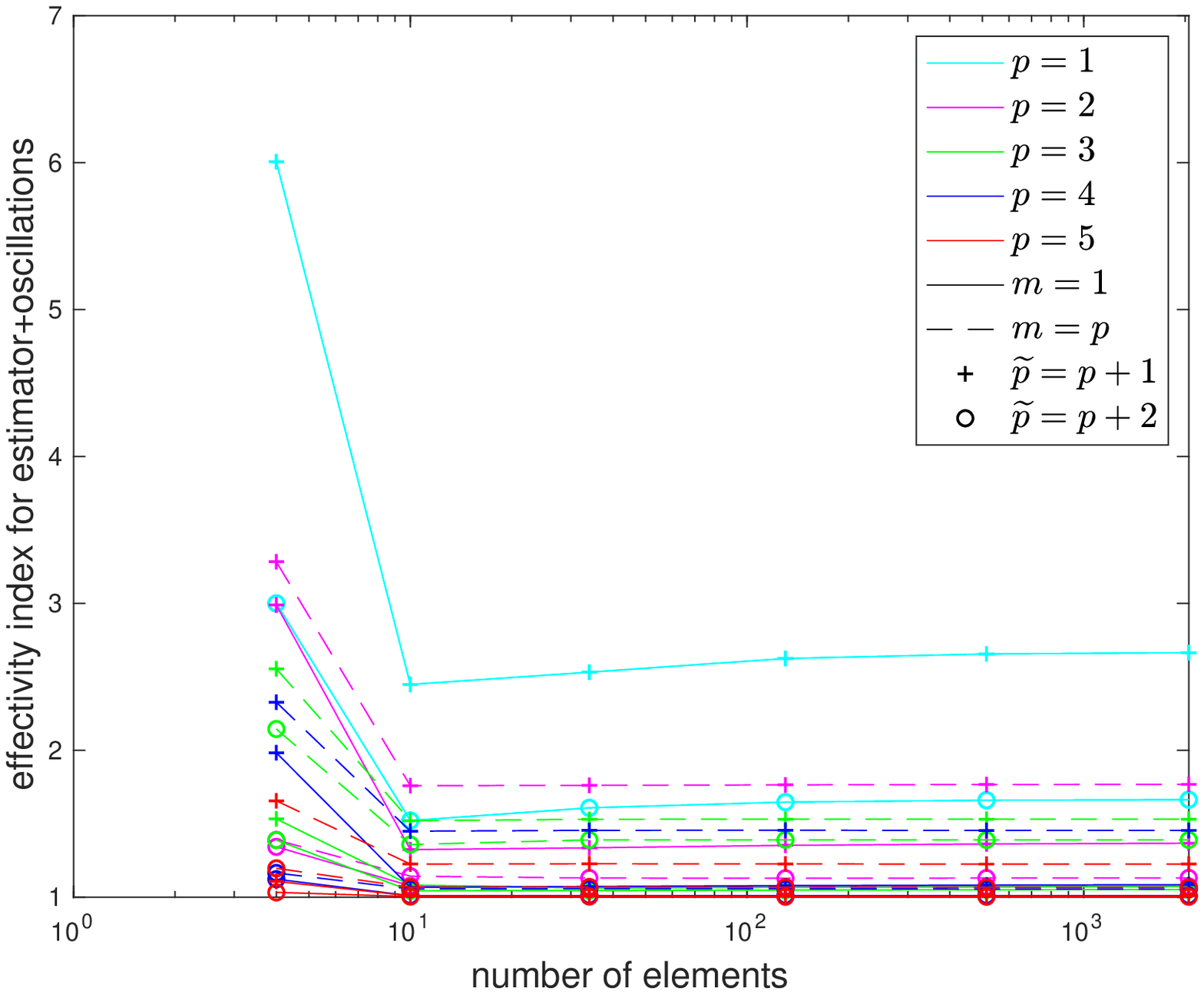}
\end{center}
\caption{\label{fig:hanging}
Effectivity indices $\frac{\norm{\ssigma_\coarse+\nabla u_\coarse}{\Omega}}{\norm{\nabla(u-u_\coarse)}{\Omega}}$ (left) and $\frac{\norm{\ssigma_\coarse+\nabla u_\coarse}{\Omega} + \osc_\coarse^{\rm rel}}{\norm{\nabla(u-u_\coarse)}{\Omega}}$  (right) corresponding to the problem of Section~\ref{sec:numerics} with artificial refinement enforcing a high number of hanging nodes (see Figure~\ref{fig:qring}, top), polynomial degrees $p\in\{1,\dots,5\}$, multiplicities $m\in\{1,p\}$.}
\end{figure}

\begin{figure}[t]
\begin{center}
\includegraphics[width=0.49\textwidth]{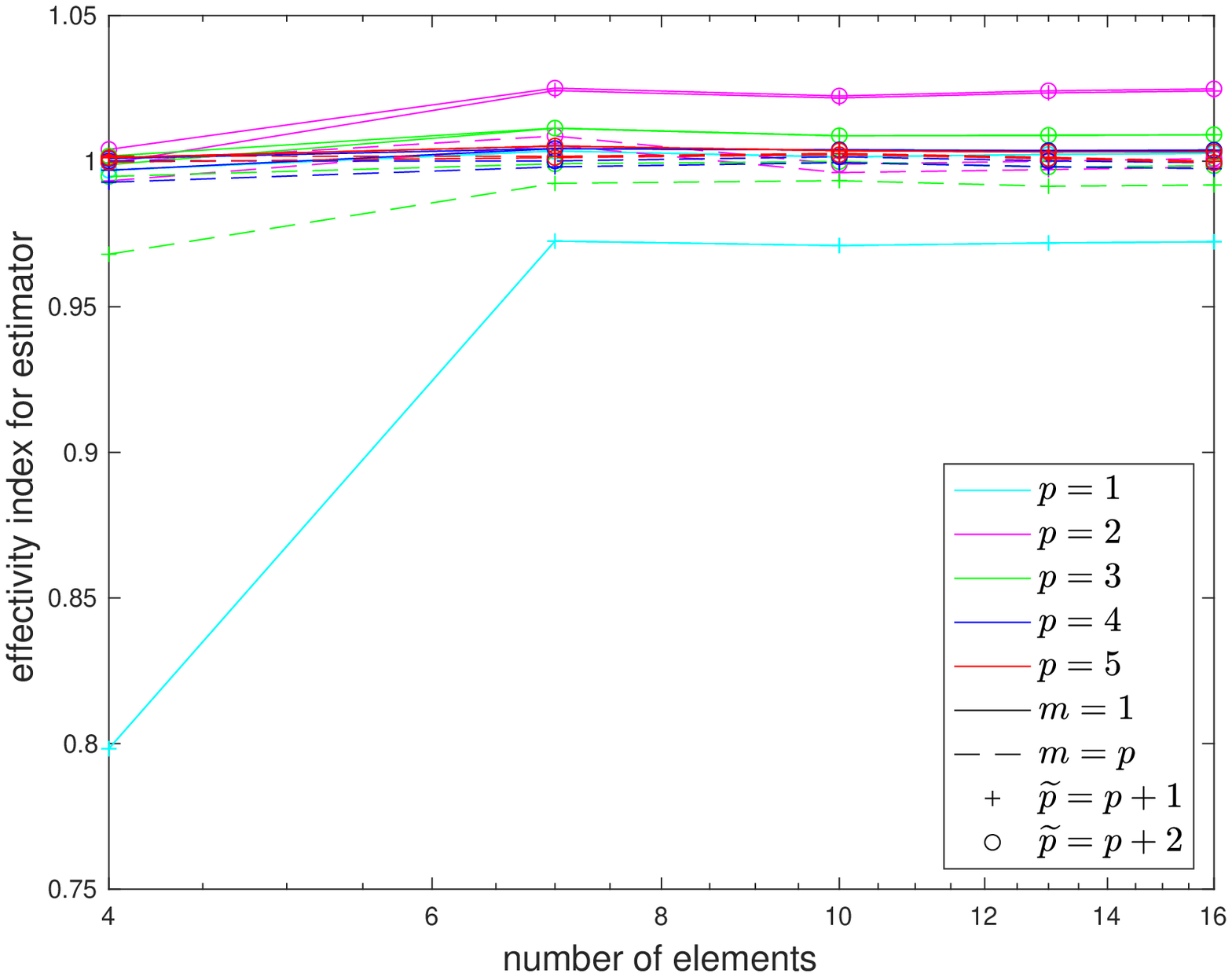}
\includegraphics[width=0.49\textwidth]{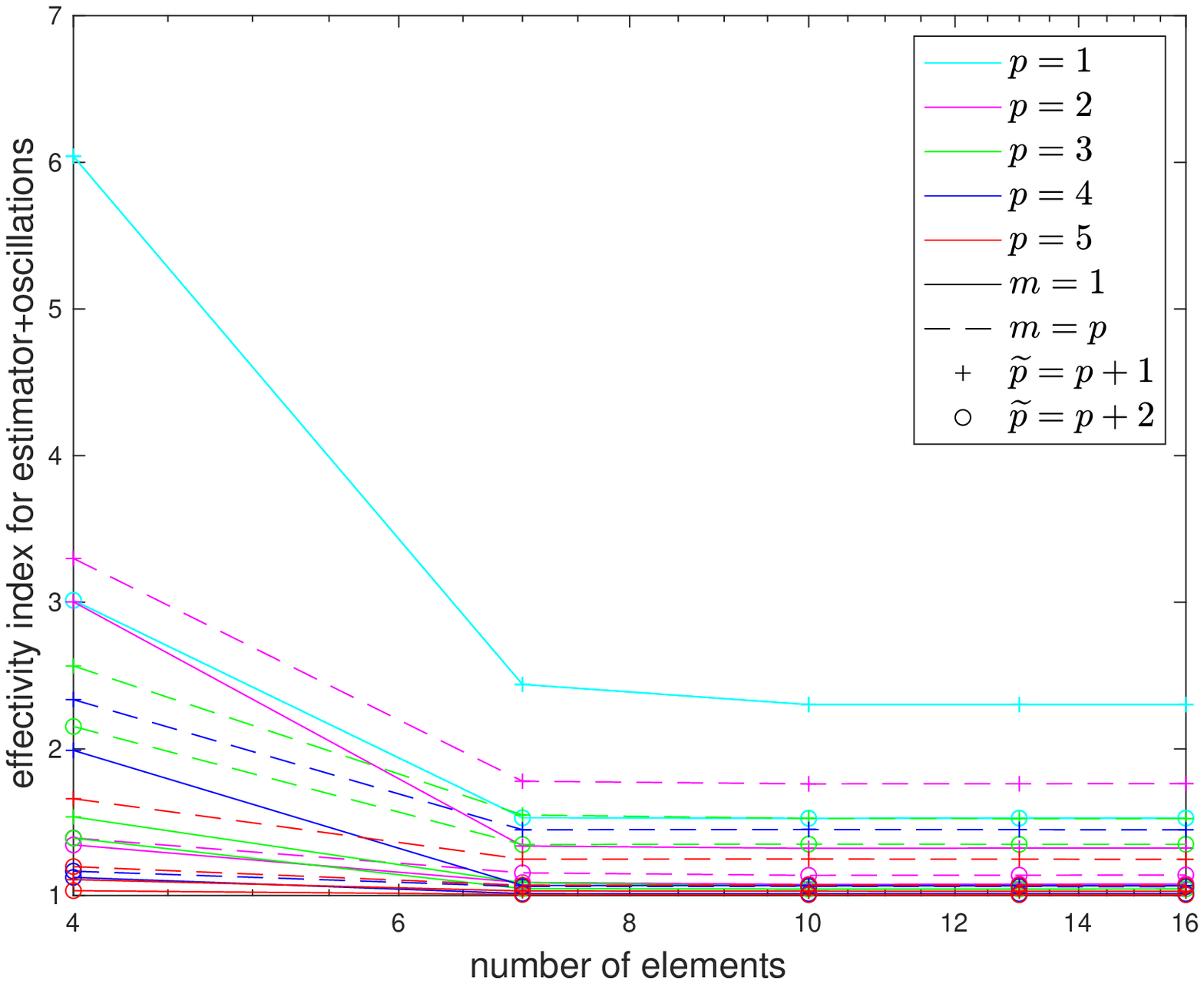}
\end{center}
\caption{\label{fig:local}
Effectivity indices $\frac{\norm{\ssigma_\coarse+\nabla u_\coarse}{\Omega}}{\norm{\nabla(u-u_\coarse)}{\Omega}}$ (left) and $\frac{\norm{\ssigma_\coarse+\nabla u_\coarse}{\Omega} + \osc_\coarse^{\rm rel}}{\norm{\nabla(u-u_\coarse)}{\Omega}}$  (right) corresponding to the problem of Section~\ref{sec:numerics} with artificial refinement enforcing a high number of overlapping patches (see Figure~\ref{fig:qring}, bottom), polynomial degrees $p\in\{1,\dots,5\}$, multiplicities $m\in\{1,p\}$.}
\end{figure}

\begin{figure}[t]
\begin{center}
\includegraphics[width=0.49\textwidth]{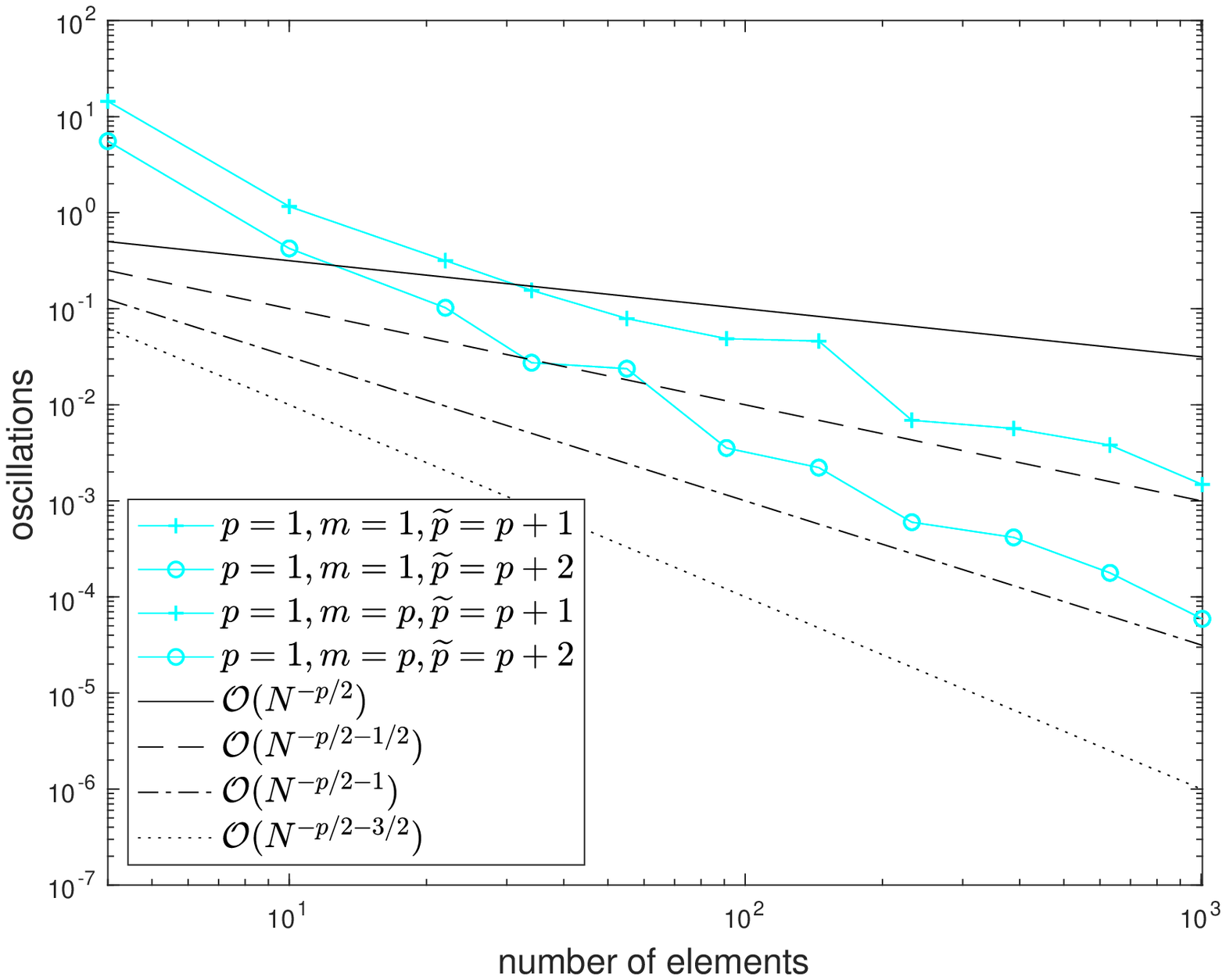}
\includegraphics[width=0.49\textwidth]{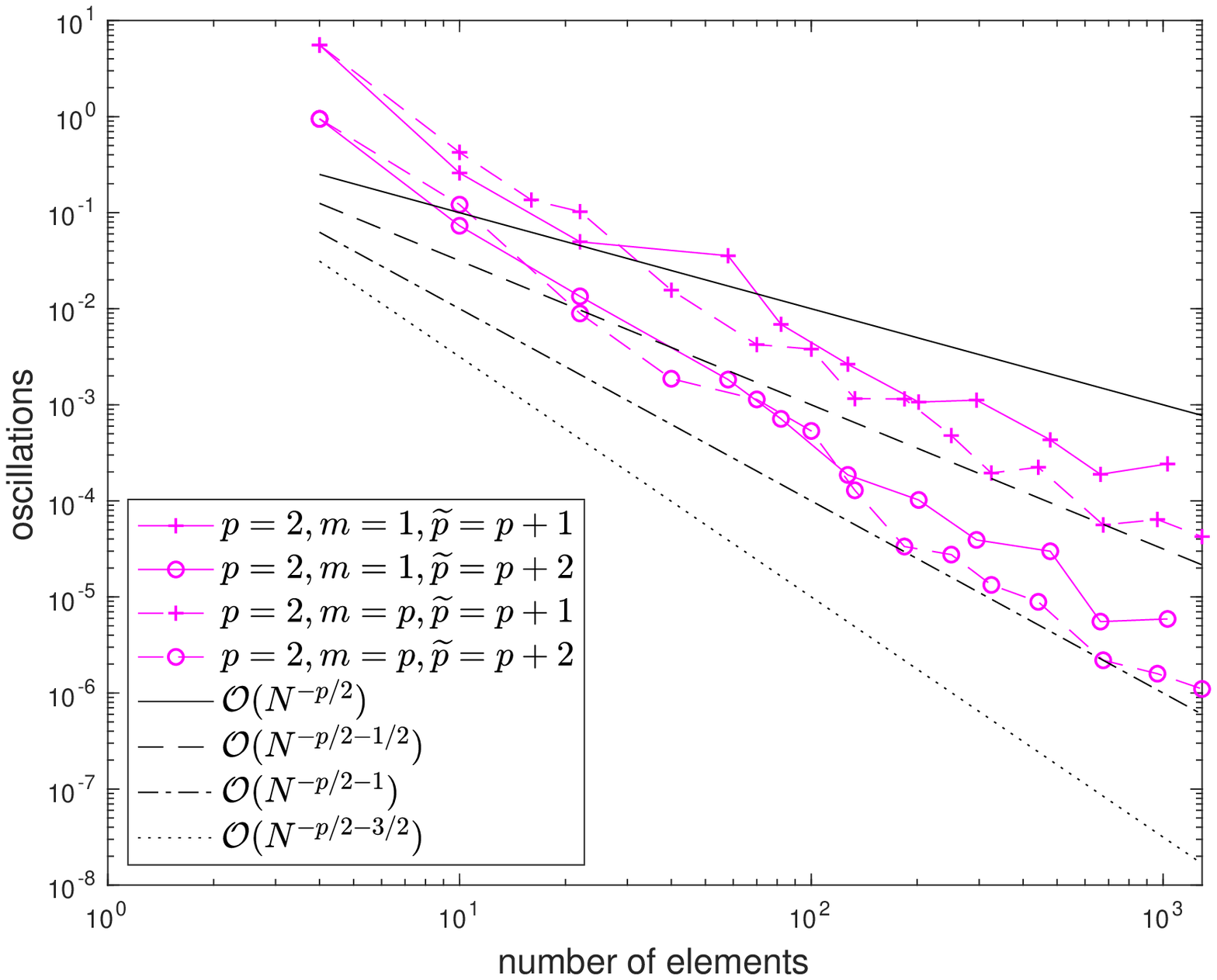}
\includegraphics[width=0.49\textwidth]{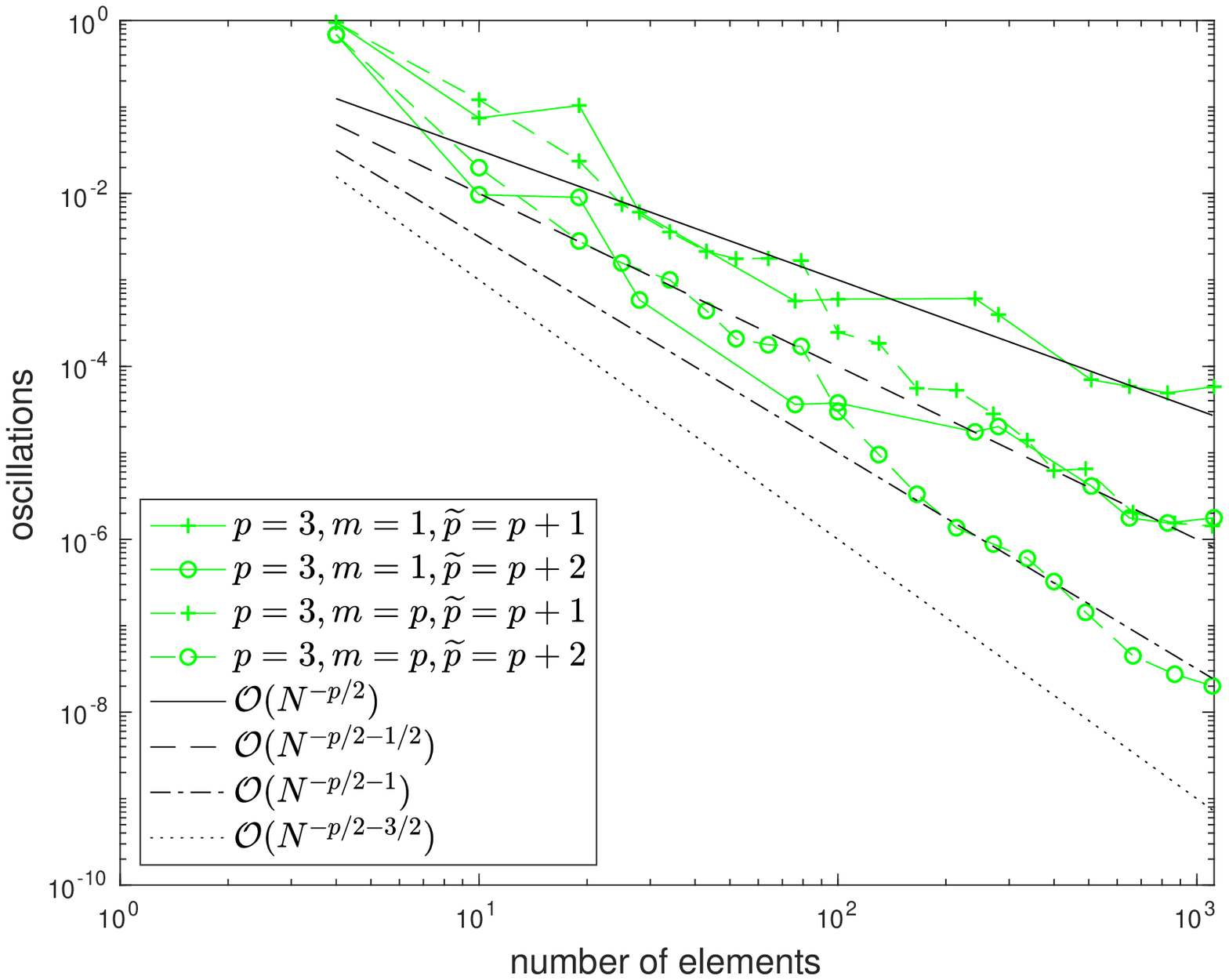}
\includegraphics[width=0.49\textwidth]{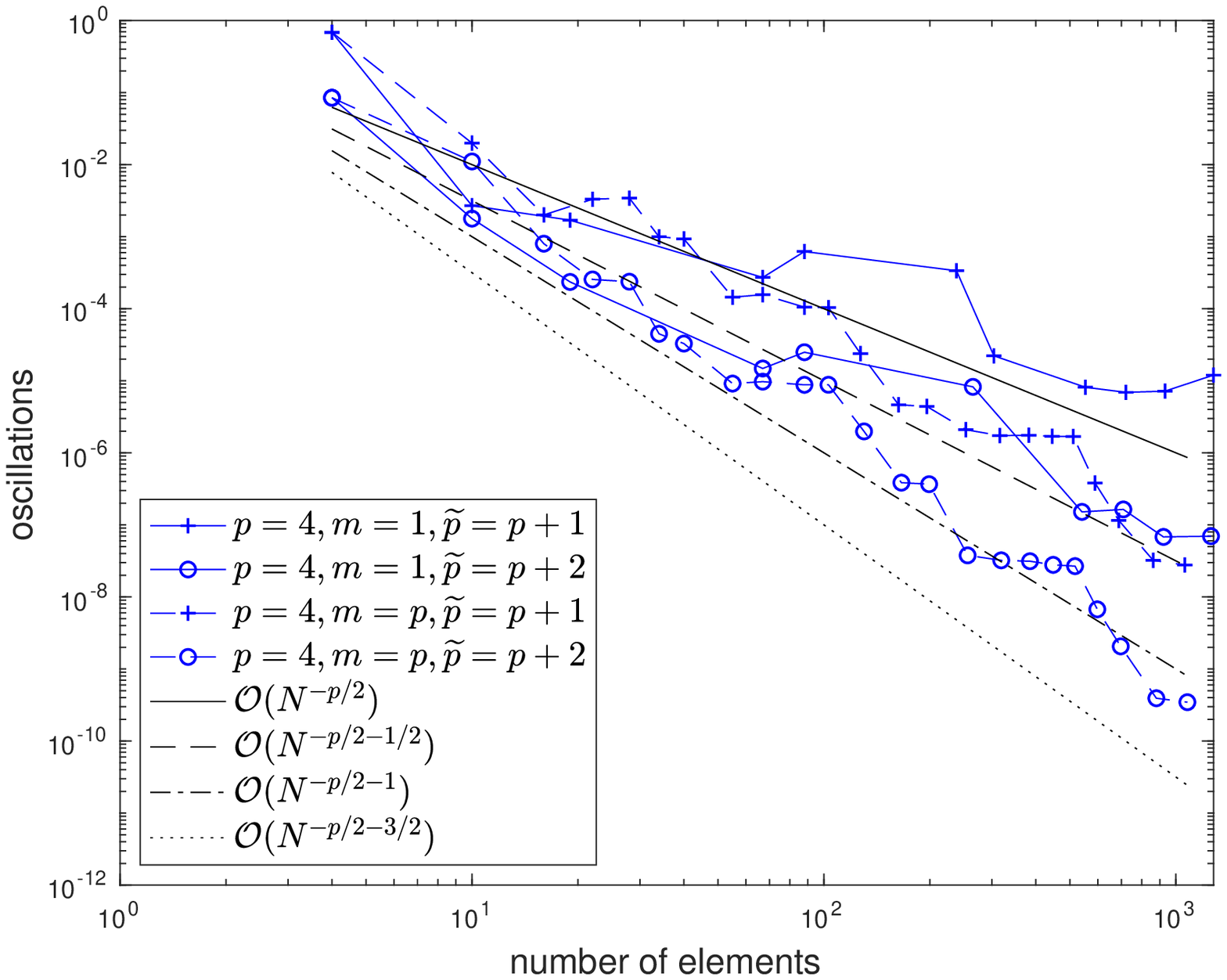}
\includegraphics[width=0.49\textwidth]{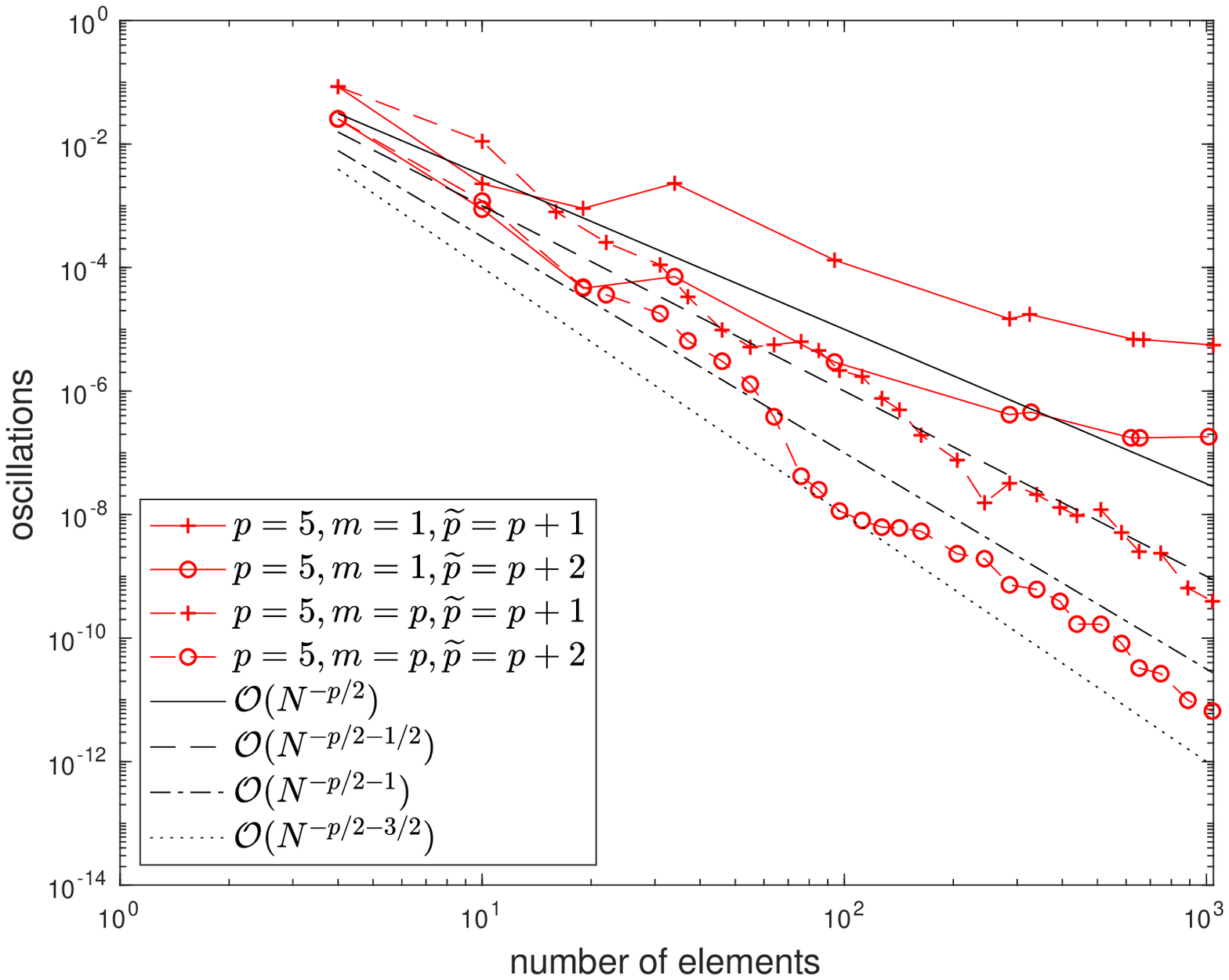}
\end{center}
\caption{\label{fig:oscs}
Oscillation terms $\osc_\coarse^{\rm rel}$ from~\eqref{eq_osc} corresponding to the problem of Section~\ref{sec:numerics} with adaptive mesh refinement, polynomial degrees $p\in\{1,\dots,5\}$, multiplicities $m\in\{1,p\}$.}
\end{figure}

\appendix
%%%%%%%%%%%%%%%%%%%%%%%%%%%%%%%%%%%%%%%%%%%%%%%%%%%%%%%%%%%%%%%%%%%%%%%%%%%%%%%%%%%%%%
%%%%%%%%%%%%%%%%%%%%%%%%%%%%%%%%%%%%%%%%%%%%%%%%%%%%%%%%%%%%%%%%%%%%%%%%%%%%%%%%%%%%%%
\section{Proof of the broken polynomial extension property}\label{sec:extension}
%%%%%%%%%%%%%%%%%%%%%%%%%%%%%%%%%%%%%%%%%%%%%%%%%%%%%%%%%%%%%%%%%%%%%%%%%%%%%%%%%%%%%%
%%%%%%%%%%%%%%%%%%%%%%%%%%%%%%%%%%%%%%%%%%%%%%%%%%%%%%%%%%%%%%%%%%%%%%%%%%%%%%%%%%%%%%

In this section, we verify Assumption~\ref{ass:min RT to cont} for the spaces defined in Section~\ref{sec:discrete spaces}.
Recall the Piola transformations from~\eqref{eq:piola1} and~\eqref{eq:piola2}
\begin{align*}
 \bPhi(\cdot)\eq \big(\det (D\F)^{-1}(D\F)(\cdot)\big)\circ\F^{-1}
 \quad\text{and}\quad
 \widetilde \Phi(\cdot)\eq\big(\det (D\F)^{-1} (\cdot)\big)\circ \F^{-1}
\end{align*}
and the identity~\eqref{eq:Piola_eq}
\begin{align}\label{eq:piola identity}
\widetilde\Phi\big(\nabla{\cdot}(\cdot)\big) = \Dv\bPhi(\cdot);
\end{align}
see, e.g., \cite[Lemma~9.6]{eg21}.
Following the imposition of the Neumann boundary condition in the space $\H_0(\div,\omega_\vertt)$ in~\eqref{eq:H0 div for b}, we denote by $\EE_\vertt^{\rm N}$ the boundary faces of the local mesh $\TT_\vertt$ if $\psi_\ver\psi_\vertt\in H_0^1(\Omega)$ and such boundary faces of $\TT_\vertt$ where $\ee\subset(\psi_\ver\psi_\vertt)^{-1}(\{0\})$ for $\psi_\ver\psi_\vertt\not\in H_0^1(\Omega)$. By $\EE_\vertt^{\rm int}$, we denote the interior faces of $\TT_\vertt$.
Analogously, we write $\widehat\EE_\vertt^{\rm N}$ and $\widehat\EE_\vertt^{\rm int}$ for the corresponding faces on the parameter mesh $\widehat\TT_\vertt$.
Finally, we write $\nabla_\vertt{\cdot}(\cdot)$ for the $\TT_\vertt$- or $\widehat\TT_\vertt$-piecewise divergence operator.

Since $\ttau_\coarse \in \RT_\coarse^{\ver,\vertt} = \RT^{\widetilde\p}(\TT_\vertt)$ and $\V_\coarse^{\ver,\vertt} = \RT^{\widetilde\p}(\TT_\vertt) \cap \H_0(\div,\omega_\vertt)$ by~\eqref{eq:V subset}, we can substitute in~\eqref{eq:min RT to cont} $\bv_h+\ttau_h=\w_h$ with $\w_h=\bPhi(\widehat\w_h)$ and $\widehat\w_h \in \RT^{\widetilde\p}(\widehat\TT_\vertt)$, also using~\eqref{eq_RT_map}.
 This shows that
\begin{align}\label{eq:long substacks}
 \min_{\substack{\bv_\coarse\in\V_\coarse^{\ver,\vertt}\\
 \Dv\bv_\coarse=g_\coarse}}
 \norm{\bv_\coarse + \ttau_\coarse}{\omega_\vertt}
 = \!\!\!
\min_{\substack{\w_h\in \RT^{\widetilde\p}(\TT_\vertt)\\
 \w_h{\cdot}\n_{\omega_\vertt}|_\ee = \ttau_h{\cdot}\n_{\omega_\vertt}|_\ee=:r_\ee  \,\,\,\forall \ee\in\EE_\vertt^{\rm N}\\
 [\w_h{\cdot}\n_{\omega_\vertt}]_\ee = [\ttau_h{\cdot}\n_{\omega_\vertt}]_\ee=:r_\ee \,\,\,\forall \ee\in\EE_\vertt^{\rm int}\\
 \nabla_\vertt{\cdot}\w_h|_T = (g_h + \nabla_\vertt{\cdot}\ttau_h)|_T=:r_T \,\,\, \forall T\in\TT_\vertt}} \!\!\!
 \norm{\w_h}{\omega_\vertt}
 = \!\!\!
  \min_{\substack{\widehat\w_h\in\RT^{\widetilde\p}(\widehat\TT_\vertt)\\
 \bPhi(\widehat\w_h){\cdot}\n_{\omega_\vertt}|_\ee = r_\ee  \,\,\,\forall \ee\in\EE_\vertt^{\rm N}\\
 [\bPhi(\widehat\w_h){\cdot}\n_{\omega_\vertt}]_\ee = r_\ee \,\,\,\forall \ee\in\EE_\vertt^{\rm int}\\
 \nabla_\vertt{\cdot}\bPhi(\widehat\w_h)|_T = r_T \,\,\, \forall T\in\TT_\vertt}} \!\!\!
 \norm{\bPhi(\widehat\w_h)}{\omega_\vertt}.
\end{align}
With $\n_{\widehat\omega_\vertt}$ denoting the outer normal vector on $\partial\widehat\omega_\vertt$, elementary analysis, cf.~\cite[Lemma~9.11]{eg21}, provides the relation
\begin{align*}
 (D\F^{\top}\circ\F^{-1})\, \n_{\omega_\vertt} = \n_{\widehat\omega_\vertt} \circ \F^{-1} |(D\F^{\top}\circ\F^{-1})\, \n_{\omega_\vertt}|
\end{align*}
and thus
\begin{align*}
 \bPhi(\widehat\w_h){\cdot}\n_{\omega_\vertt}
 & =  \Big(\big(\det(D\F)^{-1} \widehat \w_h^{\top}(D\F)^\top\big)\circ\F^{-1} \Big)\n_{\omega_\vertt} \\
 & = \Big( \widehat\w_h{\cdot}\n_{\widehat \omega_\vertt}\, \frac{|(D\F^{\top}\circ\F^{-1})\, \n_{\omega_\vertt}|}{\det(D\F)}\Big) \circ\F^{-1}.
\end{align*}
Hence, with $\widehat \ee\eq\F^{-1}(\ee)$ and $\widehat\ttau_h\eq\bPhi^{-1}(\ttau_h)$, the first equation in the last minimum of~\eqref{eq:long substacks} is equivalent to $\widehat\w_h{\cdot}\n_{\widehat\omega_\vertt}|_{\widehat \ee}=\widehat\ttau_h{\cdot}\n_{\widehat\omega_\vertt}|_{\widehat \ee}=:r_{\widehat \ee}$, and the second equation is equivalent to $[\widehat \w_h{\cdot}\n_{\widehat\omega_\vertt}]|_{\widehat \ee} = [\widehat \ttau_h{\cdot}\n_{\widehat\omega_\vertt}]_{\widehat \ee}=:r_{\widehat \ee}$.
The identity~\eqref{eq:piola identity} shows that the third equation is equivalent to $\nabla_{\vertt}{\cdot}\widehat\w_h|_{\widehat T} = \widetilde \Phi^{-1}(g_h +\nabla_{\vertt}{\cdot}\ttau_h)=:r_{\widehat T}$.
As $\norm{\bPhi(\widehat\w_h)}{\omega_\vertt} \ls \norm{\widehat \w_h}{\widehat\omega_\vertt}$ (with a hidden constant depending only on $\norm{D\F}{\infty,\widehat\Omega}$), we can  formulate~\eqref{eq:long substacks} in the parameter domain
\begin{align}\label{eq_Piola_equiv}
 \min_{\substack{\bv_\coarse\in\V_\coarse^{\ver,\vertt}\\
 \Dv\bv_\coarse=g_\coarse}}
 \norm{\bv_\coarse + \ttau_\coarse}{\omega_\vertt}
 \ls
 \min_{\substack{\widehat\w_h\in\RT^{\widetilde\p}(\widehat\TT_\vertt)\\
 \widehat\w_h{\cdot}\n_{\widehat\omega_\vertt}|_{\widehat \ee} = r_{\widehat \ee}  \,\,\,\forall \widehat \ee\in\widehat\EE_\vertt^{\rm N}\\
 [\widehat\w_h{\cdot}\n_{\widehat\omega_\vertt}]_{\widehat \ee} = r_{\widehat \ee} \,\,\,\forall \widehat \ee\in\widehat\EE_\vertt^{\rm int}\\
 \nabla_{\vertt}{\cdot}\widehat\w_h|_{\widehat T} = r_{\widehat T} \,\,\, \forall \widehat T\in\widehat\TT_\vertt}
 }
 \norm{\widehat\w_h}{\widehat \omega_\vertt}.
\end{align}
Finally, an application of~\cite[Theorems~5 and~7]{Brae_Pill_Sch_p_rob_09} in two space dimensions or a similar procedure relying on~\cite{Cost_Daug_Demk_ext_08} in three space dimensions, see also~\cite[Theorem~2.5 and Corollary~3.3]{ev20} building on~\cite{Cost_McInt_Bog_Poinc_10, Demk_Gop_Sch_ext_III_12}, yields that
\begin{align*}
 \min_{\substack{\widehat\w_h\in\RT^{\widetilde\p}(\widehat\TT_\vertt)\\
 \widehat\w_h{\cdot}\n_{\widehat\omega_\vertt}|_{\widehat \ee} = r_{\widehat \ee}  \,\,\,\forall \widehat \ee\in\widehat\EE_\vertt^{\rm N}\\
 [\widehat\w_h{\cdot}\n_{\widehat\omega_\vertt}]_{\widehat \ee} = r_{\widehat \ee} \,\,\,\forall \widehat \ee\in\widehat\EE_\vertt^{\rm int}\\
 \nabla_{\vertt}{\cdot}\widehat\w_h|_{\widehat T} = r_{\widehat T} \,\,\, \forall \widehat T\in\widehat\TT_\vertt}
 } \norm{\widehat\w_h}{\widehat \omega_\vertt}
 \leq C_{\rm st}
  \min_{\substack{\widehat\w\in\H(\div,\widehat\TT_\vertt)\\
 \widehat\w{\cdot}\n_{\widehat\omega_\vertt}|_{\widehat \ee} = r_{\widehat \ee}  \,\,\,\forall \widehat \ee\in\widehat\EE_\vertt^{\rm N}\\
 [\widehat\w{\cdot}\n_{\widehat\omega_\vertt}]_{\widehat \ee} = r_{\widehat \ee} \,\,\,\forall \widehat \ee\in\widehat\EE_\vertt^{\rm int}\\
 \nabla_{\vertt}{\cdot}\widehat\w|_{\widehat T} = r_{\widehat T} \,\,\, \forall \widehat T\in\widehat\TT_\vertt}}
 \norm{\widehat\w}{\widehat\omega_\vertt},
\end{align*}
for a generic constant $C_{\rm st}$ only depending on the shapes of the elements in $\widehat \TT_{\uni{0}}$.
Here, $\H(\div,\widehat \TT_\vertt)$ is the space of $\widehat\TT_\vertt$-piecewise  $\H(\div)$-functions on $\widehat\omega_\vertt$, with the normal component understood as in~\cite[equation~(2.6)]{ev20}.
The required compatibility condition
\begin{align*}
 \sum_{\widehat T\in \widehat \TT_\vertt}  \dual{r_{\widehat T}}{1}_{\widehat T}
 - \sum_{\widehat \ee\in\widehat \EE_\vertt^{\rm N}\cup\widehat \EE_\vertt^{\rm int}} \dual{r_{\widehat \ee}}{1}_{\widehat \ee}
 = 0
\end{align*}
in case of $\psi_\ver\psi_\vertt\in H_0^1(\Omega)$ follows from the assumption that $g_h$, and thus $\widetilde \Phi^{-1}(g_h)$, has integral mean zero then. Finally,
\[
  \min_{\substack{\widehat\w\in\H(\div,\widehat\TT_\vertt)\\
 \widehat\w{\cdot}\n_{\widehat\omega_\vertt}|_{\widehat \ee} = r_{\widehat \ee}  \,\,\,\forall \widehat \ee\in\widehat\EE_\vertt^{\rm N}\\
 [\widehat\w{\cdot}\n_{\widehat\omega_\vertt}]_{\widehat \ee} = r_{\widehat \ee} \,\,\,\forall \widehat \ee\in\widehat\EE_\vertt^{\rm int}\\
 \nabla_{\vertt}{\cdot}\widehat\w|_{\widehat T} = r_{\widehat T} \,\,\, \forall \widehat T\in\widehat\TT_\vertt}}
 \norm{\widehat\w}{\widehat\omega_\vertt}
 \ls
 \min_{\substack{\bv\in\H_0(\div,\omega_\vertt)\\
 \Dv\bv=g_\coarse}}
 \norm{\bv + \ttau_\coarse}{\omega_\vertt}
\]
with a hidden constant depending only on $\norm{(D\F)^{-1}}{\infty,\widehat\Omega}$, as in~\eqref{eq:long substacks} and~\eqref{eq_Piola_equiv}. This concludes the proof.

%%%%%%%%%%%%%%%%%%%%%%%%%%%%%%%%%%%%%%%%%%%%%%%%%%%%%%%%%%%%%%%%%%%%%%%%%%%%%%%%%%%%%
%%%%%%%%%%%%%%%%%%%%%%%%%%%%%%%%%%%%%%%%%%%%%%%%%%%%%%%%%%%%%%%%%%%%%%%%%%%%%%%%%%%%%
\bibliographystyle{alpha}
\bibliography{literature}
%%%%%%%%%%%%%%%%%%%%%%%%%%%%%%%%%%%%%%%%%%%%%%%%%%%%%%%%%%%%%%%%%%%%%%%%%%%%%%%%%%%%%
%%%%%%%%%%%%%%%%%%%%%%%%%%%%%%%%%%%%%%%%%%%%%%%%%%%%%%%%%%%%%%%%%%%%%%%%%%%%%%%%%%%%%

\end{document}